\documentclass[10pt]{amsart}

\calclayout

\usepackage{amsmath, amssymb}
\usepackage{amsfonts}
\usepackage{mathrsfs}
\usepackage[color, arrow,matrix,curve,cmtip,ps]{xy}
\usepackage{paralist}
\usepackage{quiver}
\usepackage{amsthm}
\usepackage{caption}
\usepackage{tikz-cd}
\usetikzlibrary{arrows,calc,matrix}
\tikzset{
curvarr/.style={
  to path={ -- ([xshift=2ex]\tikztostart.east)
    |- (#1) [near end]\tikztonodes
    -| ([xshift=-2ex]\tikztotarget.west)
    -- (\tikztotarget)}
  }
}
\tikzset{%
    symbol/.style={%
        draw=none,
        every to/.append style={%
            edge node={node [sloped, allow upside down, auto=false]{$#1$}}}
    }
}

\usepackage[final]{pdfpages}
\usepackage{dsfont,txfonts}
\usepackage{tikz}
\usepackage{rotating}
\usepackage{mathrsfs}
\PassOptionsToPackage{hyphens}{url}\usepackage{hyperref}
\usepackage{enumitem}
\usepackage{graphicx}
\usepackage{mathtools}
\usetikzlibrary{arrows,chains,matrix,positioning,scopes}

\usepackage[T1]{fontenc}
\usepackage[variablett]{lmodern}
\usepackage{xcolor}
\usepackage[citestyle=alphabetic,bibstyle=alphabetic]{biblatex}
\usepackage{lipsum}
\usepackage[most]{tcolorbox}

\newtheorem{theorem}{Theorem}[section]
\theoremstyle{definition}
\newtheorem{lemma}[theorem]{Lemma}

\newtheorem{conjecture}[theorem]{Conjecture}
\newtheorem{observation}[theorem]{Observation}
\newtheorem{warning}[theorem]{Warning}

\newtheorem{proposition}[theorem]{Proposition}
\newtheorem{corollary}[theorem]{Corollary}
\newtheorem{definition}[theorem]{Definition}

\newtheorem{remark}[theorem]{Remark}

\newtheorem{example}[theorem]{Example}

\newtheorem{construction}[theorem]{Construction}

\everymath{\displaystyle}
\allowdisplaybreaks

\newcommand{\pt}{\text{pt}}

\newcommand{\Fun}{\text{Fun}}

\newcommand{\Sh}{\mathrm{Sh}}
\newcommand{\Cosh}{\mathrm{Cosh}}
\newcommand{\Sp}{\mathrm{Sp}}
\newcommand{\An}{\mathrm{An}}
\newcommand{\colim}{\mathrm{colim}}
\newcommand{\Map}{\mathrm{Map}}
\newcommand{\Cat}{\mathrm{Cat}}
\newcommand{\PR}{\mathrm{Pr}}
\newcommand{\CAlg}{\mathrm{CAlg}}
\newcommand{\Ind}{\mathrm{Ind}}
\newcommand{\perf}{\mathrm{perf}}
\newcommand{\patch}{\mathrm{patch}}
\newcommand{\cont}{\mathrm{cont}}
\newcommand{\disc}{\mathrm{disc}}
\newcommand{\Alex}{\mathrm{Alex}}
\newcommand{\op}{\mathrm{op}}
\newcommand{\dual}{\mathrm{dual}}
\newcommand{\st}{\mathrm{st}}
\newcommand{\ca}{\mathrm{ca}}

\DeclareFieldFormat{postnote}{#1}
\DeclareFieldFormat{multipostnote}{#1}

\begin{document}
\title{Algebraic $K$-theory of stably compact spaces}
\author{Georg Lehner}
\subjclass[2020]{Primary: 19D99, Secondary: 06D22, 54B40, 18F20, 18F70}
\keywords{Stably compact spaces, Verdier duality, patch topology, algebraic $K$-theory, $\infty$-categories of sheaves on a space.}
\begin{abstract}
We compute the value of finitary localizing invariants, including algebraic $K$-theory, on categories of sheaves over stably locally compact spaces $X$. Our formula simultaneously generalizes the cases of locally compact Hausdorff and coherent (spectral) spaces and recovers several smaller $K$-theory calculations as special instances.
\end{abstract}
\maketitle

\begin{figure}[h!]
    \centering
    \includegraphics[width=0.45\textwidth]{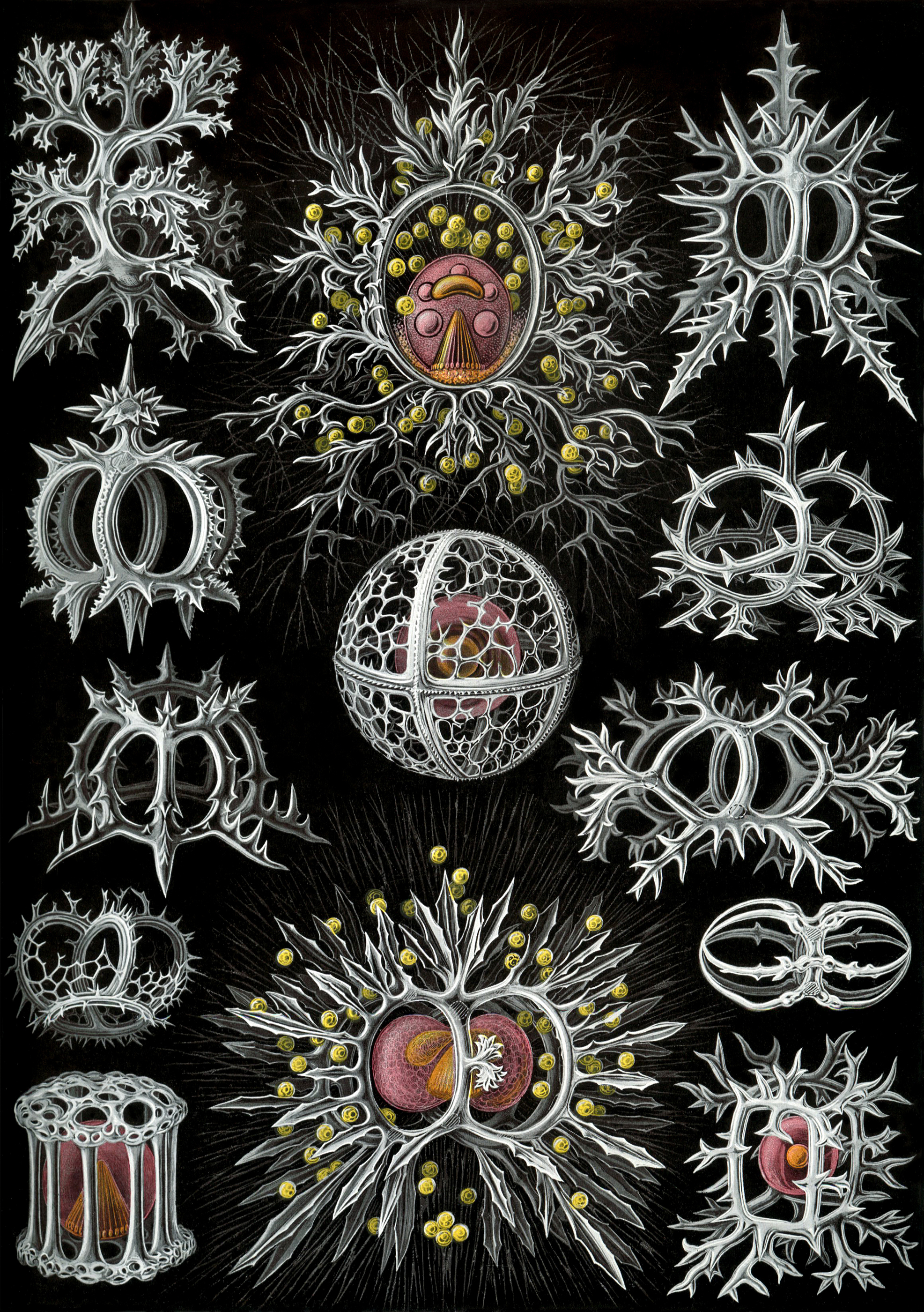}
    \caption*{\tiny Ernst Haeckel, \textit{Kunstformen der Natur}, 1904, plate 71: Stephoidea, Public domain, via Wikimedia Commons}
\end{figure}

\tableofcontents

\section{Introduction}

Recent advances in the $K$-theory of large categories as laid out by the articles \cite{efimov2025ktheorylocalizinginvariantslarge}, \cite{efimov2025localizinginvariantsinverselimits}, \cite{efimov2025rigiditycategorylocalizingmotives} and \cite{krause_nikolaus_puetzstueck}, see also the lecture \cite{nikolaus2023geometric} by Nikolaus, have raised the exciting possibility of providing a conceptual framework for assembly conjectures, such as the Farrell-Jones, Novikov, and Borel conjectures using the language of sheaf theory. It also seems possible to finally bridge the gap between the (algebraic) Farrell-Jones conjecture and its operator-theoretic counterpart, the Baum-Connes conjecture. The central insight is that localizing invariants, such as algebraic $K$-theory, can be extended from small stable $\infty$-categories to the class of \emph{dualizable} stable $\infty$-categories. This includes $\infty$-categories of spectral sheaves $\Sh(X, \Sp)$ on a locally compact Hausdorff space $X$, which play an analogous role to $C^*$-algebras of the form $C_0(X,\mathbb{C})$ in the context of operator theory. Whilst the general picture is still very much work in progress, an important computation that motivated this project is the following. 

\begin{theorem}[\cite{efimov2025ktheorylocalizinginvariantslarge} Theorem 6.11, see also \cite{krause_nikolaus_puetzstueck}  Theorem 3.6.1] \label{ktheorylch}
Let $X$ be a locally compact Hausdorff space. Then
$$K^{\cont}( \Sh(X, \Sp) ) \simeq H_{cs}(X; K(\mathbb{S}) )$$
where $\Sp$ is the $\infty$-category of spectra and the right-hand side refers to compactly supported sheaf cohomology of $X$ with respect to the local system given by the $K$-theory of the sphere spectrum. 
\end{theorem}

There are some extensions that are possible for Theorem \ref{ktheorylch}. For one, the theorem holds not just for $K$-theory, but for arbitrary finitary localizing invariants, as long as the target of the localizing invariant is a dualizable $\infty$-category. It also generalizes to non-constant coefficients, in the sense that an analogous formula holds when one considers a presheaf $\underline{\mathcal{C}}$ on $X$ valued in dualizable categories.

We propose a generalization of Theorem \ref{ktheorylch} by enlarging the class of spaces $X$ considerably. As was already observed by Efimov in \cite[Remark 6.2]{efimov2025ktheorylocalizinginvariantslarge}, Hausdorffness of a space $X$ is not essential to guarantee that $\Sh(X, \Sp)$ is dualizable, rather a weaker notion suffices. Recall the notion of the way-below relation between open subsets of a topological space $X$: Two open sets $U,V$ satisfy the relation $U \ll V$ iff for all directed sets of opens $W_i, i\in I$ it holds that
$$ \bigcup_{i\in I} W_i \supset V \text{ implies } \exists i \in I : W_i \supset U. $$ 
(In other words, the inclusion $U \subset V$ is a compact morphism in the poset $\mathcal{O}(X)$ of open sets of $X$.)

\begin{definition}
A topological space $X$ is called \emph{stably locally compact} if it is sober, and the way-below relation $\ll$ satisfies the two conditions:
\begin{itemize}
\item For all open $U$, we have $U = \bigcup_{V \ll U} V$.
\item The way-below relation is stable under intersection, in the sense that $U \ll V_1$ and $U \ll V_2$ implies $U \ll V_1 \cap V_2$.
\end{itemize}
A continuous map $f : X \rightarrow Y$ will be called \emph{perfect} if $f^{-1}$ preserves the way-below relation. A partial map $f : X \rightarrow Y$ will be called \emph{partial perfect}, if its domain of definition is open, and $f$ restricted to its domain is a perfect map. We define the category $\mathrm{SLC}_p$ of stably locally compact spaces and partial perfect maps. 
\end{definition}

Any locally compact Hausdorff space $X$ is in particular an example of a stably locally compact space. Moreover, the inclusion of the full subcategory of locally compact Hausdorff spaces into stably locally compact spaces has a right adjoint
\[\begin{tikzcd}
	{\mathrm{LCH}_p} & {\mathrm{SLC}_p}
	\arrow[""{name=0, anchor=center, inner sep=0}, hook, from=1-1, to=1-2]
	\arrow[""{name=1, anchor=center, inner sep=0}, "\patch", curve={height=-12pt}, from=1-2, to=1-1]
	\arrow["\dashv"{anchor=center, rotate=-90}, draw=none, from=0, to=1]
\end{tikzcd}\]
which equips a stably locally compact space $X$ with its \emph{patch topology}. (See Section \ref{patchtopology})

If $X$ is stably locally compact, the category of sheaves $\Sh(X,\Sp)$ is dualizable (Corollary \ref{sheavesonstablylocallycompact}), and hence it makes sense to apply localizing invariants, such as $K$-theory. We have obtained the following generalization of Theorem \ref{ktheorylch}.

\begin{theorem}[See Theorem \ref{maintheoremmain2}] \label{maintheorem}
Let $X$ be stably locally compact, let $\mathcal{C}$ be a dualizable stable $\infty$-category and let $F$ be a finitary localizing invariant $ \Cat^{\perf} \rightarrow  \mathcal{E}$ with values in a dualizable stable $\infty$-category $\mathcal{E}$. Then
$$F^{\cont}(\Sh(X,\mathcal{C})) \simeq H_{cs}(X^{\patch}; F^{\cont}(\mathcal{C})).$$
\end{theorem}
Here, the right-hand side refers to the \emph{compactly supported cohomology} of $X^{\patch}$. We also obtain the following perhaps surprising duality statement. Given $X$ and $\mathcal{C}$ as above, define the $\infty$-category of cosheaves on $X$ as $$ \Cosh(X,\mathcal{C}) = \mathrm{Fun}^L( \Sh(X, \Sp), \mathcal{C}). $$

\begin{corollary}[See Corollary \ref{sheavesequalcosheaves}]
Let $X$ be stably locally compact, let $\mathcal{C}$ be a dualizable stable $\infty$-category and let $F$ be a finitary localizing invariant $ \Cat^{\perf} \rightarrow  \mathcal{E}$ with values in a presentable stable $\infty$-category $\mathcal{E}$. Then
$$F^{\cont}(\Sh(X,\mathcal{C})) \simeq F^{\cont}(\Cosh(X,\mathcal{C})).$$
\end{corollary}

In the case of $X$ being locally compact Hausdorff the above statement easily follows from Verdier duality. However, in the absence of Hausdorffness, there is no a priori existing functor comparing sheaves and cosheaves on $X$ directly. (There is a version of Verdier duality for stably compact spaces, which involves the \emph{de Groot dual} $X^\vee$ of $X$, a statement which we will discuss in Section \ref{sectionverdierduality}.)

\begin{remark}
By Efimov \cite{efimov2025rigiditycategorylocalizingmotives}, the target of the \emph{universal} finitary localizing invariant $\mathcal{U} :  \Cat^{\perf} \rightarrow  \mathrm{Mot}$ is a dualizable stable $\infty$-category. Since every finitary localizing invariant $F$ factors through $\mathcal{U}$, this means that at least in principle one knows the value of $\Sh(X, \mathcal{C})$ for any finitary localizing invariant $F : \Cat^{\perf} \rightarrow  \mathcal{E}$, even when $\mathcal{E}$ is not dualizable.
\end{remark}

\subsection{Motivation}

We should say that the purpose of this article is two-fold. For one, a simple motivation for Theorem \ref{maintheorem} is that it is widely applicable to many different examples. Let us elaborate on some of them.

\begin{itemize}
    \item In the case of $X$ being Hausdorff, we have $X = X^{\patch}$ and recover the known fact that
    $$K^{\cont}(\Sh(X,\mathcal{C})) \simeq H_{cs}(X; K^{\cont}(\mathcal{C}))$$
    given by Theorem \ref{ktheorylch}.
    \item The special case of $X = \mathbf{S}$ being the Sierpinski space, i.e. the topological space given by a two element set $\mathbf{2}$ with a topology with a unique open point, can be enlightening. We have $\mathbf{S}^{\patch} = \mathbf{2}^{\disc}$ is the two point space equipped with the discrete topology. There is a natural equivalence $\Sh(\mathbf{S},\mathcal{C}) \simeq \mathcal{C}^{\Delta^1}$ and the continuous map $p : \mathbf{2}^{\disc} \rightarrow \mathbf{S}$ induces the functor
    $$\begin{array}{rcc}
    p_* : \mathcal{C} \times \mathcal{C} & \rightarrow & \mathcal{C}^{ \Delta^1 } \\
            ( c , d ) & \mapsto & ( c \times d \rightarrow c ).
	\end{array}$$      
    The statement that $p_*$ induces an equivalence on localizing invariants recovers the well-known \emph{additivity theorem}. We thus would like to think of Theorem \ref{maintheorem} as an analytic generalization of the additivity theorem.
	\item The case $X = \overrightarrow{[0, \infty)}$, with topology given by opens of the form $[0,a)$ for $0 \leq a \leq \infty$, is also quite important. There is an equivalence of $\infty$-categories
	$$\Sh(\overrightarrow{[0, \infty)}, \mathcal{C}) \simeq \Sh(\mathbb{R}, \mathcal{C})_{\geq 0},$$
	using the terminology of Efimov \cite[p.\ 76]{efimov2025ktheorylocalizinginvariantslarge}. In the case $\mathcal{C} = \Sp$, the dualizable $\infty$-category $\Sh(\mathbb{R}, \Sp)_{\geq 0}$ corepresents the functor that sends a dualizable $\infty$-category $\mathcal{C}$ to $\mathbb{Q}$-indexed diagrams in $\mathcal{C}$ with compact transition morphisms whenever $a < b \in \mathbb{Q}$. Efimov shows that $\Sh(\mathbb{R}, \Sp)_{\geq 0}$ is an $\omega_1$-compact generator of $\PR^L_{\dual}$, \cite[Theorem D.1]{efimov2025ktheorylocalizinginvariantslarge}, and computes that $F^{\cont}( \Sh(\mathbb{R}, \mathcal{C})_{\geq 0} ) = 0$ for any finitary localizing invariant $F$, \cite[Proposition 4.21]{efimov2025ktheorylocalizinginvariantslarge}. We recover this result as a special case, as $\overrightarrow{[0, \infty)}^{\patch} = [0, \infty)$ with the ordinary metric topology, which has a contractible one-point compactification.
    \item In the case of a \emph{coherent space} $X$, also sometimes called \emph{spectral space}, the patch topology on $X$ agrees with the constructible topology. The statement of Theorem \ref{maintheorem} in this case was verified in the author's previous article \cite{lehner2025algebraicktheorycoherentspaces}.
     \item If $P$ is a \emph{locally finite poset}, which means that the down-sets $P_{\leq p}$ are finite for each $p \in P$, we can consider $X = P_{\Alex}$ to be the set $P$ equipped with the Alexandrov topology, which means that open sets are defined to be downward closed sets. This gives a locally coherent space (\cite[Section 3.1]{lehner2025algebraicktheorycoherentspaces}), and one can show that there is a natural equivalence
$$\Sh(X,\mathcal{C}) \simeq \mathrm{Fun}(P^{\op},\mathcal{C}),$$
see \cite[Proposition 5.18]{Aoki2023posets}.\footnote{Warning: We are using the opposite convention from Aoki for the definition of a locally finite poset.} One can see that $P_{\Alex}^{\patch} = P^{\disc}$ is just $P$ equipped with the discrete topology, hence we obtain
$$ K^{\cont}(\mathrm{Fun}(P^{\op},\mathcal{C})) \simeq  K^{\cont}(\Sh(X,\mathcal{C})) \simeq \bigoplus_{p \in P} K^{\cont}(\mathcal{C}), $$
a result that can of course also be verified independently using the fact that $\mathrm{Fun}(P^{\op},\mathcal{C})$ admits a $P$-indexed semiorthogonal decomposition, \cite[Proposition 4.14]{efimov2025ktheorylocalizinginvariantslarge}. Examples of locally finite posets are face posets of finite dimensional simplicial or polyhedral complexes, or variations thereof. They arise naturally as gadgets for homotopy theory (the case of finite posets models the homotopy theory of finite CW-complexes), or when considering stratifications of spaces, and sheaves on these spaces play a crucial role in combinatorial areas of geometry, e.g. Combinatorial Intersection Cohomology (See \cite{BBFK}, \cite{Karu2004} and \cite{Braden2006}), or the theory of Tropical (Co)-Homology, \cite{brugallé2015briefintroductiontropicalgeometry}. The ability to have these examples in the same framework as the infinitary analytic examples is one of the strengths of considering the general framework of sheaves on stably locally compact spaces, and not just locally compact Hausdorff spaces.
    \item Let $V$ be a finite-dimensional real vector space and $\gamma$ a non-zero proper convex closed cone. 		Kashiwara-Schapira \cite{kashiwara2002sheaves} define the $\gamma$-topology $V_\gamma$ on $V$, which is 		given by the set of open subsets $U$ of $V$ which satisfy $U + \gamma = U$. They show that 
      $$ \Sh_{V \times \gamma}( V,\mathcal{C} ) \simeq \Sh( V_\gamma,\mathcal{C} ),$$
    where $\Sh_{V \times \gamma}( V,\mathcal{C} )$ is the $\infty$-category of sheaves on $V$ with microsupport $V \times \gamma$. (See Theorem 1.5 in \cite{Kashiwara2017PersistentHA}.) 
    Zhang proved in \cite{zhang2025remarkcontinuousktheoryfouriersato} that the value of any finitary localizing invariant on this $\infty$-category vanishes. The topological space $V_\gamma$ is not sober, however its sobrification adds boundary points towards the direction of the cone $\gamma$. This results in a stably locally compact space. (We remark that $\infty$-categories of sheaves cannot distinguish between a space and its sobrification.) It would be an interesting avenue for future work to reprove Zhang's result using purely topological means; by geometrically understanding the patch topology of $(V_\gamma)^{sob}$ and 	applying Theorem \ref{maintheorem}.
\end{itemize}

The secondary purpose of this article is to illustrate that the apparent usefulness of the class of stably locally compact spaces given by the above list of examples is by no means an accident. One can think of it as a partial answer to the following question.

\begin{tcolorbox}[colback=gray!10, colframe=gray!30, boxrule=0pt, sharp corners, enhanced]
What is the largest (convenient) class of spaces $X$ such that $\Sh(X,\Sp)$ is dualizable?
\end{tcolorbox}

Since it is the case that $\Sh(X,\Sp) \simeq \Sh(X,\An) \otimes \Sp$, where $\An$ denotes the $\infty$-topos of anima/$\infty$-groupoids, one might ask instead for the slightly stronger condition that the $\infty$-topos $\Sh(X,\An)$ is a compactly assembled $\infty$-category. Let us mention the following facts.

\begin{itemize}
\item Given any Grothendieck topology $\tau$ on a poset $P$, which is closed under binary meets, there exists a natural frame $F = F(P,\tau)$ such that $\Sh(P,\tau,\An) \cong \Sh(F,\An)$. Such Grothendieck topologies appear naturally in many contexts.
\item If for a given frame $F$ the $\infty$-topos $\Sh(F,\An)$ is a compactly assembled $\infty$-category, then $F$ is isomorphic to the set of opens of a sober and locally compact space $X$. (See Theorem \ref{localicreflection} and Theorem \ref{spatialityoflocallycompact}.)
\item However, it is not true that for every sober and locally compact space $X$ the $\infty$-topos $\Sh(X,\An)$ is also compactly assembled. An example that already fails in the case of $1$-topoi is given in \cite[Example 5.5]{JOHNSTONE1982255}.
\item Under the additional assumption that $X$ is in fact stably locally compact this issue disappears, and $\Sh(X,\An)$ is compactly assembled. Stability of $X$ corresponds precisely to the additional property of $\Sh(X,\An)$ that compact morphisms are closed under products.
\end{itemize}

While it can still happen that for a topological space $X$ (or frame, for that matter) the category of sheaves $\Sh(X, \mathcal{C})$ is dualizable for a given dualizable $\mathcal{C}$, without $X$ itself being stably locally compact, the category of stably locally compact spaces is thus in a sense an optimal candidate for a category of spaces for which we get dualizable categories of sheaves.

\begin{remark} 
The condition of an $\infty$-topos $\mathcal{X}$ to be compactly assembled as an $\infty$-category is equivalent to $\mathcal{X}$ being \emph{exponentiable} in the $\infty$-category $\mathrm{RTop}$ of topoi, a fact proven independently by Anel-Lejay \cite{anel2018exponentiablehighertoposes}, \cite{anel_lejay_exponentiable_V2}, as well as Lurie \cite[Section 21.1.6]{lurieSAG}, and generalizes the analogous statement for $1$-topoi due to Johnstone-Joyal \cite{JOHNSTONE1982255}.
\end{remark}

On a deeper level, there are many close parallels between the theory of stably locally compact spaces and that of dualizable categories. Efimov originally observed shadows of this picture as several analogies between dualizable categories and compact Hausdorff spaces, see \cite[Appendix F]{efimov2025ktheorylocalizinginvariantslarge}. Among these are for example versions of Urysohn's lemma (in the form as stating that $[0,1]$ and $\Sh(\mathbb{R}, \Sp)_{\geq 0}$ are generators of the categories $\mathrm{CH}^{\op}$ and $\PR^L_{\dual}$ respectively), and the Tychonoff Theorem (stating that the inclusion functors $\mathrm{CH}^{\op} \rightarrow \mathrm{Top}^{\op}$ and $\PR^L_{\dual} \rightarrow \PR^L$ preserve colimits).

However, this analogy remains somewhat vague. When one switches the perspective to stably locally compact spaces, one can observe the following two similarities.

\begin{itemize}
\item A space $X$ is stably locally compact iff $X$ is sober and there exists a diagram of adjoints
\[\begin{tikzcd}
	{\mathcal{O}(X)} & {\Ind(\mathcal{O}(X))}
	\arrow[""{name=0, anchor=center, inner sep=0}, "y"', curve={height=12pt}, from=1-1, to=1-2]
	\arrow[""{name=1, anchor=center, inner sep=0}, "{\hat{y}}", curve={height=-12pt}, from=1-1, to=1-2]
	\arrow[""{name=2, anchor=center, inner sep=0}, "k"{description}, from=1-2, to=1-1]
	\arrow["\dashv"{anchor=center, rotate=-90}, draw=none, from=1, to=2]
	\arrow["\dashv"{anchor=center, rotate=-90}, draw=none, from=2, to=0]
\end{tikzcd}\]
internal to the $2$-category of lower bounded distributive lattices.
\item An accessible stable $\infty$-category $\mathcal{C}$ is dualizable iff there exists a diagram of adjoints
\[\begin{tikzcd}
	{\mathcal{C}} & {\Ind(\mathcal{C})}
	\arrow[""{name=0, anchor=center, inner sep=0}, "y"', curve={height=12pt}, from=1-1, to=1-2]
	\arrow[""{name=1, anchor=center, inner sep=0}, "{\hat{y}}", curve={height=-12pt}, from=1-1, to=1-2]
	\arrow[""{name=2, anchor=center, inner sep=0}, "k"{description}, from=1-2, to=1-1]
	\arrow["\dashv"{anchor=center, rotate=-90}, draw=none, from=1, to=2]
	\arrow["\dashv"{anchor=center, rotate=-90}, draw=none, from=2, to=0]
\end{tikzcd}\]
internal to the $(\infty,2)$-category of (large) stable $\infty$-categories.
\end{itemize}

Similarly, partial perfect maps between stably locally compact spaces, as well as strongly continuous functors can also be characterized in purely $2$-categorical terms. The formal setup for this is that both cases are instances of \emph{continuous algebras} for a lax-idempotent monad on an $(\infty,2)$-category.\footnote{The general theory of lax-idempotent monads on $(\infty,2)$-categories is still work in progress, with forthcoming results by Abellán--Blom \cite{abellanblom:laxidempotent}.}
The corresponding $2$-categorical version is sometimes also referred to as a KZ-monad. The requirement of the sobriety condition on $X$ is an indicator that truly, one should rather consider the ambient setting of frame/locales (of which $\mathcal{O}(X)$ is an example), rather than topological spaces, as the right formal analogue to the $\infty$-category $\PR^L_{\st}$. Frames are in fact the algebras for the $\Ind$-monad on distributive lattices, the same way presentable stable $\infty$-categories are (up to size issues) the algebras for the $\Ind$-monad on stable $\infty$-categories. The functor which assigns a frame $F$ to its $\infty$-category of sheaves $\Sh(F,\Sp)$ provides a bridge between the two cases, and descends to a well-defined functor from $(\mathrm{SLC}_p)^{\op}$ to $\PR^L_{\dual}$, which does in fact send the generator $\overrightarrow{[0, \infty)}$ (see Corollary \ref{urysohnstablylocallycompact}) to the generator $\Sh(\mathbb{R}, \Sp)_{\geq 0}$.

Aoki has observed that the relationship between frames and presentable categories becomes quite tight, which he formalized in his sheaves-spectrum adjunction.

\begin{theorem}[\cite{aoki2025sheavesspectrumadjunction}] \label{sheafspectrumadjunction}
There exists an adjunction
\[\begin{tikzcd}
	{\mathrm{Frm}} & {\CAlg(\mathrm{Pr^L})}
	\arrow[""{name=0, anchor=center, inner sep=0}, "{\Sh(-,\An)}", curve={height=-12pt}, from=1-1, to=1-2]
	\arrow[""{name=1, anchor=center, inner sep=0}, "{\mathrm{cIdem}}", curve={height=-12pt}, from=1-2, to=1-1]
	\arrow["\dashv"{anchor=center, rotate=-90}, draw=none, from=0, to=1]
\end{tikzcd}\]
between the category of frames $\mathrm{Frm}$ and the $\infty$-category of presentably symmetric monoidal $\infty$-categories, where $\Sh(F,\An)$ is the $\infty$-category of sheaves on a frame $F$, and $\mathrm{cIdem}(\mathcal{C})$ is the spectrum of coidempotents of a presentably symmetric monoidal $\infty$-category $\mathcal{C}$. 
\end{theorem}

He fleshed this adjunction out further in the article \cite{aoki2025schwartzcoidempotentscontinuousspectrum}, in a manner which we summarize by the following commutative diagram of adjunctions. Let us call a space $X$ \emph{stably compact} if it is compact and stably locally compact. Denote by $\mathrm{SC}$ the category of stably compact spaces and perfect maps, and by $\mathrm{CH}$ the full subcategory of compact Hausdorff spaces. (We note that any continuous map between compact Hausdorff spaces is automatically perfect.)

\[\begin{tikzcd}
	{\mathrm{Rig}} & {\CAlg(\PR^L_{\dual})} & {\CAlg(\PR^L_{\st})} & {\CAlg(\Cat^{\perf})} \\
	{\mathrm{CH}^{\op}} & {\mathrm{SC}^{\op}} & {\mathrm{Frm}} & {\mathrm{DLatt}_{\mathrm{bd}}}
	\arrow[""{name=0, anchor=center, inner sep=0}, color={rgb,255:red,92;green,92;blue,214}, hook, from=1-1, to=1-2]
	\arrow[""{name=0p, anchor=center, inner sep=0}, phantom, from=1-1, to=1-2, start anchor=center, end anchor=center]
	\arrow[""{name=0p, anchor=center, inner sep=0}, phantom, from=1-1, to=1-2, start anchor=center, end anchor=center]
	\arrow[""{name=1, anchor=center, inner sep=0}, "{\mathrm{Sm}^{rig}}", color={rgb,255:red,214;green,92;blue,92}, curve={height=-12pt}, from=1-1, to=2-1]
	\arrow[""{name=1p, anchor=center, inner sep=0}, phantom, from=1-1, to=2-1, start anchor=center, end anchor=center, curve={height=-12pt}]
	\arrow[""{name=2, anchor=center, inner sep=0}, color={rgb,255:red,214;green,92;blue,92}, curve={height=-12pt}, from=1-2, to=1-1]
	\arrow[""{name=2p, anchor=center, inner sep=0}, phantom, from=1-2, to=1-1, start anchor=center, end anchor=center, curve={height=-12pt}]
	\arrow[""{name=3, anchor=center, inner sep=0}, "{``\mathrm{add~duals}"}"', color={rgb,255:red,92;green,214;blue,92}, curve={height=12pt}, from=1-2, to=1-1]
	\arrow[""{name=3p, anchor=center, inner sep=0}, phantom, from=1-2, to=1-1, start anchor=center, end anchor=center, curve={height=12pt}]
	\arrow[""{name=4, anchor=center, inner sep=0}, "forget", color={rgb,255:red,92;green,92;blue,214}, from=1-2, to=1-3]
	\arrow[""{name=4p, anchor=center, inner sep=0}, phantom, from=1-2, to=1-3, start anchor=center, end anchor=center]
	\arrow[""{name=5, anchor=center, inner sep=0}, "{\mathrm{Sm}^{con}}", color={rgb,255:red,214;green,92;blue,92}, curve={height=-12pt}, from=1-2, to=2-2]
	\arrow[""{name=5p, anchor=center, inner sep=0}, phantom, from=1-2, to=2-2, start anchor=center, end anchor=center, curve={height=-12pt}]
	\arrow[""{name=6, anchor=center, inner sep=0}, "{\Ind}"{description}, color={rgb,255:red,214;green,92;blue,92}, curve={height=-12pt}, from=1-3, to=1-2]
	\arrow[""{name=6p, anchor=center, inner sep=0}, phantom, from=1-3, to=1-2, start anchor=center, end anchor=center, curve={height=-12pt}]
	\arrow[""{name=7, anchor=center, inner sep=0}, "{forget}"{description} , color={rgb,255:red,214;green,92;blue,92}, curve={height=12pt}, from=1-3, to=1-4]
	\arrow[""{name=7p, anchor=center, inner sep=0}, phantom, from=1-3, to=1-4, start anchor=center, end anchor=center, curve={height=12pt}]
	\arrow[""{name=8, anchor=center, inner sep=0}, "{\mathrm{Sm}}", color={rgb,255:red,214;green,92;blue,92}, curve={height=-12pt}, from=1-3, to=2-3]
	\arrow[""{name=8p, anchor=center, inner sep=0}, phantom, from=1-3, to=2-3, start anchor=center, end anchor=center, curve={height=-12pt}]
	\arrow[""{name=9, anchor=center, inner sep=0}, "{\Ind}"', color={rgb,255:red,92;green,92;blue,214}, from=1-4, to=1-3]
	\arrow[""{name=9p, anchor=center, inner sep=0}, phantom, from=1-4, to=1-3, start anchor=center, end anchor=center]
	\arrow[""{name=10, anchor=center, inner sep=0}, "{\mathrm{Idemp}}", color={rgb,255:red,214;green,92;blue,92}, curve={height=-12pt}, from=1-4, to=2-4]
	\arrow[""{name=11, anchor=center, inner sep=0}, "{\Sh}", color={rgb,255:red,92;green,92;blue,214}, hook, from=2-1, to=1-1]
	\arrow[""{name=11p, anchor=center, inner sep=0}, phantom, from=2-1, to=1-1, start anchor=center, end anchor=center]
	\arrow[""{name=12, anchor=center, inner sep=0}, color={rgb,255:red,92;green,92;blue,214}, hook, from=2-1, to=2-2]
	\arrow[""{name=12p, anchor=center, inner sep=0}, phantom, from=2-1, to=2-2, start anchor=center, end anchor=center]
	\arrow[""{name=12p, anchor=center, inner sep=0}, phantom, from=2-1, to=2-2, start anchor=center, end anchor=center]
	\arrow[""{name=13, anchor=center, inner sep=0}, "{\Sh}", color={rgb,255:red,92;green,92;blue,214}, hook, from=2-2, to=1-2]
	\arrow[""{name=13p, anchor=center, inner sep=0}, phantom, from=2-2, to=1-2, start anchor=center, end anchor=center]
	\arrow[""{name=14, anchor=center, inner sep=0}, "{}", color={rgb,255:red,214;green,92;blue,92}, curve={height=-12pt}, from=2-2, to=2-1]
	\arrow[""{name=14p, anchor=center, inner sep=0}, phantom, from=2-2, to=2-1, start anchor=center, end anchor=center, curve={height=-12pt}]
	\arrow[""{name=15, anchor=center, inner sep=0}, "\patch"{description}, color={rgb,255:red,92;green,214;blue,92}, curve={height=12pt}, from=2-2, to=2-1]
	\arrow[""{name=15p, anchor=center, inner sep=0}, phantom, from=2-2, to=2-1, start anchor=center, end anchor=center, curve={height=12pt}]
	\arrow[""{name=16, anchor=center, inner sep=0}, "forget", color={rgb,255:red,92;green,92;blue,214}, from=2-2, to=2-3]
	\arrow[""{name=16p, anchor=center, inner sep=0}, phantom, from=2-2, to=2-3, start anchor=center, end anchor=center]
	\arrow[""{name=17, anchor=center, inner sep=0}, "{\Sh}", color={rgb,255:red,92;green,92;blue,214}, from=2-3, to=1-3]
	\arrow[""{name=17p, anchor=center, inner sep=0}, phantom, from=2-3, to=1-3, start anchor=center, end anchor=center]
	\arrow[""{name=18, anchor=center, inner sep=0}, "{\Ind}"{description}, color={rgb,255:red,214;green,92;blue,92}, curve={height=-12pt}, from=2-3, to=2-2]
	\arrow[""{name=18p, anchor=center, inner sep=0}, phantom, from=2-3, to=2-2, start anchor=center, end anchor=center, curve={height=-12pt}]
	\arrow[""{name=19, anchor=center, inner sep=0}, "{forget}"', color={rgb,255:red,214;green,92;blue,92}, curve={height=12pt}, from=2-3, to=2-4]
	\arrow[""{name=20, anchor=center, inner sep=0}, "{\Sh(-,fin)^{\omega}}", color={rgb,255:red,92;green,92;blue,214}, hook, from=2-4, to=1-4]
	\arrow[""{name=21, anchor=center, inner sep=0}, "{\Ind}"{description}, color={rgb,255:red,92;green,92;blue,214}, from=2-4, to=2-3]
	\arrow["\dashv"{anchor=center, rotate=-90}, draw=none, from=0p, to=2p]
	\arrow["\dashv"{anchor=center, rotate=-90}, draw=none, from=4p, to=6p]
	\arrow["\dashv"{anchor=center, rotate=-90}, draw=none, from=3p, to=0p]
	\arrow["\dashv"{anchor=center, rotate=-90}, draw=none, from=9p, to=7p]
	\arrow["\dashv"{anchor=center, rotate=-90}, draw=none, from=12p, to=14p]
	\arrow["\dashv"{anchor=center, rotate=0}, draw=none, from=11p, to=1p]
	\arrow["\dashv"{anchor=center, rotate=-90}, draw=none, from=15p, to=12p]
	\arrow["\dashv"{anchor=center, rotate=-90}, draw=none, from=16p, to=18p]
	\arrow["\dashv"{anchor=center}, draw=none, from=13p, to=5p]
	\arrow["\dashv"{anchor=center, rotate=-0}, draw=none, from=17p, to=8p]
	\arrow["\dashv"{anchor=center}, draw=none, from=20, to=10]
	\arrow["\dashv"{anchor=center, rotate=-90}, draw=none, from=21, to=19]
\end{tikzcd}\]

A few comments. 
\begin{itemize}
	\item The functors $\Sh$ and $\Sh(-,fin)$ in this diagram refer to sheaves with values in spectra.
    \item The category $\CAlg(\PR^L_{\dual})$ refers to presentably symmetric monoidal stable $\infty$-categories, whose underlying category is dualizable, with compact $1$ and such that $\otimes$ preserves compact morphisms in both variables.
    \item The category $\mathrm{Rig}$ refers to \emph{rigid} presentably symmetric monoidal categories. For a reference, see \cite{krause_nikolaus_puetzstueck} and \cite{ramzi2024locallyrigidinftycategories}.
    \item All blue arrows in the diagram are left adjoints, and all squares composed of them commute. The red arrows are all right adjoints to the blue ones, and the green arrows are all left adjoints to the blue ones.
    \item  The central columns are given as algebras and continuous algebras, respectively for the $\Ind$-monad.
    \item The images of the right most columns under $\Ind$ give the case of coherent spaces and compactly generated stable $\infty$-categories. These can be understood as \emph{free} continuous algebras, which generate the rest under retracts.
    \item The relationship of $\mathrm{CH}$ and $\mathrm{Rig}$ to the rest of the picture will not be further discussed in this paper, with the exception of the patch-topology functor, but is highly interesting.
\end{itemize}

\begin{remark}
On a more philosophical perspective on why one could expect Theorem \ref{maintheorem} to be true, we remark that $K$-theory can be considered as a form of \emph{universal measure} - In fact, there are several close analogies between measure theory and the $K$-theory of categories of sheaves over certain classes of locales, as investigated to some degree in the author's previous articles \cite[7.2]{lehner2025algebraicktheorycoherentspaces} and \cite[Remark 1.20]{lehner2025measuretheorylocales}, motivated by the observation that both areas of math are morally speaking just formalizations of the inclusion-exclusion principle together with the principle of exhaustion from below. We can define a \emph{measure} on a locale $L$ to be a continuous valuation $\mu : \mathcal{O}(L) \rightarrow [0, \infty]$. If $L$ corresponds to a locally compact Hausdorff space $X$, then \emph{locally finite} measures on $L$ agree with the classical notion of \emph{Radon measure} on $X$, see \cite[Theorem 9.4 and Section 12.2]{lehner2025measuretheorylocales}. It is in fact true that locally finite measures on a stably locally compact space $X$ extend uniquely to locally finite measures on $X^{\patch}$, as proven in \cite[Theorem 8.3]{keimel_lawson}. Conversely, any locally finite measure on $X^{\patch}$ produces a locally finite measure on $X$ via pushforward along the perfect map $X^{\patch} \rightarrow X$, resulting in an isomorphism
$$ \mathrm{Meas}_{\mathrm{lf}}(X) \cong \mathrm{Meas}_{\mathrm{lf}}(X^{\patch}),$$
natural in partially defined perfect maps of stably locally compact spaces with open support. Analogously, one could expect $K$-theory to behave in the same way. This was one of the motivations for the author to investigate the formula given in Theorem \ref{maintheorem}.
\end{remark}

\subsection{The proof strategy}

The proof of Theorem \ref{maintheorem} rests on a meta theorem which is a generalization of an argument due to Clausen in the case of compact Hausdorff spaces, see \cite[Theorem 3.6.11]{krause_nikolaus_puetzstueck}.

\begin{theorem}[See Theorem  \ref{properandcofilteredexcision}] Let $\mathcal{D}$ be a compactly assembled presentable $\infty$-category, and let $$F : \mathrm{SC}^{\op} \rightarrow \mathcal{D}$$ be a contravariant functor on the category of stably compact spaces such that:
\begin{enumerate}
    \item $F(\emptyset) = 1.$
    \item \emph{Perfect descent:} Whenever $K, L \subset X$ are two perfect embeddings, then \[\begin{tikzcd}
	{F(K \cup L)} & {F(K)} \\
	{F(L)} & {F(K \cap L)}
	\arrow[from=1-1, to=1-2]
	\arrow[from=1-1, to=2-1]
	\arrow["\lrcorner"{anchor=center, pos=0.125}, draw=none, from=1-1, to=2-2]
	\arrow[from=1-2, to=2-2]
	\arrow[from=2-1, to=2-2]
\end{tikzcd}\]
is a pullback.
\item \emph{Cofiltered descent:} Whenever $X_i, i \in I,$ is a cofiltered system in $\mathrm{SC},$ then
$$F( \lim_{i \in I} X_i ) \simeq \colim_{i \in I} F(X_i).$$
\end{enumerate}
Then there is a natural equivalence of functors
$$F \simeq \Gamma( (-)^{\patch}; F(\mathrm{pt}) ).$$
\end{theorem}

The proof works formally almost the same as the one given by Clausen, along with one crucial observation: The set of perfect embeddings $K \subset X$ is simply isomorphic to the set of closed subsets $C \subset X^{\patch}$. Hence for each given stably compact space $X$, the restriction of $F$ to the set of perfect embeddings of $X$ produces a $K$-sheaf on $X^{\patch}$. The corresponding sheaf $F^X$ under the equivalence of $K$-sheaves with sheaves (Theorem \ref{sheavesandksheaves}) satisfies that its compactly supported sections are computed as $F^X(X) = F(X)$. Now one needs to show that this sheaf is locally constant. We verify this in detail in Section \ref{descent}.

It is not too hard to verify that $K$-theory, or any finitary localizing invariant for that matter, satisfies (1) and (3). (See Corollary \ref{cofiltereddescent}.) However, showing perfect descent for finitary localizing invariants needs careful analysis. We reduce it to two different descent conditions: \emph{Closed} and \emph{Saturated compact} descent, and then use an inductive argument to extend it to all perfect embeddings in Section \ref{descentlocalizinginvariants}. This rests on a concrete description of the closed sets of the patch topology as generated under intersections by pairwise unions of closed and saturated compact sets, see Section \ref{patchtopology}.

\begin{remark}
As the reader will observe, the proof of Theorem \ref{maintheorem} needs little in the sense of actual $K$-theory computations. Most of the difficulty relies in understanding the correct ``point-set topological'' features of stably locally compact spaces and their $\infty$-categories of sheaves.
\end{remark}

\begin{remark}
We originally hoped that the proof of Theorem \ref{maintheorem} could be reduced to the following claim:
\begin{conjecture} \label{ktheoryrightkan}
The universal finitary localizing invariant $\mathcal{U} :  \CAlg(\PR^L_{\dual}) \rightarrow \CAlg(\mathrm{Mot})$ is right Kan extended from its restriction to $ \CAlg(\PR^L_{\st,\omega}).$
\end{conjecture}
The analogous statement without reference to symmetric monoidality is true by Efimov \cite[Theorem 4.16]{efimov2025ktheorylocalizinginvariantslarge}. If Conjecture \ref{ktheoryrightkan} were in fact true, one could use the adjointability of the square of adjoints 
\[\begin{tikzcd}
	{(\mathrm{CohSpc})^{\op}} & {\CAlg(\PR^L_{\st,\omega})} \\
	{(\mathrm{SC})^{\op}} & {\CAlg(\PR^L_{\dual})}
	\arrow[""{name=0, anchor=center, inner sep=0}, "{\Sh}", from=1-1, to=1-2]
	\arrow[""{name=1, anchor=center, inner sep=0}, hook', from=1-1, to=2-1]
	\arrow[""{name=2, anchor=center, inner sep=0}, "{\mathrm{Sm}^{con}}", color={rgb,255:red,92;green,92;blue,214}, curve={height=-12pt}, from=1-2, to=1-1]
	\arrow[""{name=3, anchor=center, inner sep=0}, hook', from=1-2, to=2-2]
	\arrow[""{name=4, anchor=center, inner sep=0}, color={rgb,255:red,92;green,92;blue,214}, curve={height=12pt}, from=2-1, to=1-1]
	\arrow[""{name=5, anchor=center, inner sep=0}, "{\Sh}", from=2-1, to=2-2]
	\arrow[""{name=6, anchor=center, inner sep=0}, color={rgb,255:red,92;green,92;blue,214}, curve={height=12pt}, from=2-2, to=1-2]
	\arrow[""{name=7, anchor=center, inner sep=0}, "{\mathrm{Sm}^{con}}", color={rgb,255:red,92;green,92;blue,214}, curve={height=-12pt}, from=2-2, to=2-1]
	\arrow["\dashv"{anchor=center, rotate=-90}, draw=none, from=0, to=2]
	\arrow["\dashv"{anchor=center}, draw=none, from=1, to=4]
	\arrow["\dashv"{anchor=center}, draw=none, from=3, to=6]
	\arrow["\dashv"{anchor=center, rotate=-90}, draw=none, from=5, to=7]
\end{tikzcd}\]
to show that $\mathcal{U}( \Sh(-))$ is right Kan extended along its restriction to $(\mathrm{CohSpc})^{\op}$, where $\mathrm{CohSpc}$ is the category of \emph{coherent spaces}, also called spectral spaces. From here, we have the existence of the adjointable square
\[\begin{tikzcd}
	{\mathrm{ProFin}} & {\mathrm{CohSpc}} \\
	{\mathrm{CH}} & {\mathrm{SC}}
	\arrow[""{name=0, anchor=center, inner sep=0}, hook, from=1-1, to=1-2]
	\arrow[hook', from=1-1, to=2-1]
	\arrow[""{name=1, anchor=center, inner sep=0}, "{\mathrm{\mathrm{const}}}", curve={height=-12pt}, from=1-2, to=1-1]
	\arrow[hook', from=1-2, to=2-2]
	\arrow[""{name=2, anchor=center, inner sep=0}, hook, from=2-1, to=2-2]
	\arrow[""{name=3, anchor=center, inner sep=0}, "{\patch}", curve={height=-12pt}, from=2-2, to=2-1]
	\arrow["\dashv"{anchor=center, rotate=-90}, draw=none, from=0, to=1]
	\arrow["\dashv"{anchor=center, rotate=-90}, draw=none, from=2, to=3]
\end{tikzcd}\]
which reduces the statement to the result that 
$$\mathcal{U}( \Sh( X ; \Sp )) \simeq H^\bullet( X^{\mathrm{const}}, \mathcal{U}( \Sp ) ), $$
for a coherent space $X$, as proven by the author in \cite{lehner2025algebraicktheorycoherentspaces}, together with a statement about whether the sheaf cohomology functor is the right Kan extension of its restriction along $(\mathrm{ProFin})^{\op} \rightarrow (\mathrm{CH})^{\op}$, a question that is dealt with using the newly developed setup of condensed cohomology, see  \cite{ScholzeCondensed}. As it stands, we are not sure whether Conjecture \ref{ktheoryrightkan} is correct, which is why we are left with the more granular proof strategy as presented in this article.
\end{remark}

\begin{remark}
In the case of $X$ being locally compact Hausdorff, Efimov \cite[Theorem 6.11]{efimov2025ktheorylocalizinginvariantslarge} proves a more general statement, namely that for any presheaf $\underline{\mathcal{C}} : \mathcal{O}(X)^{\op} \rightarrow \PR^L_{\dual}$ on $X$ with values in dualizable $\infty$-categories it holds that
$$\mathcal{U}( \Sh(X, \underline{\mathcal{C}} ) ) \simeq H^\bullet_c( X, \mathcal{U}(\underline{\mathcal{C}})^\# ).$$
It is plausible that Theorem \ref{maintheorem} holds in this greater generality as well for stably locally compact spaces, however we have not attempted this during the work on this article. Perhaps the main issue would be in understanding why $\mathcal{U}(\underline{\mathcal{C}})$ should lift to a presheaf on $X^{\patch}$, where $\mathcal{U}$ denotes the universal localizing invariant. This point is something to be addressed in future work.
\end{remark}

\subsection{Acknowledgements}
We want to thank Thomas Nikolaus, Maxime Ramzi, Thorger Geiß, Phil Pützstück, Dustin Clausen, Bingyu Zhang, Mathieu Anel and Bastiaan Cnossen for helpful comments and discussions. Furthermore, the author was funded by the Deutsche Forschungsgemeinschaft (DFG, German Research Foundation) – Project-ID 427320536–SFB 1442, as well as under Germany's Excellence Strategy EXC 2044/2 - 390685587, Mathematics Münster: Dynamics-Geometry-Structure.

\section{Preliminaries}

We will largely follow the conventions on $\infty$-category theory as laid out in \cite{luriehtt}, \cite{lurieha} and \cite{lurieSAG}. The symbol $[1]$ refers to the poset $\{0 \leq 1\}$. The symbol $\An$ denotes the $\infty$-category of \emph{anima} (equivalently \emph{spaces}, or $\infty$-groupoids). We denote the $\infty$-category of small $\infty$-categories by $\Cat_\infty$, and the $\infty$-category of large $\infty$-categories as $\widehat{\Cat_\infty}$. If $\mathcal{C}$ is an $\infty$-category, then $\mathcal{C}^{\op}$ refers to the $\infty$-category obtained by reversing the $1$-morphisms. The symbol $\PR^L$ refers to the $\infty$-category of presentable $\infty$-categories together with left adjoint functors as morphisms. This $\infty$-category is symmetric monoidal when equipped with the Lurie tensor product $ \mathcal{C} \otimes \mathcal{D} = \mathrm{Fun}^R( \mathcal{C}^{\op}, \mathcal{D})$, with unit $\An$ and which preserves colimits in both variables. The corresponding internal hom is given by $\mathrm{Fun}^L( \mathcal{C}, \mathcal{D})$. Commutative algebra objects in $(\PR^L,\otimes)$ are given by presentable symmetric monoidal $\infty$-categories $(\mathcal{C}, \otimes)$ such that the given tensor product preserves colimits in both variables. A functor $F : \mathcal{C} \rightarrow \mathcal{D}$ in $\PR^L$ with fully faithful right adjoint is called a \emph{Bousfield localization}. We remark that $F$ is a Bousfield localization iff for every other presentable $\infty$-category $\mathcal{E}$ precomposition with $F$ induces a fully faithful functor
$$ \Fun^L( \mathcal{D}, \mathcal{E} ) \rightarrow \Fun^L( \mathcal{C}, \mathcal{E} ),$$
see \cite[Proposition 5.5.4.20]{luriehtt}. Using the tensor-hom adjunction, it is clear that Bousfield localizations are closed under tensor products in $\PR^L$.

We will use light $2$-categorical techniques occasionally. If $\mathcal{C}$ is an $(\infty,2)$-category, it makes sense to speak of an \emph{internal adjunction} between objects $c,d$,
\[\begin{tikzcd}
	c & d
	\arrow[""{name=0, anchor=center, inner sep=0}, "g"', curve={height=6pt}, from=1-1, to=1-2]
	\arrow[""{name=1, anchor=center, inner sep=0}, "f"', curve={height=6pt}, from=1-2, to=1-1]
	\arrow["\dashv"{anchor=center, rotate=-90}, draw=none, from=1, to=0]
\end{tikzcd}\]
defined analogously to how one defines an adjunction between $\infty$-categories, see e.g.\ \cite[Ch.~12]{gaitsgory2019study}. An internal left adjoint $f$ is called an \emph{internal embedding} if the unit $ id_d \Rightarrow g f $ is an equivalence. It is called an \emph{internal quotient} if the counit $ f g \Rightarrow id_c $ is an equivalence.  Note that these notions only depend on the underlying homotopy $2$-category of $\mathcal{C}$. The following lemma will be useful on occasion.

\begin{lemma}
If $F : \mathcal{C} \rightarrow \mathcal{D}$ is a $2$-functor between $(\infty,2)$-categories it preserves internal adjunctions, internal embeddings and internal quotients.
\end{lemma}

The proof of this lemma reduces to the observation that internal adjunctions, internal embeddings and internal quotients are characterized as being equivalent to $2$-functors out of basic $2$-categorical diagram shapes, and composition of $2$-functors gives $2$-functors. We will occasionally use the following standard lemma.

\begin{lemma} \label{existenceofadjoint} Suppose $\mathcal{C}$ is an $(\infty,2)$-category and the diagram
\[\begin{tikzcd}
	{c_1} & {d_1} \\
	{c_0} & {d_0}
	\arrow[""{name=0, anchor=center, inner sep=0}, "{f_1}"', from=1-1, to=1-2]
	\arrow[""{name=1, anchor=center, inner sep=0}, "{i_c^L}"', curve={height=12pt}, two heads, from=1-1, to=2-1]
	\arrow[""{name=2, anchor=center, inner sep=0}, "{f_1^L}"', curve={height=12pt}, from=1-2, to=1-1]
	\arrow[""{name=3, anchor=center, inner sep=0}, "{i_d^L}"', curve={height=12pt}, two heads, from=1-2, to=2-2]
	\arrow[""{name=4, anchor=center, inner sep=0}, "{i_c}"', hook, from=2-1, to=1-1]
	\arrow["{f_0}"', from=2-1, to=2-2]
	\arrow[""{name=5, anchor=center, inner sep=0}, "{i_d}"', hook, from=2-2, to=1-2]
	\arrow["\dashv"{anchor=center}, draw=none, from=1, to=4]
	\arrow["\dashv"{anchor=center, rotate=-90}, draw=none, from=2, to=0]
	\arrow["\dashv"{anchor=center}, draw=none, from=3, to=5]
\end{tikzcd}\]
is a diagram in $\mathcal{C}$ with $i^L_c$ and $i^L_d$ being internal quotients and the square involving $i_c$, $i_d$, $f_0$ and $f_1$ being commutative. Then $f_0$ has a left adjoint given by $ i_c^L f_1^L i_d : d_0 \rightarrow c_0$.
\end{lemma}

The proof follows directly from the same statement in the situation of an ordinary $2$-category, which is a standard fact. Of course a dual version, demanding that $i^L_c$ and $i^L_d$ are internal embeddings, holds true as well.

\subsection{Compactly assembled $\infty$-categories}

The notion of a compactly assembled $\infty$-category is central to this article. We note that if specialized to 1-category theory, a compactly assembled category corresponds to what is also known as a \emph{continuous} category.

\begin{definition}[See also \cite{lurieSAG} Section 21.1.2, \cite{efimov2025ktheorylocalizinginvariantslarge} Definition 1.9 and 1.10] \label{compactlyassembled}
Let $\mathcal{C}$ be an accessible $\infty$-category. We call $\mathcal{C}$ \emph{compactly assembled} if the canonical inclusion $y : \mathcal{C} \rightarrow \Ind(\mathcal{C})$ admits two further left adjoints $\hat{y} \dashv k \dashv y$,
\[\begin{tikzcd}
	{\mathcal{C}} & {\Ind(\mathcal{C}).}
	\arrow[""{name=0, anchor=center, inner sep=0}, "y"', curve={height=12pt}, hook', from=1-1, to=1-2]
	\arrow[""{name=1, anchor=center, inner sep=0}, "{\hat{y}}", curve={height=-12pt}, hook, from=1-1, to=1-2]
	\arrow[""{name=2, anchor=center, inner sep=0}, "k"{description}, from=1-2, to=1-1]
	\arrow["\dashv"{anchor=center, rotate=-90}, draw=none, from=1, to=2]
	\arrow["\dashv"{anchor=center, rotate=-90}, draw=none, from=2, to=0]
\end{tikzcd}\]
A functor $F : \mathcal{C} \rightarrow \mathcal{D}$ will be called \emph{strongly continuous} if the commuting square
\[\begin{tikzcd}
	{\mathcal{C}} & {\Ind(\mathcal{C})} \\
	{\mathcal{D}} & {\Ind(\mathcal{D})}
	\arrow["y", from=1-1, to=1-2]
	\arrow["F"', from=1-1, to=2-1]
	\arrow["{\Ind(F)}", from=1-2, to=2-2]
	\arrow["y"', from=2-1, to=2-2]
\end{tikzcd}\]
is twice left-adjointable, i.e.\ the resulting natural transformations $ k \Ind(F) \rightarrow  F k$ and $\hat{y} F \rightarrow \Ind(F) \hat{y}$ are equivalences.

The resulting subcategory of $\widehat{\Cat_\infty}$ of large $\infty$-categories given by compactly assembled $\infty$-categories and strongly continuous functors will be denoted by $\Cat_{\ca}$. Further, we denote by $\PR^L_{\ca}$ the $\infty$-category of presentable and compactly assembled $\infty$-categories, and strongly continuous left adjoint functors between them, as a (non-full) subcategory of $\widehat{\Cat_\infty}$.
\end{definition}

Note that $y : \mathcal{C} \rightarrow \Ind(\mathcal{C})$ is fully faithful, and therefore there is always a natural transformation $$\hat{y} \rightarrow y.$$

\begin{definition}
Let $\mathcal{C}$ be a compactly assembled $\infty$-category. A map $f : c \rightarrow d$ in $\mathcal{C}$ is called \emph{compact} if $y(f) : y(c) \rightarrow y(d)$ factors through the natural map $\hat{y}(d) \rightarrow y(d)$.
\end{definition}

\begin{remark} Left adjointability of the square in Definition \ref{compactlyassembled} for a functor $F : \mathcal{C} \rightarrow \mathcal{D}$ simply expresses that $F$ preserves filtered colimits. If $F$ itself has a further right adjoint $F^R$, then the statement that $F$ is strongly continuous is equivalent to $F^R$ preserving filtered colimits.
\end{remark}

\begin{example}
If $\mathcal{C}$ is an $\infty$-category with filtered colimits, denote by $\mathcal{C}^\omega$ the full subcategory of \emph{compact objects}, i.e. those $c \in \mathcal{C}$ such that $\mathrm{Map}_{\mathcal{C}}(c, - )$ preserves filtered colimits. Then $\mathcal{C}$ is called \emph{compactly generated}, if $\mathcal{C} \simeq \Ind( \mathcal{C}^\omega )$. If this is the case, $\mathcal{C}$ is in particular compactly assembled with $\hat{y}$ obtained by applying $\Ind$ to the inclusion $\mathcal{C}^\omega \hookrightarrow \mathcal{C}$.
\end{example}

\begin{theorem}[\cite{lurieSAG}, Section 21.1.2] \label{closureunderretracts}
An $\infty$-category $\mathcal{C}$ admitting filtered colimits is compactly assembled iff it is a retract via filtered colimit preserving functors of a compactly generated $\infty$-category. In particular, if $\mathcal{C}$ is a retract of a compactly assembled $\infty$-category $\mathcal{D}$ via filtered colimit preserving functors, then $\mathcal{C}$ is also compactly assembled.
\end{theorem}

Let us point out two special cases of Theorem \ref{closureunderretracts}.

\begin{itemize}
\item If we have an adjunction
\[\begin{tikzcd}
	{\mathcal{C}} & {\mathcal{D}}
	\arrow[""{name=0, anchor=center, inner sep=0}, "R"', curve={height=6pt}, hook, from=1-1, to=1-2]
	\arrow[""{name=1, anchor=center, inner sep=0}, "L"', curve={height=6pt}, two heads, from=1-2, to=1-1]
	\arrow["\dashv"{anchor=center, rotate=-90}, draw=none, from=1, to=0]
\end{tikzcd}\]
with $R$ being fully faithful, and both $L,R$ preserving filtered colimits, then $\hat{y}_\mathcal{C}$ is obtained as the composite $\Ind(L) \hat{y}_\mathcal{D} R$, as a direct application of Lemma \ref{existenceofadjoint}.
\item If we have an adjunction
\[\begin{tikzcd}
	{\mathcal{C}} & {\mathcal{D}}
	\arrow[""{name=0, anchor=center, inner sep=0}, "L", curve={height=-6pt}, hook', from=1-1, to=1-2]
	\arrow[""{name=1, anchor=center, inner sep=0}, "R", curve={height=-6pt}, two heads, from=1-2, to=1-1]
	\arrow["\dashv"{anchor=center, rotate=-90}, draw=none, from=0, to=1]
\end{tikzcd}\]
with $L$ being fully faithful, and both $L,R$ preserving filtered colimits, then the composite $\hat{y}_\mathcal{D} L$ factors through $\Ind(L)$, giving the functor $\hat{y}_\mathcal{C}$. Concretely, this means that for each $c \in \mathcal{C}$, the formal diagram $\hat{y}_\mathcal{D} L(c) = \text{``colim''}_{ i \in I} d_i$ has a cofinal system of objects $L(c_j), j \in J$ in the essential image of $L$, resulting in the diagram $\hat{y}_\mathcal{C} (c) = \text{``colim''}_{ j \in J} c_j$. This is typically inexplicit to compute in practice, but we remark that a map $f : c \rightarrow c'$ in $\mathcal{C}$ is compact iff $L(f)$ is compact in $\mathcal{D}$, see \cite[Lemma 3.26]{aoki2025schwartzcoidempotentscontinuousspectrum}.
\end{itemize}

Compactly assembled presentable $\infty$-categories are furthermore closed under tensor products.

\begin{theorem}[\cite{krause_nikolaus_puetzstueck} Proposition 2.12.2.]
The Lurie tensor product $\otimes$ restricts to a symmetric monoidal structure on $\PR^L_{\ca}$, which is closed.
\end{theorem}

Lastly, the following lemma will be useful for detecting strongly continuous functors.

\begin{lemma}[\cite{krause_nikolaus_puetzstueck}, Lemma 2.12.1] \label{retractsofcompactfunctors} Suppose $F : \mathcal{C} \rightarrow \mathcal{D}$ is a left adjoint functor in $\PR^L$. Then $F$ is strongly continuous iff there exists a commutative square in $\PR^L$
\[\begin{tikzcd}
	{\mathcal{C}} & {\mathcal{D}} \\
	{\mathcal{C}'} & {\mathcal{D}'}
	\arrow["F", from=1-1, to=1-2]
	\arrow["{L_\mathcal{C}}"', hook', from=1-1, to=2-1]
	\arrow["{L_\mathcal{D}}", hook', from=1-2, to=2-2]
	\arrow["{F'}", from=2-1, to=2-2]
\end{tikzcd}\]
such that $F' : \mathcal{C}' \rightarrow \mathcal{D}'$ is a left-adjoint functor preserving compact objects between compactly generated presentable $\infty$-categories $\mathcal{C}', \mathcal{D}'$, and $L_\mathcal{C}$ and $L_\mathcal{D}$ are internal embeddings in $\PR^L$, i.e.\ they are both fully faithful and have right adjoints preserving colimits.
\end{lemma}

\subsection{Dualizable stable $\infty$-categories and localizing invariants}

Let us recall some basic properties of dualizable $\infty$-categories. The standard references are \cite{efimov2025ktheorylocalizinginvariantslarge}, \cite{krause_nikolaus_puetzstueck} and \cite{ramzi2024dualizablepresentableinftycategories}.

\begin{definition}
A compactly assembled presentable $\infty$-category $\mathcal{C}$ is called \emph{dualizable} if it is also stable. Denote by $\PR^L_{\dual} \subset \PR^L_{\ca}$ the full subcategory spanned by dualizable $\infty$-categories. A sequence
$$ \mathcal{C} \xrightarrow{F} \mathcal{D} \xrightarrow{G} \mathcal{E}$$
of dualizable $\infty$-categories $\mathcal{C}, \mathcal{D},  \mathcal{E}$ and strongly continuous left adjoints $F, G$ is called \emph{exact}, or also a \emph{Verdier sequence}, if it is a cofiber sequence in  $\PR^L_{\dual}$, and $F$ is fully faithful.
A functor $F : \PR^L_{\dual} \rightarrow \mathcal{E}$, where $\mathcal{E}$ is an accessible stable $\infty$-category, is called a \emph{localizing invariant}, if $F(0) = 0$ and $F$ maps exact sequences to fiber sequences in $\mathcal{E}$. If $\mathcal{E}$ is furthermore presentable, then $F$ will be called \emph{finitary} if $F$ preserves filtered colimits.
\end{definition}

\begin{theorem}[See \cite{efimov2025ktheorylocalizinginvariantslarge} Proposition 1.65, and \cite{krause_nikolaus_puetzstueck} Proposition 2.12.9]
The inclusion $\PR^L_{\dual} \rightarrow \PR^L_{\st}$ preserves and creates colimits.  Moreover, the Lurie tensor product $\otimes$ restricts to a symmetric monoidal structure on $\PR^L_{\dual}$, which is closed, and the stabilization functor $- \otimes \mathrm{Sp} : \PR^L_{\ca} \rightarrow \PR^L_{\dual}$ is strong monoidal and left adjoint to the inclusion of $\PR^L_{\dual}$ into $\PR^L_{\ca}$.
\end{theorem}

Let $\Cat^{\perf}$ denote the $\infty$-category of stable, idempotent complete $\infty$-categories. The essential image of the functor $\Ind : \Cat^{\perf} \rightarrow \PR^L_{\dual}$ are called \emph{compactly generated} stable $\infty$-categories. There already exists a notion of localizing invariant on $\Cat^{\perf}$, see \cite{blumberg_gepner_tabuada}, which includes for example the algebraic $K$-theory functor, or topological Hochschild homology $\mathrm{THH}$. Efimov observed that the values of any localizing invariant on $\PR^L_{\dual}$ are determined uniquely by their restriction along $\Ind$.

\begin{theorem}[Efimov, see \cite{efimov2025ktheorylocalizinginvariantslarge} Theorem 0.1]
Restriction along $\Ind$
$$\mathrm{Fun}( \PR^L_{\dual} , \mathcal{E} ) \rightarrow \mathrm{Fun}( \Cat^{\perf} , \mathcal{E} )$$
induces an equivalence on the full subcategories of finitary localizing invariants.
\end{theorem}

We will denote the inverse under this functor of a localizing invariant $F$ defined on $\Cat^{\perf}$ by $F^{\cont}$. The following is a simple, but sometimes useful, observation.

\begin{lemma} \label{localizinginvariantreduction}
Let $F^{\cont} : \PR^L_{\dual} \rightarrow \mathcal{E}$ be a finitary localizing invariant, and $\mathcal{D}$ be a dualizable $\infty$-category. Then $F^{\cont}( - \otimes \mathcal{D} )$ is again a finitary localizing invariant.
\end{lemma}

\begin{proof} The operation of $- \otimes \mathcal{D}$ preserves the zero category and filtered colimits. It furthermore preserves cofiber sequences. It remains to argue that if $F : \mathcal{C} \hookrightarrow \mathcal{E}$ is a fully faithful strongly continuous left adjoint, then also $F \otimes \mathrm{id}_\mathcal{D}$ is fully faithful. This follows from the observation that between dualizable $\infty$-categories, the right adjoint $F^R$ of a strongly continuous left adjoint has a further right adjoint $F^{RR}$. In other words, $F$ is an internal embedding in $\PR^L$, and will be mapped to an internal embedding under the image of the $2$-functor $- \otimes \mathcal{D}$.
\end{proof}

\subsection{Topoi}

Let us introduce some general notions about higher topoi. The reference for (higher) topoi that we follow is \cite{luriehtt}. For the sake of this article, the word topos will denote what is usually called an $\infty$-topos, by which we mean a left exact accessible localization of an $\infty$-category of the form $\mathrm{Fun}(C^{\op}, \An)$, where $C$ is some small $\infty$-category. If $\mathcal{X}, \mathcal{Y}$ are two $\infty$-topoi, a geometric morphism $ f : \mathcal{X} \rightarrow \mathcal{Y} $ is defined via its \emph{pullback part} $f^* : \mathcal{X} \leftarrow \mathcal{Y}$, which is required to be a left adjoint functor preserving finite limits. Its right adjoint will be denoted as $f_* : \mathcal{X} \rightarrow \mathcal{Y}$ and called the \emph{direct image} part of $f$. The $\infty$-category of topoi and geometric morphisms will be denoted by $\mathrm{RTop}$, where $\mathrm{R}$ refers to the fact that we consider maps as directed along the right adjoint direct image functors.

If $\mathcal{X}$ is a topos, there exists a unique geometric morphism $\mathcal{X} \rightarrow \An$, whose direct and inverse image parts will be denoted as
\[\begin{tikzcd}
	{\mathcal{X}} & {\An.}
	\arrow[""{name=0, anchor=center, inner sep=0}, "{\mathcal{X}_*}"', curve={height=6pt}, from=1-1, to=1-2]
	\arrow[""{name=1, anchor=center, inner sep=0}, "{\mathcal{X}^*}"', curve={height=6pt}, from=1-2, to=1-1]
	\arrow["\dashv"{anchor=center, rotate=-90}, draw=none, from=1, to=0]
\end{tikzcd}\]
If $\mathcal{X} = \Sh(X, \An)$ is the $\infty$-category of sheaves on some topological space or locale $X$, we also use the notation $X^* \dashv X_*$.

\begin{definition}
Let $\mathcal{X}$ a topos and $\mathcal{E}$ a presentable $\infty$-category. We call the $\infty$-category
$$ \mathcal{X} \otimes \mathcal{E} = \mathrm{Fun}^R( \mathcal{X}^{\op}, \mathcal{E} )$$
the $\infty$-category of $\mathcal{E}$-valued sheaves over $\mathcal{X}$. (Compare with e.g.\ \cite[Section 2.2]{volpe2025operationstopology}.) Objects in the essential image of 
$$ \mathcal{E} \xrightarrow{ \mathcal{X}^* \otimes \mathrm{id}_\mathcal{E} } \mathcal{X} \otimes \mathcal{E} $$
are referred to as \emph{constant} $\mathcal{E}$-valued sheaves. If $E \in \mathcal{E}$, and the context is clear, we will use the notation $\underline{E}$ for the associated constant sheaf with value $E$. We denote the right adjoint to $ \mathcal{X}^* \otimes \mathrm{id}_ \mathcal{E} $ by
$$ \Gamma( \mathcal{X}; - ) : \mathcal{X} \otimes \mathcal{E} \rightarrow \mathcal{E}$$
and call it the \emph{global sections functor}.
\end{definition}

\begin{definition}
Let $f : \mathcal{X} \rightarrow \mathcal{Y}$ be a geometric morphism. We say:
\begin{itemize}
\item $f$ is \emph{contractible} if $f^*$ is fully faithful.\footnote{This is equivalent to  the statement that the relative shape of $\mathcal{X}$ over $\mathcal{Y}$ is trivial in the sense of \cite[Definition 3.1]{volpe2025operationstopology}.}
\item $f$ is \emph{essential} if $f^*$ admits a further left adjoint $f_!$. The functor $f_!$ is referred to as \emph{relative shape} of $\mathcal{X}$ over $\mathcal{Y}$.
\end{itemize}
A topos $\mathcal{X}$ is called \emph{contractible} iff the unique geometric morphism $\mathcal{X} \rightarrow \An$ is contractible, and \emph{locally contractible} iff $\mathcal{X} \rightarrow \An$ is essential. For a locally contractible topos $\mathcal{X}$, the morphism $\mathcal{X}_!$ is sometimes also written as $\Pi_\infty^\mathcal{X}$. The value $\Pi_\infty(\mathcal{X}) = \Pi_\infty^\mathcal{X}(1)$ is called the \emph{shape} of $\mathcal{X}$.
\end{definition}

For more information about shape theory of topoi, we refer the reader to \cite[Appendix A]{lurieha} as well as \cite{HOYOIS20181859}. We will need the following simple lemma about contractible and essential geometric morphisms.

\begin{lemma} \label{preservationofconstantobjects} Let $f : \mathcal{X} \rightarrow \mathcal{Y}$ be a geometric morphism and $\mathcal{E}$ a presentable $\infty$-category.  Then $f^* \otimes \mathrm{id}_\mathcal{E}$ preserves $\mathcal{E}$-valued constant sheaves. Moreover, if $f$ is a contractible and essential geometric morphism then the right adjoint to $f^* \otimes \mathrm{id}_\mathcal{E}$ also preserves constant sheaves.
\end{lemma}

\begin{proof} The statement about $f^* \otimes \mathrm{id}_\mathcal{E}$ is immediate. Now assume that $f$ is contractible and essential, in other words we have an adjunction of the form
\[\begin{tikzcd}
	{\mathcal{X}} & {\mathcal{Y}}
	\arrow[""{name=0, anchor=center, inner sep=0}, "{f_!}", curve={height=-12pt}, two heads, from=1-1, to=1-2]
	\arrow[""{name=1, anchor=center, inner sep=0}, "{f_*}"', curve={height=12pt}, two heads, from=1-1, to=1-2]
	\arrow[""{name=2, anchor=center, inner sep=0}, "{f^*}"{description}, hook', from=1-2, to=1-1]
	\arrow["\dashv"{anchor=center, rotate=-90}, draw=none, from=0, to=2]
	\arrow["\dashv"{anchor=center, rotate=-90}, draw=none, from=2, to=1]
\end{tikzcd}\]
with $f^*$ fully faithful. Since $- \otimes \mathcal{E}$ is a 2-functor on $\PR^L$, the same holds for the diagram
\[\begin{tikzcd}
	{\mathcal{X} \otimes \mathcal{E}} && {\mathcal{Y} \otimes \mathcal{E}}
	\arrow[""{name=0, anchor=center, inner sep=0}, "{f_! \otimes \mathrm{id}_\mathcal{E}}", curve={height=-12pt}, two heads, from=1-1, to=1-3]
	\arrow[""{name=1, anchor=center, inner sep=0}, "{(f^* \otimes \mathrm{id}_\mathcal{E} )_*}"', curve={height=12pt}, two heads, from=1-1, to=1-3]
	\arrow[""{name=2, anchor=center, inner sep=0}, "{f^* \otimes \mathrm{id}_\mathcal{E}}"{description}, hook', from=1-3, to=1-1]
	\arrow["\dashv"{anchor=center, rotate=-90}, draw=none, from=0, to=2]
	\arrow["\dashv"{anchor=center, rotate=-90}, draw=none, from=2, to=1]
\end{tikzcd}\]
where $(f^* \otimes \mathrm{id}_\mathcal{E} )_*$ denotes the right adjoint to $f^* \otimes \mathrm{id}_\mathcal{E}$, and $f^* \otimes \mathrm{id}_\mathcal{E}$ is fully faithful. Therefore we see that
$$ (f^* \otimes \mathrm{id}_\mathcal{E} )_* (\mathcal{X}^* \otimes  \mathrm{id}_\mathcal{E} ) \simeq (f^* \otimes \mathrm{id}_\mathcal{E} )_* (f^* \otimes  \mathrm{id}_\mathcal{E} ) (\mathcal{Y}^* \otimes  \mathrm{id}_\mathcal{E} ) \simeq (\mathcal{Y}^* \otimes  \mathrm{id}_\mathcal{E} ),$$
in other words $(f^* \otimes \mathrm{id}_\mathcal{E} )_* $ preserves constant objects.
\end{proof}

\begin{example} \label{projectionhilbertcube} If $X$ is a paracompact topological space with the homotopy type of a CW-complex, then $X$ is locally contractible, and its shape $\Pi_\infty(X)$ agrees with the homotopy type of $X$, \cite[Remark A.1.4]{lurieha}. In particular, if $I = [0,1]$ is the closed interval and $X = I^S$ is a Hilbert cube for some set $S$, then $\Sh(I^S)$ is an example of a contractible and locally contractible topos.

If furthermore $N$ is a subset of $S$, then the canonical projection $I^S \rightarrow I^N$ arises as the product of the canonical map $I^{S \setminus N} \rightarrow \mathrm{pt}$ with $I^N$, and hence the induced geometric morphism $ \Sh( I^S ) \rightarrow \Sh( I^N )$ is contractible and essential.
\end{example}

\subsection{Locales} \label{locales}

We will use the notion of frames and locales as developed for example in \cite{picado_pultr} and \cite{johnstone1982stone}. Locales should be thought of as a suitable generalization of the notion of topological spaces,\footnote{While it is not technically true that all topological spaces embed fully faithfully into locales, typically in practice no mathematical content is lost by only considering sober spaces, which do embed.} that fix several pathological shortcomings that the notion of topological spaces has. Locales can be understood as the $0$-categorical analogue of topoi. (In fact locales embed fully faithfully into topoi by taking sheaves.) Let us recap some basic notions and fix notation.

A frame is a complete lattice $(F, \leq)$ such that binary meets distribute over arbitrary joins.\footnote{Or in the language of category theory, a cartesian closed presentable category that is a poset.} A left adjoint $f^* : F \rightarrow F'$ is called a \emph{partial frame homomorphism} if $f^*$ commutes with binary meets. It is further called a \emph{frame homomorphism} if $f^*$ also preserves the top element, equivalently all finite meets. Analogously to geometric morphisms, we will denote the corresponding right adjoint by $f_*$. We denote by $\mathrm{Frm}_{\mathrm{part}}$ and $\mathrm{Frm}$ the categories of frames with partial frame homomorphism, respectively frame homomorphisms, and define $\mathrm{Loc}_{\mathrm{part}} = (\mathrm{Frm}_{\mathrm{part}})^{\op} $ and $\mathrm{Loc} = \mathrm{Frm}^{\op}$ as the categories of locales and (partial) continuous maps between them. There exists an adjunction
\[\begin{tikzcd}
	{\mathrm{Top}} & {\mathrm{Loc}}
	\arrow[""{name=0, anchor=center, inner sep=0}, "{\mathcal{O}}", curve={height=-6pt}, from=1-1, to=1-2]
	\arrow[""{name=1, anchor=center, inner sep=0}, "{\mathrm{pts}}", curve={height=-6pt}, from=1-2, to=1-1]
	\arrow["\dashv"{anchor=center, rotate=-90}, draw=none, from=0, to=1]
\end{tikzcd}\]
where the left adjoint sends a topological space $X$ to its frame $\mathcal{O}(X)$, and a locale $L$ to the space $\mathrm{pts}(L)$ with underlying set $\Map_{Loc}(\mathrm{pt}, L)$ where $\mathrm{pt}$ is the locale obtained from the one point space, with frame $[1]$. This adjunction is idempotent, and identifies the subcategory of sober topological spaces with the full subcategory of spatial locales. Soberness is a mild separation condition, which implies $\mathrm{T}0$ and is implied by $\mathrm{T}2$, see \cite[Chapter I.1]{picado_pultr}. For the sake of notation, we sometimes also write $\mathcal{O}(L)$ for the associated frame representing a locale, and call it the \emph{frame of opens} of $L$. The notion of partial frame homomorphism corresponds to the notion of a \emph{partial continuous map} $f : X \rightarrow Y$, whose domain $U$ is open in $X$. We note that $\mathrm{Frm}_{\mathrm{part}}$ and $\mathrm{Frm}$ are naturally $(1,2)$-categories, since their hom-sets are naturally posets.

\begin{example}
If $F$ is a frame and $U \in F$ an open, then $F_{/U} = \{ V \in F ~|~ V \leq U \}$ is again a frame, called the \emph{open sublocale} associated with $U$. The inclusion $F_{/U} \subset F$ is a partial frame homomorphism. It has an right adjoint $ - \wedge U : F \rightarrow F_{/U}$, which is frame homomorphism (and hence an internal right adjoint in $\mathrm{Frm}_{\mathrm{part}}$), corresponding topologically to the inclusion of the open sublocale. 

Any other partial frame homomorphism $f^* : F \rightarrow F'$ naturally factors as
$$  F \xrightarrow{f^*} F'_{/f^*(1)} \rightarrow F'$$
where the left functor is a frame homomorphism. Topologically, this means that every partial map of locales has an open domain $f^*(1)$, on which it is an actual map of locales.
\end{example}

\begin{construction}[The non-Hausdorff one-point compactification] \label{onepointcompact}
The inclusion of the (wide) subcategory $\mathrm{Frm} \hookrightarrow \mathrm{Frm}_{\mathrm{part}}$ admits a left adjoint, which is given by sending a frame $F$ to the frame $F_\top$ given by adding a single new top element. Topologically, this adds a new point $+$ given by the frame homomorphism $ +^* : F_\top \rightarrow [1]$ with $+^*(U) = 0$ for all $U \in F$. This point is a \emph{co-generic closed} point, in the sense that its only open neighborhood is the top element $1$. The adjunction
\[\begin{tikzcd}
	{\mathrm{Frm}} & {\mathrm{Frm}_{\mathrm{part}}}
	\arrow[""{name=0, anchor=center, inner sep=0}, curve={height=12pt}, from=1-1, to=1-2]
	\arrow[""{name=1, anchor=center, inner sep=0}, "{(-)^+}"', curve={height=12pt}, from=1-2, to=1-1]
	\arrow["\dashv"{anchor=center, rotate=-90}, draw=none, from=1, to=0]
\end{tikzcd}\]
lifts for formal reasons to an adjunction
\[\begin{tikzcd}
	{\mathrm{Frm}_{/[1]}} & {\mathrm{Frm}_{\mathrm{part}}}
	\arrow[""{name=0, anchor=center, inner sep=0}, "{\mathrm{coim}}"', curve={height=6pt}, from=1-1, to=1-2]
	\arrow[""{name=1, anchor=center, inner sep=0}, "{(-)^+}"', curve={height=6pt}, hook', from=1-2, to=1-1]
	\arrow["\dashv"{anchor=center, rotate=-90}, draw=none, from=1, to=0]
\end{tikzcd}\]
where the right adjoint sends a frame homomorphism $f^* : F \rightarrow [1]$ to the frame $\mathrm{coim}(f^*) = (f^*)^{-1}(0) \subset F$, and the left adjoint $(-)_\top$ becomes fully faithful. Topologically, this corresponds to a fully faithful embedding $L \mapsto L^+$ of $\mathrm{Loc}_{\mathrm{part}}$ into pointed locales $\mathrm{Loc}_{\mathrm{pt}/}$. Its left adjoint sends a pointed locale to the (maximal) open complement of the chosen basepoint. We note that for the case of $L^+$, the newly added point $+$ is closed in $L^+$, with open complement given by the original top element of $\mathcal{O}(L)$.
\end{construction}	

Given a locale $L$, its $\infty$-category of sheaves $\Sh(L) = \Sh(L, \An)$ is given by the full subcategory of $\mathrm{Fun}(\mathcal{O}(L)^{\op},\An)$ consisting of those functors $\mathcal{F}$ such that:
\begin{itemize}
\item $\mathcal{F}(0) \simeq 1.$
\item For all $U, V \in \mathcal{O}(L)$ it holds that
\[\begin{tikzcd}
	{\mathcal{F}(U \vee V)} & {\mathcal{F}(U)} \\
	{\mathcal{F}(V)} & {\mathcal{F}(U \wedge V)}
	\arrow[from=1-1, to=1-2]
	\arrow[from=1-1, to=2-1]
	\arrow[from=1-2, to=2-2]
	\arrow[from=2-1, to=2-2]
\end{tikzcd}\]
is a pullback square.
\item For all directed sets $U_i, i \in I,$ the canonical map
$$ \mathcal{F}( \bigvee_{i \in I} U_i ) \rightarrow \lim_{i \in I} \mathcal{F}( U_i ) $$
is an equivalence.
\end{itemize}

The $\infty$-category $\Sh(L)$ is a left exact and accessible localization of $\mathrm{Fun}(\mathcal{O}(L)^{\op},\An)$ and hence a topos, \cite[Corollary 6.2.1.7 and Lemma 6.2.2.7]{luriehtt}. Moreover, if $f : L \rightarrow L'$ is a partial map of locales, it is straightforward to see that the precomposition $\mathcal{F} \mapsto 	( U \mapsto \mathcal{F}(f^*(U)) )$ preserves the sheaf condition, and hence gives a well-defined functor $f_* : \Sh(L) \rightarrow \Sh(L')$. It follows that $f_*$ has a left adjoint $f^*$, which is given by left Kan extension, followed by sheafification. In case $f$ is an actual map of locales, i.e.\ $f^* : \mathcal{O}(L') \rightarrow \mathcal{O}(L)$ preserves the top element, the induced left adjoint $f^*$ on sheaf $\infty$-categories preserves finite limits, i.e.\ $f_* : \Sh(L) \rightarrow \Sh(L')$ is a geometric morphism. For a partial map of locales $f : L \rightarrow L'$, we have the decomposition
$L \supset f^*(1) \xrightarrow{f} L',$ 
which means that the left adjoint $f^* : \Sh(L') \rightarrow \Sh(L)$ factors as
$$ \Sh(L') \rightarrow \Sh(L)_{/y_{f^*(1)}} \rightarrow \Sh(L)$$
hence does not preserve the terminal object, however pullbacks and binary products.\footnote{The statement that the forget functor from an overcategory preserves contractible limits is given as Lemma 2.2.7. in \cite{gepner2020inftyoperadsanalyticmonads}, see also \cite[Proposition 4.4.2.9]{luriehtt}. The fact that in this case also binary products are preserved is because $y_U \in \Sh(L)$ is a subobject of $1$. } We record this by the statement that there exist functors
$$ \Sh : (\mathrm{Loc}_{\mathrm{part}})^{\op} \rightarrow \PR^L \text{ and } \Sh : \mathrm{Loc} \rightarrow \mathrm{RTop} \hookrightarrow (\PR^L)^{\op}.$$
Topoi in the essential image of $\mathrm{Loc}$ in $\mathrm{RTop}$ are called \emph{localic topoi}.

The functor $\Sh$ has a left adjoint. Given any topos $\mathcal{X}$, the subcategory $\mathcal{X}_{\leq 0} = \mathrm{Sub}(1_\mathcal{X})$ of subobjects of $1$ is a frame, \cite[Proposition 6.4.5.4.]{luriehtt}. The assignment $\mathcal{X} \mapsto \mathcal{X}_{\leq 0} $ is in fact left adjoint to the sheaves functor, \cite[Proposition 6.4.5.7.]{luriehtt}, i.e.\ we have the adjunction\footnote{We remark that this adjunction arises via restriction from the more general sheaves-coidempotents adjunction due to Aoki, given in Theorem \ref{sheafspectrumadjunction}.}
\[\begin{tikzcd}
	{\mathrm{Loc}} & {\mathrm{RTop}.}
	\arrow[""{name=0, anchor=center, inner sep=0}, "{\Sh}"', curve={height=6pt}, hook', from=1-1, to=1-2]
	\arrow[""{name=1, anchor=center, inner sep=0}, "{(-)_{\leq 0}}"', curve={height=6pt}, from=1-2, to=1-1]
	\arrow["\dashv"{anchor=center, rotate=-90}, draw=none, from=1, to=0]
\end{tikzcd}\]
The unit of this adjunction is given by the geometric morphism $\mathcal{X} \rightarrow \Sh(\mathcal{X}_{\leq 0})$, whose pullback part is given by left Kan extending the inclusion $y : \mathcal{X}_{\leq 0} \hookrightarrow \mathcal{X}$. The functor $y$ has a left adjoint $\tau_{\leq -1}$, whose main properties we summarize here.

\begin{proposition} \label{subobjectframeadjunction}
Let $\mathcal{X}$ be a topos. Then there exists an adjunction
\[\begin{tikzcd}
	{\mathcal{X}_{\leq 0} } & {\mathcal{X}}
	\arrow[""{name=0, anchor=center, inner sep=0}, "y"', curve={height=6pt}, hook', from=1-1, to=1-2]
	\arrow[""{name=1, anchor=center, inner sep=0}, "{\tau_{\leq -1}}"', curve={height=6pt}, two heads, from=1-2, to=1-1]
	\arrow["\dashv"{anchor=center, rotate=-90}, draw=none, from=1, to=0]
\end{tikzcd}\]
where the left adjoint is given by sending $X \in \mathcal{X}$ to its \emph{support}, given by the canonical effective epimorphism-monomorphism factorization
\[\begin{tikzcd}
	X && 1 \\
	& {\tau_{\leq -1}(X).}
	\arrow["{\exists !}", from=1-1, to=1-3]
	\arrow[two heads, from=1-1, to=2-2]
	\arrow[hook, from=2-2, to=1-3]
\end{tikzcd}\]
The inclusion $y$ preserves filtered colimits and $\tau_{\leq -1}$ preserves finite products.
\end{proposition}

\begin{proof} We cite the relevant parts from \cite{luriehtt}. The existence of the left adjoint is provided by Proposition 5.5.6.18. The statement that $\tau_{\leq -1}$ preserves finite products is given by \cite[Lemma 6.5.1.2]{luriehtt}. The claim that $y$ preserves filtered colimits follows from the statement that filtered colimits are exact in $\infty$-topoi (Example 7.3.4.7.), and hence monomorphisms are closed under filtered colimits.
\end{proof}

Let $(P, \tau)$ be a Grothendieck topology $\tau$ on a poset $P$. Define the poset $F( P, \tau )$ as the set of all $U \subset P$ that are downward closed and compatible with sieves, in the sense that if $S \subset \{q \in P ~|~ q \leq p \}$ is a covering sieve of $p \in P$, we have $p \in U$ iff $q \in U$ for all $q \in S$.

\begin{proposition}
Let $(P, \tau)$ be a site with $P$ a poset. Then $F( P, \tau )$ is a frame. If $P$ is closed under binary meets then it holds that
$$ \Sh(P, \tau) \simeq \Sh( F( P, \tau ) ).$$
\end{proposition}

\begin{proof} The poset $F( P, \tau )$ can be seen to agree with the localic reflection of $\Sh(P, \tau)$, which is given by $\mathrm{Sub}(1) \simeq \Sh(P, \tau; [1])$, by sending $ \mathcal{F} : P^{\op} \rightarrow [1]$ to the downward closed set $\mathcal{F}^{-1}(1)$. The rest of the claim is a special case of Lemma 6.4.5.6 in \cite{luriehtt}, which states that $\Sh(P, \tau)$ is localic if $P$ is closed under binary meets.
\end{proof}

\begin{warning}
The requirement that $P$ is closed under binary meets cannot be dropped in general. As an example, the trivial Grothendieck topology on $P$ leads to the frame $\Fun(P^{\op},[1])$, which is isomorphic to the set of open sets of the Alexandroff topology. However, it is not the case in general that $\Fun(P^{\op}, \An)$ is equivalent to $\Sh(P_{\Alex})$. Rather, the former agrees with the latter after hypercompletion, see \cite[Appendix A]{Aoki_2023} for a discussion.
\end{warning}

\subsection{Sublocales}

Let us collect some basic notions on sublocales.

\begin{definition} \label{definitionquotientsublocale} Let $f : L \rightarrow M$ be a continuous map of locales. We call $f$ a \emph{quotient map}, if $f^*$ is injective. We call $f$ an \emph{embedding} if $f^*$ is surjective. In this case $L$ will also be called a \emph{sublocale} of $M$.
\end{definition}

\begin{remark}
It is sensible to extend Definition \ref{definitionquotientsublocale} to the case where $f : L \rightarrow M$ is a partial map, and talk about partial quotients and partial embeddings. The notion of partial embedding however does not add anything new: A partial frame homomorphism $f^*$ such that $f^*$ is surjective automatically preserves $1$, since $1$ is obtained the supremum of all opens in $M$. 
\end{remark}

We remark that since frames are in particular posets, it holds that $f$ is an embedding iff the right adjoint $f_*$ is injective. The notion of sublocale can be expressed in multiple different forms.

\begin{proposition}[\cite{picado_pultr}, Chapter III, also \cite{johnstone1982stone}, Chapter II.2] 
Let $L$ be a locale. The following posets are isomorphic.
\begin{enumerate}
\item The set of isomorphism classes of embeddings $\{ S \rightarrow L \}$.
\item The set of isomorphism classes of regular monomorphisms $\{ S \rightarrow L \}$, that is, maps $S \rightarrow L$ that arise as equalizers of pairs of maps $f,g : L \rightarrow M$.
\item The set of subsets $S \subset \mathcal{O}(L)$ that are closed under arbitrary meets and such that for every $V \in S, U \in \mathcal{O}(L)$ it holds that the Heyting implication $U \rightarrow V \in S$.
\item The set of \emph{frame congruences} of $L$, that is, equivalence relations of $\mathcal{O}(L)$ that are subframes of $\mathcal{O}(L) \times \mathcal{O}(L) = \mathcal{O}(L \amalg L)$.
\item The set of \emph{nuclei} of $L$, that is functors $N : \mathcal{O}(L) \rightarrow \mathcal{O}(L)$ such that for all $U,V \in \mathcal{O}(L)$ it holds $U \leq N(U)$, $N^2 = N$ and $N(U \wedge V) = N(U) \wedge N(V)$. 
\end{enumerate}
Moreover, the poset of embeddings is in fact a co-frame, or dually the poset of nuclei (ordered pointwise) is a frame.
\end{proposition}

Let us mention just a few key facts about how to change between these descriptions. An embedding $i : S \hookrightarrow L$ leads to the nucleus $N = i_* i^*$. The corresponding congruence can be understood as the quotient $p : L \amalg L \rightarrow L \amalg_S L$. Here the frame of $L \amalg_S L$ is given as
$$ \{ (U,V) \in \mathcal{O}(L) \times \mathcal{O}(L) ~|~ i^*(U) = i^*(V) \} \subset \mathcal{O}(L) \times \mathcal{O}(L),$$
and $S$ is recovered as the equalizer of the natural maps $i_1, i_2 : L \rightarrow L \amalg_S L$, with $(i_1)^*$ and $(i_2)^*$ being given by the projection onto the first or second variable, and $(i_1)_*(U) = (U,N(U))$. We remark here as well, that the pushforward $p_*$ is given by the formula
$$ p_*(U,V) = ( U \wedge N(V), V \wedge N(U) ),$$
a fact that will prove useful later on. 

\begin{lemma}
Let $i : S \hookrightarrow L$ be the embedding of a sublocale, and $\mathcal{C}$ be a presentable $\infty$-category. Then the induced functor
$$ i_* : \mathrm{Sh}(S,\mathcal{C}) \rightarrow  \mathrm{Sh}(L,\mathcal{C})$$
is fully faithful. 
\end{lemma}

\begin{proof}
Taking presheaves is a $2$-functor, hence sends the internal quotient map
\[\begin{tikzcd}
	{\mathcal{O}(S)} & {\mathcal{O}(L)}
	\arrow[""{name=0, anchor=center, inner sep=0}, "{i_*}"', curve={height=6pt}, hook, from=1-1, to=1-2]
	\arrow[""{name=1, anchor=center, inner sep=0}, "{i^*}"', curve={height=6pt}, two heads, from=1-2, to=1-1]
	\arrow["\dashv"{anchor=center, rotate=-90}, draw=none, from=1, to=0]
\end{tikzcd}\]
to the adjunction
\[\begin{tikzcd}
	{\mathrm{Fun}(\mathcal{O}(S)^{\op}, \mathcal{C})} & {\mathrm{Fun}(\mathcal{O}(L)^{\op}, \mathcal{C})}
	\arrow[""{name=0, anchor=center, inner sep=0}, "{i_*}"', curve={height=10pt}, hook, from=1-1, to=1-2]
	\arrow[""{name=1, anchor=center, inner sep=0}, "{i^*}"', curve={height=10pt}, two heads, from=1-2, to=1-1]
	\arrow["\dashv"{anchor=center, rotate=-90}, draw=none, from=1, to=0]
\end{tikzcd}\]
with fully faithful right adjoint. Then $i_*$ on the level of sheaves is simply obtained by restriction to subcategories, hence still fully faithful.
\end{proof}

Let us highlight two main examples of sublocales.

\begin{example}
Let $L$ be a locale, and $U$ an open of $L$. Then $i^* : - \wedge U : \mathcal{O}(L) \rightarrow \mathcal{O}(L)_{/U}$ is a surjective frame homomorphism. The corresponding sublocale is called the \emph{open sublocale} of $L$ given by $U$. The functor $i^*$ has a further left adjoint $i_! : \mathcal{O}(L)_{/U} \rightarrow \mathcal{O}(L)$ which is given by inclusion. It is a partial frame homomorphism. On the level of $\infty$-categories of sheaves, we get induced adjunctions
\[\begin{tikzcd}
	{\mathrm{Sh}(U,\mathcal{C})} & {\mathrm{Sh}(L,\mathcal{C})}
	\arrow[""{name=0, anchor=center, inner sep=0}, "{i_!}", curve={height=-12pt}, hook, from=1-1, to=1-2]
	\arrow[""{name=1, anchor=center, inner sep=0}, "{i_*}"', curve={height=12pt}, hook, from=1-1, to=1-2]
	\arrow[""{name=2, anchor=center, inner sep=0}, "{i^*}"{description}, two heads, from=1-2, to=1-1]
	\arrow["\dashv"{anchor=center, rotate=-89}, draw=none, from=0, to=2]
	\arrow["\dashv"{anchor=center, rotate=-91}, draw=none, from=2, to=1]
\end{tikzcd}\]
with $i^*$ given by precomposition with $i_!$, and $i_*$ given by precomposition with $i^*$. Furthermore, $i_!$ is fully faithful. We remark that $\mathrm{Sh}(U,\An)$ can be identified with $\mathrm{Sh}(L;  \An)_{/y_U}$ and $i_!$ corresponds to the forget functor, \cite[Section 6.3.5]{luriehtt}.
\end{example}

\begin{example}
Let $L$ be a locale, and $U$ an open of $L$. Then $j^* : - \vee U : \mathcal{O}(L) \rightarrow \mathcal{O}(L)_{U/}$ is a surjective frame homomorphism. The corresponding sublocale is called the \emph{closed sublocale} of $L$ complementary to $U$.
\end{example}

The following simple statement will be very useful later on, hence we want to highlight it. We remark that meets of sublocales are computed as pullbacks, and open / closed sublocales are closed under pullback \cite[Section 6.3]{picado_pultr}.

\begin{proposition}[Second Isomorphism Theorem for Locales] \label{secondiso}
Let $L$ be a locale, $S \hookrightarrow L$ a sublocale and $C \hookrightarrow L$ a closed sublocale.
Then the inclusions $ S \wedge C \hookrightarrow S$ and $ C \hookrightarrow S \vee C $ are closed, and the inclusion
$$ S \hookrightarrow S \vee C $$
induces a homeomorphism on their respective open complements.
\end{proposition}

\begin{proof}
The inclusions $ S \wedge C \hookrightarrow S$ and $ C \hookrightarrow S \vee C $ are obtained from pullback along the closed inclusion $C \hookrightarrow L$, hence are themselves closed. Now let $U$ denote the open complement of $C$, and denote the nucleus corresponding to $S$ by $N$. The nucleus corresponding to $S \vee C$ is then given by $N \wedge ( - \vee U )$. We can verify that the left adjoint to the inclusion
$$ S \hookrightarrow S \vee C  $$
is simply given by $N$, as for $V$ open we have
$$ N( N(V) \wedge (V \vee U) ) = N(V) \wedge N( V \vee U ) = N( V ),$$
as $N(V \vee U ) \geq N(V)$. This means the induced adjunction on the open complements is given as
\[\begin{tikzcd}
	{\mathcal{O}(S)_{ / N(U) } } & {\mathcal{O}(S \vee C)_{ / U } }
	\arrow[""{name=0, anchor=center, inner sep=0}, "{- \wedge U}"', curve={height=12pt}, from=1-1, to=1-2]
	\arrow[""{name=1, anchor=center, inner sep=0}, "N"', curve={height=12pt}, from=1-2, to=1-1]
	\arrow["\dashv"{anchor=center, rotate=-90}, draw=none, from=1, to=0]
\end{tikzcd}\]
We simply need to verify that these two adjoints are in fact inverse to each other. To see this, we verify the following:
\begin{itemize}
\item Let $V \in \mathcal{O}(S)_{ / N(U) }$, i.e.\ $V = N(V)$ and $V \leq N(U)$. Then
$$ N( V \wedge U ) = N(V) \wedge N(U) = N(V) = V $$
since $V \leq N(U)$ implies $N(V) \leq N(U)$.
\item Let $W \in \mathcal{O}(S \vee C)_{ / U }$, i.e.\ $W = N(W) \wedge ( W \vee U)$, and $W \leq U$. Then
$$ W = N(W) \wedge ( W \vee U) = ( N(W) \wedge W ) \vee ( N(W) \wedge U ) = W \vee ( N(W) \wedge U )$$
hence $ N(W) \wedge U \leq W$. But $W \leq U$ and $W \leq N(W)$ also imply $W \leq N(W) \wedge U$, hence $W = N(W) \wedge U$.
\end{itemize}
\end{proof}

A core property of sheaves on a locale is given by the existence of open-closed decompositions.\footnote{See also \cite[Appendix A.8]{lurieha} for more information.}

\begin{proposition}[Open-closed decomposition] \label{opencloseddecomposition}
Let $L$ be a locale, $i : U \hookrightarrow L$ an open sublocale, and $j : U^c \hookrightarrow L$ its closed complement. Let $\mathcal{C}$ be a presentable $\infty$-category. Then the sequence
\[\begin{tikzcd}
	{\Sh(U,\mathcal{C})} & {\Sh(L,\mathcal{C})} & {\Sh(U^c,\mathcal{C})}
	\arrow[""{name=0, anchor=center, inner sep=0}, "{i_!}", curve={height=-12pt}, from=1-1, to=1-2]
	\arrow[""{name=1, anchor=center, inner sep=0}, "{i^*}", curve={height=-12pt}, from=1-2, to=1-1]
	\arrow[""{name=2, anchor=center, inner sep=0}, "{j^*}", curve={height=-12pt}, from=1-2, to=1-3]
	\arrow[""{name=3, anchor=center, inner sep=0}, "{j_*}", curve={height=-12pt}, from=1-3, to=1-2]
	\arrow["\dashv"{anchor=center, rotate=-90}, draw=none, from=0, to=1]
	\arrow["\dashv"{anchor=center, rotate=-90}, draw=none, from=2, to=3]
\end{tikzcd}\]
is a cofiber sequence in $\PR^L$. If $\mathcal{C}$ is furthermore stable, it is an exact sequence in $\PR_{\st}^L$. 
\end{proposition}

\begin{proof} The functor $i^*$ is obtained by precomposition with the inclusion $\mathcal{O}(L)_{/U}$, and the functor $j_*$ is obtained by precomposition with $ - \vee U : \mathcal{O}(L) \rightarrow \mathcal{O}(L)_{U/}$, hence it is clear that the composite $i^* j_*$, given by precomposition with the constant partial frame homomorphism $\mathcal{O}(L)_{/U} \rightarrow \mathcal{O}(L)_{U/}$ agrees with the constant functor given by the terminal object. We need to show that the essential image of $j_*$ actually agrees with the kernel of $i^*$.

Assume that $\mathcal{F}$ is a sheaf on $L$ such that $i^* \mathcal{F} \simeq 1$, or in other words, such that for all $V \leq U$ we have $\mathcal{F}(V) \simeq 1$. Define $\mathcal{F}'$ to be the restriction of $\mathcal{F}$ to $\mathcal{O}(L)_{U/}$. We claim this is a sheaf. Since the inclusion $i : \mathcal{O}(L)_{U/} \hookrightarrow \mathcal{O}(L)$ preserves non-empty suprema, the sheaf condition for non-empty coverings is satisfied. Moreover, the initial object is given by $U \in \mathcal{O}(L)_{U/}$, and by assumption $\mathcal{F}(U) \simeq 1$, hence also in this case the sheaf condition holds. Next we argue that there is a natural equivalence $\mathcal{F} \simeq j_* \mathcal{F}'$, or in other words, that for all $W \in \mathcal{O}(L)$ we have that the restriction maps $ \mathcal{F}(W \vee U) \rightarrow \mathcal{F}(W)$ are equivalences. But observe that we have the pullback squares
\[\begin{tikzcd}
	{\mathcal{F}(W \vee U)} & {\mathcal{F}(W)} \\
	{\mathcal{F}(U)} & {\mathcal{F}(W \wedge U).}
	\arrow[from=1-1, to=1-2]
	\arrow[from=1-1, to=2-1]
	\arrow["\lrcorner"{anchor=center, pos=0.125}, draw=none, from=1-1, to=2-2]
	\arrow[from=1-2, to=2-2]
	\arrow[from=2-1, to=2-2]
\end{tikzcd}\]
By assumption, both of the bottom terms are the terminal object, hence the top map is an equivalence.
\end{proof}

\begin{corollary} \label{closedinclusionperfect}
Let $j : C \hookrightarrow L$ be a closed inclusion of locales, and let $\mathcal{C}$ be a presentable $\infty$-category. Then $j_* : \Sh(C,\mathcal{C})  \rightarrow \Sh(L,\mathcal{C}) $ preserves filtered colimits.
\end{corollary}

\begin{proof}
This is immediate since $\Sh(C,\mathcal{C})$ can be identified with the full subcategory of sheaves $\mathcal{F}$ on $L$ such that $i^*( \mathcal{F} ) \simeq 1$. The functor $i^*$ preserves colimits, and filtered colimits with constant value the terminal object remain terminal.
\end{proof}

\begin{corollary}
Let $L$ be a locale, and assume that $\mathcal{C}$ is a presentable stable $\infty$-category. If $\mathrm{Sh}(L, \mathcal{C})$ is a dualizable $\infty$-category it then follows that for any open $U$ of $L$ the sequence given in Proposition \ref{opencloseddecomposition} is an exact sequence of dualizable $\infty$-categories.
\end{corollary}

Given any continuous map of locales $f : L \rightarrow M$, we can factor $f$ uniquely (up to the obvious notion of isomorphism) into
$$ L \twoheadrightarrow \mathrm{im}(f) \hookrightarrow M $$
where the map $f : L  \twoheadrightarrow \mathrm{im}(f)$ is a quotient map, and the map $\mathrm{im}(f) \hookrightarrow M$ is the inclusion of a sublocale corresponding to the nucleus $f_* f^* $, see \cite[IV.1.4.]{picado_pultr}. The existence of this image factorization has the following pleasant consequence.
\begin{lemma}[See \cite{johnstone2002sketches} A.1.3.1.] \label{sublocaleslimits} Let $L$ be a locale. The inclusion $$ \mathrm{Sub}(L) \hookrightarrow \mathrm{Loc}_{/L}$$
has a left adjoint given by $( f : M \rightarrow L) \mapsto \mathrm{im}(f)$.
\end{lemma}
In particular, directed intersections of sublocales agree with limits computed in the category $\mathrm{Loc}$, a fact that is useful for the following corollary.

\begin{corollary} \label{cofiltereddescent}
Let $L$ be a locale, $S_i, i \in I$ a cofiltered family of sublocales of $L$, and $\mathcal{C}$ a presentable $\infty$-category. Denote by $S = \bigwedge_{i \in I} S_i$ the intersection of the sublocales $S_i$. Then
$$ \Sh(S,\mathcal{C}) \simeq \colim_{i \in I} \Sh( S_i,\mathcal{C} ),$$
where the colimit is taken in the category $\PR^L$.
\end{corollary}

\begin{proof}
Since the Lurie tensor product preserves $\PR^L$-colimits in each variable, we can assume w.l.o.g.\ that $\mathcal{C} = \An$. Now the statement follows directly from Lemma \ref{sublocaleslimits}. The functor
$$ \Sh(-,\An) : (\mathrm{Loc}^{\op}) = \mathrm{Frm} \rightarrow \CAlg(\PR^L)$$
is a left adjoint by Theorem \ref{sheafspectrumadjunction}, and filtered colimits in $\CAlg(\PR^L)$ are computed as filtered colimits in $\PR^L$.
\end{proof}

\section{Stable local compactness}

In the following section we will discuss stably locally compact spaces, as well as sheaves on them. It will be convenient to introduce the more general notion of a stably locally compact topos first.

\subsection{Stably locally compact topoi}

\begin{definition} Let $\mathcal{X}$ be a topos. We call $\mathcal{X}$ \emph{compact} if $1 \in \mathcal{X}$ is a compact object. We call $\mathcal{X}$ \emph{locally compact} if $\mathcal{X}$ is a compactly assembled $\infty$-category. A geometric morphism $f : \mathcal{X} \rightarrow \mathcal{Y}$ between locally compact topoi is called \emph{perfect} if $f^*$ is strongly continuous. We call $\mathcal{X}$ \emph{stably locally compact} or also \emph{quasi-separated} if it is locally compact and the diagonal $\Delta : \mathcal{X} \rightarrow \mathcal{X} \otimes \mathcal{X}$ is a perfect geometric morphism. We call $\mathcal{X}$ \emph{stably compact} if it is stably locally compact and compact.\footnote{We remark that quasi-separation is much weaker than the notion of a topos $\mathcal{X}$ being \emph{separated} in the sense of \cite[Definition 6.1.]{Aoki2023posets}, which is the requirement that the diagonal is proper. As an example, the topos of sheaves on the Sierpinski space is quasi-separated but not separated.}
\end{definition}

Local compactness for a topos $\mathcal{X}$ is a pleasant property to have. It was shown independently by Anel and Lejay \cite{anel2018exponentiablehighertoposes}, \cite{anel_lejay_exponentiable_V2}, as well as Lurie \cite[Section 21.1.6 ]{lurieSAG} that a topos $\mathcal{X}$ is locally compact iff $\mathcal{X}$ is an exponentiable object in the $\infty$-category $\mathrm{RTop}$, i.e.\ the functor $\mathcal{X} \otimes - $ has a right adjoint. Moreover, it is clear that if a topos $\mathcal{X}$ is locally compact, and $\mathcal{C}$ is a dualizable stable $\infty$-category, the $\infty$-category of sheaves $\Sh(\mathcal{X},\mathcal{C}) = \mathcal{X} \otimes \mathcal{C}$ is dualizable, hence locally compact topoi give us a healthy supply of dualizable stable $\infty$-categories to work with. The property of stability should be thought of as a separation condition analogous to quasi-separatedness of schemes, and can be phrased in different ways.

\begin{proposition} \label{quasiseparated}
Let  $\mathcal{X}$ be a locally compact topos. Then the following are equivalent.
\begin{enumerate}
\item $\mathcal{X}$ is stably locally compact.
\item The functor $\hat{y} : \mathcal{X} \rightarrow \Ind( \mathcal{X} )$ preserves binary products. 
\item If $f, g$ are two compact maps in $\mathcal{X}$, then so is their product $f \times g$.
\end{enumerate}
\end{proposition}

It will be useful to have the following lemma.

\begin{lemma}[\cite{aoki2025schwartzcoidempotentscontinuousspectrum}, Proposition 3.32] \label{aokilemma} Suppose $\mathcal{C}, \mathcal{D}, \mathcal{E} \in \PR^L_{\ca}$ and $F : \mathcal{C} \times \mathcal{D} \rightarrow \mathcal{E}$ is a functor preserving colimits in both variables. Then the induced left adjoint functor $\mathcal{C} \otimes \mathcal{D} \rightarrow \mathcal{E}$ is continuous iff for any two compact maps $f$ in $\mathcal{C}$ and $g$ in $\mathcal{D}$ the map $F(f,g)$ is a compact map in $\mathcal{E}$.
\end{lemma}

We will add a slight alteration to Lemma \ref{aokilemma}. Note that
$$\Ind(\mathcal{C}) \otimes \Ind(\mathcal{D}) \simeq \Ind( \mathcal{C} \otimes^{rex} \mathcal{D})$$
where $\mathcal{C} \otimes^{rex} \mathcal{D}$ is the tensor product of $\infty$-categories with \emph{finite} colimits, \cite[Section 4.8.1]{lurieha}. Since $\mathcal{C} \otimes^{rex} \mathcal{D} \hookrightarrow \mathcal{C} \otimes \mathcal{D}$ embeds full faithfully, we see that
$$ \Ind(\mathcal{C}) \otimes \Ind(\mathcal{D}) \hookrightarrow \Ind(\mathcal{C} \otimes \mathcal{D})$$
is a fully faithful, strongly continuous left adjoint. Now, if $F : \mathcal{C} \times \mathcal{D} \rightarrow \mathcal{E}$ is as above, we can consider the diagram
\[\begin{tikzcd}
	{\mathcal{C} \otimes \mathcal{D}} & {\mathcal{E}} \\
	{\Ind(\mathcal{C}) \otimes \Ind(\mathcal{D})} & {\Ind(\mathcal{E})} \\
	{\Ind(\mathcal{C} \otimes \mathcal{D})}
	\arrow["F", from=1-1, to=1-2]
	\arrow["{\hat{y} \otimes \hat{y}}"', hook, from=1-1, to=2-1]
	\arrow["{\hat{y}}", hook, from=1-2, to=2-2]
	\arrow["{\tilde{F}}", from=2-1, to=2-2]
	\arrow[hook, from=2-1, to=3-1]
	\arrow["{\Ind(F)}"', from=3-1, to=2-2]
\end{tikzcd}\]
where $\tilde{F}$ is defined as the composite $\Ind(F) i$. We obtain the following as a trivial consequence.

\begin{lemma} \label{tensorlemma} The functor $F$ is strongly continuous iff $\tilde{F} ( \hat{y} \otimes \hat{y} ) \simeq \hat{y} F$.
\end{lemma}

\begin{proof}[Proof of Proposition \ref{quasiseparated}]
This is a special case of Lemma \ref{aokilemma} and Lemma \ref{tensorlemma} applied to $\times : \mathcal{X} \otimes \mathcal{X} \rightarrow \mathcal{X}$.
\end{proof}

\begin{example} \label{overtoposanima}
The $\infty$-topos $\An$ is stably compact, since $\hat{y} : \An \rightarrow \Ind(\An)$ is obtained by applying $\Ind$ to the inclusion $\An^\omega \hookrightarrow \An$, which is closed under finite products. More generally, if $X \in \An$, then the $\infty$-topos $\An_{/X} \simeq \mathrm{Fun}(X, \An)$ is always locally compact. It is compact iff $X \in \An^\omega$. However, stability can often fail: Consider $X = S^1$. Then $\hat{y}$ is given by applying $\Ind$ to the inclusion $\An^\omega_{/S^1} \hookrightarrow \An_{/S^1}$. This inclusion is not closed under binary products: The pullback of the inclusion of a point $\pt \rightarrow S^1$ against itself in $\An$ is $\mathbb{Z}$, which is not a compact anima.
\end{example}

\begin{example}
Let $D$ be a small $\infty$-category admitting binary products. Then $\mathrm{PSh}(D) = \mathrm{Fun}(D^{\op},\An)$ is stably locally compact. Under the identification $\mathrm{PSh}(D) \otimes \mathrm{PSh}(D) \simeq \mathrm{PSh}(D \times D)$, the diagonal $\Delta^*$ is given by left Kan extension of the product functor $\times : D \times D \rightarrow D$, and therefore is strongly continuous.
\end{example}

\begin{definition}
A geometric morphism $i : \mathcal{Y} \rightarrow \mathcal{X}$ is called a subtopos of $\mathcal{X}$ if $i_*$ is fully faithful. It is furthermore called a \emph{perfect} subtopos if $i_*$ preserves filtered colimits.
\end{definition}

\begin{proposition} \label{perfectsubtopoi}
Let $\mathcal{X}$ be a topos and $i : \mathcal{Y} \rightarrow \mathcal{X}$ a perfect subtopos.
\begin{itemize}
\item If $\mathcal{X}$ is compact, then so is $\mathcal{Y}$.
\item If $\mathcal{X}$ is locally compact, then so is $\mathcal{Y}$.
\item If $\mathcal{X}$ is stably locally compact, then so is $\mathcal{Y}$. 
\end{itemize}
\end{proposition}

\begin{proof}
If $1 \in \mathcal{X}$ is compact, then so is $i^*(1) = 1 \in \mathcal{Y}$, proving the first claim. The second claim follows by Theorem \ref{closureunderretracts}. In particular $\hat{y}_\mathcal{Y}$ exists and is given by  $ \Ind(i^*)  \hat{y}_\mathcal{X}  i_*$. This implies the last claim, since if $\mathcal{X}$ is stably locally compact, the functor $\hat{y}_\mathcal{Y}$ is given as a composite of functors that preserve binary products.
\end{proof}

Stably locally compact topoi are also closed under passage to open subtopoi.

\begin{proposition}
Let $\mathcal{X}$ be a topos and $U \in \mathcal{X}$ be a subobject of $1$.
\begin{itemize}
\item If $\mathcal{X}$ is locally compact, then so is $\mathcal{X}_{/U}$.
\item If $\mathcal{X}$ is stably locally compact, then so is $\mathcal{X}_{/U}$.
\item If $\mathcal{X}$ is compact, and $U$ is a compact object, then $\mathcal{X}_{/U}$ is compact.
\end{itemize}
\end{proposition}

\begin{proof}
This is a direct corollary of Theorem \ref{closureunderretracts}, by noting that the forget functor $\mathcal{X}_{/U} \hookrightarrow \mathcal{X}$ is fully faithful, since $U$ is assumed to be a subobject of $1$, and has the colimit preserving left adjoint $- \times U$. Note that the remark following Theorem \ref{closureunderretracts} also states that an object/morphism in $\mathcal{X}_{/U}$ is compact iff it is compact when considered as an object/morphism. The forget functor preserves binary products, hence the statement about stable local compactness follows. The statement about compactness of $\mathcal{X}_{/U}$ is clear by the same reasoning.
\end{proof}

\begin{remark}
Compactness is of course not inherited by passing to arbitrary open subtopoi.
\end{remark}

\subsection{Stably locally compact spaces}

Let us define the main class of topological spaces of interest for the purposes of this article: The so-called \emph{stably locally compact spaces.} Standard References for the theory of stably locally compact spaces are \cite[Section VI-6]{Gierz_Hofmann_Keimel_Lawson_Mislove_Scott_2003}, \cite[Section 9]{Goubault-Larrecq_2013} and \cite{LAWSON_2011}.\footnote{We warn the reader about a conflict of terminology. The term \emph{coherent} as a property of spaces in the references \cite{Goubault-Larrecq_2013} and \cite{Gierz_Hofmann_Keimel_Lawson_Mislove_Scott_2003} refers to the property that the way-below relation is stable under intersection, whereas \emph{coherent space} used in \cite{lehner2025algebraicktheorycoherentspaces} refers to what is otherwise also known as a \emph{spectral space}, and is in accordance with the term \emph{coherent locale}, as used in \cite{johnstone1982stone}. This terminology matches the more general notion of a coherent topos, following \cite[Appendix A]{lurieSAG}.}

\begin{definition} \label{stablycompactframe}
Let $F$ be a frame. We say that $F$ is
\begin{itemize}
\item \emph{locally compact}, if $F$ is compactly assembled when considered as an $\infty$-category,
\item \emph{stably locally compact}, if furthermore $\hat{y} : F \rightarrow \Ind(F)$ is a partial frame homomorphism, and
\item \emph{stably compact}, if $\hat{y} : F \rightarrow \Ind(F)$ is a frame homomorphism.
\end{itemize}
A partial frame homomorphism $f^* : F \rightarrow F'$ between locally compact frames is called \emph{perfect} if it is strongly continuous. We define the categories:
\begin{itemize}
\item $\mathrm{SLCFrm}_{p}$ of stably locally compact frames and partial perfect frame homomorphisms.
\item $\mathrm{SCFrm}$ of stably compact frames and perfect frame homomorphisms.
\end{itemize}
\end{definition}

\begin{definition} \label{stablycompactspace}
A sober topological space $X$ is called (stably) (locally) compact if its corresponding frame $\mathcal{O}(X)$ is (stably) (locally) compact.\footnote{We note that our notion of local compactness for a space is sometimes referred to as \emph{core-compactness} in the literature and is a stronger notion than just the naive requirement that every point has a compact neighborhood.} A partially defined continuous map $f$ with open support is called \emph{perfect}, if the corresponding partial frame homomorphism $f^{-1}$ is. We define the categories:
\begin{itemize}
\item $\mathrm{SLC}_{p}$ of stably locally compact spaces and perfect partially defined continuous maps.
\item $\mathrm{SC}$ of stably compact spaces and perfect continuous maps.
\end{itemize}
\end{definition}

We can rephrase Definitions \ref{stablycompactframe} and \ref{stablycompactspace} into the following more classically known terms.  Recall the notion of the \emph{way below} relation between opens of a frame $F$: If $U \leq V \in F$ are two opens, we say that $U$ is \emph{way below} $V$, in symbols $U \ll V$, if for every directed family of opens $W_i, i \in I,$ it holds that
$$ V \leq \bigvee_{ i \in I } W_i \text{ implies } \exists i \in I : U \leq W_i. \footnote{This is just saying that $U \leq V$ is a \emph{compact morphism} when considering $F$ as a category.}$$
We call an open $U$ \emph{compact open} if $U \ll U$. A frame / space is compact iff $1$ is a compact open.

\begin{proposition}
Let $F$ be a frame.
\begin{itemize}
\item $F$ is locally compact iff for all opens $U \in F$, we have $ U = \bigvee_{V \ll U} V. $
\item $F$ is stably locally compact iff it is locally compact, and whenever $U \ll V_1$ and $U \ll V_2$ it follows that $U \ll V_1 \wedge V_2$.
\item $F$ is stably compact iff $F$ is stably locally compact and $1 \in F$ is a compact open.
\end{itemize}
Let $f^* : F \rightarrow F'$ be a partial frame homomorphism. Then $f^*$ is perfect iff it preserves the way-below relation $\ll$, or equivalently its right adjoint $f_*$ preserves directed joins.
\end{proposition}

The proofs for these facts are straightforward, and can be found (with slight variation) in the proofs given in \cite[Proposition 4.4 and 4.6]{aoki2025schwartzcoidempotentscontinuousspectrum}.

\begin{remark}
It is tempting to refer to perfect maps as \emph{proper}, however, the terminology of \emph{proper map} often additionally assumes closedness in the literature. 
\end{remark}

\begin{proposition} \label{quasiseparatedlocale}
A sober, locally compact space $X$ is stably locally compact iff the diagonal
$$ \Delta : X \rightarrow X \times X $$
is a perfect map. Furthermore, $X$ is compact iff the canonical map $X \rightarrow \mathrm{pt}$ is perfect.
\end{proposition}

\begin{proof}
The inverse image part $\Delta^* : \mathcal{O}(X) \otimes \mathcal{O}(X) \rightarrow \mathcal{O}(X) $ is given by sending $U \otimes V$ to $ U \wedge V$. With this understood, we see that the first claim is a special case of Lemma \ref{aokilemma}. The statement about compactness follows by observing that $\mathcal{O}(\mathrm{pt}) =  \{ \emptyset \subset \{\mathrm{pt}\} \}$ has a single non-trivial compact element and thus $p : X \rightarrow \mathrm{pt}$ is perfect iff $p^{-1}(\{\mathrm{pt}\})$ is compact open.
\end{proof}

By definition, the notions of stably (locally) compact spaces make reference only to their corresponding frames. Since we also assumed sobriety, it follows that the opposite of the category of stably (locally) compact spaces embeds fully faithfully into stably (locally) compact frames. A variant of Birkhoff's completeness theorem guarantees that this is an equivalence of categories.

\begin{theorem}[\cite{johnstone1982stone}, page 311] \label{spatialityoflocallycompact}
A locally compact frame is spatial.
\end{theorem}

\begin{corollary} \label{stablycompactframesandspaces}
There exist equivalences of categories
$$(\mathrm{SLC}_{p})^{\op} \simeq \mathrm{SLCFrm}_{p} \text{ and } (\mathrm{SC})^{\op} \simeq \mathrm{SCFrm}, $$
given on objects by $X \mapsto \mathcal{O}(X)$.
\end{corollary}

There are several inheritance properties we want to discuss. Stably locally compact spaces are closed under perfect subspaces, partial perfect quotients and disjoint unions.

\begin{definition}
Let $L$ be a locale. A \emph{partial quotient} of $L$ is given by a partial map $p : L \rightarrow S$ such that $p^*$ is fully faithful. We call a partial quotient \emph{perfect}, if furthermore $f_*$ preserves directed suprema. A \emph{partial quotient} $p : L \rightarrow S$ is simply called a \emph{quotient}, if $p$ is a (globally defined) map of locales.
\end{definition}

\begin{proposition}
Let $X$ be a locale, and $p : X \rightarrow S$ a partial quotient.
\begin{itemize}
\item If $X$ is locally compact, then so is $S$.
\item If $X$ is stably locally compact, then so is $S$.
\item If $X$ is locally compact, and the domain $p^*(1)$ is compact open, then $S$ is also compact.
\end{itemize}
\end{proposition}

\begin{proof}
This is a direct application of Theorem \ref{closureunderretracts}.
\end{proof}

\begin{example} A particular example is given by considering an open $U$ of a locally compact frame $F$. The inclusion $F_{/U} \hookrightarrow F$ is in fact the inverse image part of a partial perfect quotient, as it has the colimit preserving left adjoint $ - \wedge U $. It follows that the open sublocale $F_{/U}$ is locally compact, and the inclusion $F_{/U} \hookrightarrow F$ is a partial perfect frame homomorphism. If $F$ is stably locally compact, then so is $F_{/U}$, however compactness is only given if $U$ is a compact open of $F$.

If $f^* : F \rightarrow F'$ is a partial perfect frame homomorphism, then in the natural factorization 
$$  F \xrightarrow{f^*} F'_{/f^*(1)} \rightarrow F'$$
the left functor is also perfect, as its right adjoint is given by $f_* : F'_{/f^*(1)} \rightarrow F$, and suprema in the over frame are computed as suprema in $F'$. Hence, analogously to the situation for partial maps of locales, we can think of partial perfect maps as perfect maps defined on some open domain $f^*(1)$.
\end{example}

\begin{definition}
Let $L$ be a locale. A sublocale $i : S \hookrightarrow L$ is called \emph{perfect}, if $i_*$ preserves directed suprema.
\end{definition}

\begin{proposition} \label{perfectsublocales1}
Let $X$ be a locale, and $i : S \hookrightarrow L$ a perfect sublocale.
\begin{itemize}
\item If $X$ is locally compact, then so is $S$.
\item If $X$ is stably locally compact, then so is $S$.
\item If $X$ is compact, then $S$ is also compact.
\end{itemize}
\end{proposition}

\begin{proof}
The proof of this is completely analogous to the proof of Proposition \ref{perfectsubtopoi}.
\end{proof}

\begin{example}
Let $X$ be a locale, and $U$ an open. The sublocale $C$ corresponding to the closed complement of $U$ is given by $i_* : \mathcal{O}(X)_{U/} \hookrightarrow \mathcal{O}(X)$, and is evidently perfect. (The inclusion preserves non-empty suprema, in particular directed ones.)
\end{example}

We now come to closure under coproducts.

\begin{lemma} \label{localcompactnesscoverings}
Let $X$ be a locale, and $U_i, i \in I$ and open cover such that all $U_i$ are locally compact. Then $X$ is locally compact.
\end{lemma}

\begin{proof}
The inclusions $\mathcal{O}(X)_{/U_i} \hookrightarrow \mathcal{O}(X)$ are strongly continuous partial frame homomorphisms, hence preserve the way-below relation. Let $V$ be an open of $X$. Then
$$ V \wedge U_i = \bigvee_{W \ll V \wedge U_i} W$$
since $U_i$ is locally compact. Hence
$$ V = \bigvee  V \wedge U_i = \bigvee_{i \in I, W \ll V \wedge U_i} W$$
is a join of way below elements of $V$.
\end{proof}

\begin{remark}
We cannot expect the analogous statement for stable local compactness to work. Consider the locale $L$ obtained as $[0,1] \amalg_{(0,1)} [0,1]$. Then $L$ admits a cover by two stably locally compact opens, hence is locally compact. However, $L$ is not stably locally compact.
\end{remark}

\begin{proposition}
Let $X_i, i \in I,$ be a collection of locales and let $X = \coprod_{i \in I} X_i$.
\begin{itemize}
\item If all $X_i$ are locally compact, then so is $X$. 
\item If all $X_i$ are stably locally compact, then so is $X$. 
\item If $I$ is finite, and all $X_i$ are compact, then so is $X$. 
\end{itemize}
\end{proposition}

We remark that the frame $\mathcal{O}(X)$ is obtained as the \emph{product} of the frames $\mathcal{O}(X_i)$. The statement applied to finite index sets can also be found in \cite[Proposition 9.2.1]{Goubault-Larrecq_2013}.

\begin{proof} The canonical inclusions $j_i : X_i \hookrightarrow X$ are open, and correspond to the frame homomorphism given by the projection
$$\prod_{i \in I} \mathcal{O}( X_i) \rightarrow \mathcal{O}( X_i)$$
Its left adjoint sends an open $U$ of $X_i$ to the sequence $(j_i)_!(U)$ which is $U$ at value $i$, and $0$ otherwise. The locale $X$ admits an open cover given by the images of the $X_i$ under $j_i$. Hence local compactness of $X$ follows from local compactness of $X_i$ by Lemma \ref{localcompactnesscoverings}.

To see stable local compactness, note that $V \ll V'$ holds in $X$ only if $V$ is contained in the image of finitely many inclusions $j_i$, since $V'$ can be exhausted by the opens $V' \wedge X_i$. Hence if $V \ll V_1$ and $V \ll V_2$ holds, we can restrict ourself to being supported on some finite union of $X_i$, where the statement is clear. The same reasoning applies to the last point about compactness.
\end{proof}

Let us give some explicit examples of stably locally compact locales.

\begin{example} Every locally compact Hausdorff space $X$ is stably locally compact. Moreover, between locally compact Hausdorff spaces, one sees that (partial) perfect maps are exactly the (partial) proper maps.
\end{example}

\begin{example}[The directed interval]
A crucial example of a stably compact space is given by the following. The \emph{directed interval} $I_\leq$ is given by equipping the real interval $I = [0,1]$ with the topology consisting of open sets of the form $[0,a)$ for some $a \in [0,1]$, together with the open set $I$ itself. Another example of a stably locally compact space is given by the open subspace $[0,1)_{\leq}$, which is homeomorphic to the space $\overrightarrow{[0,\infty)}$ mentioned in the introduction.
\end{example}

\begin{remark}
One may equip the reals $\mathbb{R}$ with the topology given by open sets of the form $(-\infty, a)$ for some $a \in \{-\infty\} \sqcup \mathbb{R} \sqcup \{+\infty\}$ to obtain a space $\mathbb{R}_{\leq}$. The frame $(\{-\infty\} \sqcup \mathbb{R} \sqcup \{+\infty\}, \leq)$ is in fact stably locally compact, however $\mathbb{R}_{\leq}$ is \emph{not} sober, hence not a stably locally compact space. Intersections of non-empty opens always remain non-empty, so the sobrification adds the additional point $\{-\infty\}$, which is contained in every non-empty open set. The resulting space is of course homeomorphic to $[0,1)_{\leq}$.
\end{remark}

\begin{example}[Finite spaces] \label{finitespaces}
Any finite frame $F$ is automatically stably compact, as then $\Ind(F) \simeq F$, and hence all structure maps are just equal to the identity. Finite frames, or equivalently finite distributive lattices always correspond to finite posets equipped with the Alexandroff topology, and will be referred to as \emph{finite spaces}. (See \cite[Section 3.4]{lehner2025algebraicktheorycoherentspaces} for a summary.) They are equivalently described as topological spaces that are finite and T0. Special cases worth mentioning are the point $\mathrm{pt}$ and the Sierpinski space $\mathbf{2}$, i.e.\ the space obtained from the poset $[1] = \{0 \leq 1 \}$.

Moreover, any continuous map $f : X \rightarrow Y$ with $X$ being a finite space and $Y$ stably locally compact is automatically perfect. We note that the same does not hold for maps \emph{into} a finite space, as we have just seen in Proposition \ref{quasiseparatedlocale}.
\end{example}

\begin{example}[Locally coherent spaces] \label{coherentspaces} Let us elaborate on the case when $X$ is stably locally compact, and the frame $\mathcal{O}(X)$ is compactly generated. In this case $X$ is called a \emph{locally coherent space}. In case $X$ is furthermore compact, $X$ is also called a \emph{coherent space}, which is equivalent to $X$ arising as the spectrum of a ring. (See \cite[Section 3]{lehner2025algebraicktheorycoherentspaces} for a summary.) A space $X$ is locally coherent if it is sober, the topology of $X$ is generated by compact opens, and the subset of compact opens forms a lower bounded distributive lattice. In fact, Stone duality gives an equivalence of categories
$$ \mathrm{LocCohSp}_p^{\op} \simeq \mathrm{DLatt}_{\mathrm{lb}} $$
where $\mathrm{LocCohSp}_p$ is the full subcategory of $\mathrm{SLC}_p$ spanned by locally coherent spaces, and a locally coherent space is mapped to its lower bounded distributive lattice of compact opens. The inverse sends a lower bounded distributive lattice $D$ to the stably locally compact frame $\mathrm{Ind}(D)$. This equivalence restricts to the classical Stone duality between coherent spaces and bounded distributive lattices. The importance of the category of locally coherent spaces for the category of stably locally compact spaces is the following. The natural adjunction
\[\begin{tikzcd}
	F & {\mathrm{Ind}(F)}
	\arrow[""{name=0, anchor=center, inner sep=0}, "{\hat{y}}", curve={height=-12pt}, from=1-1, to=1-2]
	\arrow[""{name=1, anchor=center, inner sep=0}, "y"{description}, curve={height=12pt}, from=1-1, to=1-2]
	\arrow[""{name=2, anchor=center, inner sep=0}, from=1-2, to=1-1]
	\arrow["\dashv"{anchor=center, rotate=-90}, draw=none, from=0, to=2]
	\arrow["\dashv"{anchor=center, rotate=-90}, draw=none, from=2, to=1]
\end{tikzcd}\]
for any stably locally compact frame $F$ means that every stably locally compact space arises as a (partial) retract of a coherent space, and every partial perfect map of stably locally compact spaces arises as a retract of a partial perfect map between coherent spaces. We also remark that coherent spaces arise exactly as inverse limits of finite posets equipped with the Alexandroff topology, tying in with the previous example.

The additional property of Hausdorffness for locally coherent spaces is also worth elaborating on. A locally coherent space is Hausdorff iff it is totally disconnected. This is the case iff the lattice of compact opens forms a Boolean ring. Totally disconnected locally compact Hausdorff spaces are also called \emph{$td$-spaces}. The additional assumption on $X$ to be compact corresponds to the requirement that the compact opens form a Boolean algebra. In this case $X$ is called a \emph{profinite space}. Profinite spaces arise exactly as inverse limits of finite and discrete sets. Let us summarize this picture in the following table.
\begin{table}[h]
\centering
\begin{tabular}{|l|l|l|}
\hline
\textbf{Spaces} & \textbf{Compact Opens} & \textbf{Compactly Generated Case} \\
\hline
Locally compact 
  & Join-semilattice 
  & \\

Stably locally compact 
  & Lower bounded distributive lattice 
  & Locally coherent \\

Stably compact 
  & Bounded distributive lattice 
  & Coherent \\

Locally compact Hausdorff 
  & Boolean ring 
  & $td$-space \\

Compact Hausdorff 
  & Boolean algebra 
  & Profinite space \\
\hline
\end{tabular}
\end{table}
\end{example}

The definition of a stably (locally) compact frame is no accident. Observe that for a frame $F$, the poset $\Ind(F)$ is again a frame. In fact, $\Ind$ gives an endofunctor $\mathrm{Frm}_{\mathrm{part}} \rightarrow \mathrm{Frm}_{\mathrm{part}}$. We observe that there are natural transformations $ k_F : \Ind(F) \rightarrow F$ and $ \Ind(y)_F : \Ind(F) \rightarrow \Ind^2(F)$, which give $\Ind$ the structure of a comonad on $\mathrm{Frm}_{\mathrm{part}} \rightarrow \mathrm{Frm}_{\mathrm{part}}$. Moreover, one perhaps unusual feature is that $\Ind(y)_F$ is left adjoint to $k_{\Ind(F)}$. The same holds for $\Ind$ defined on $\mathrm{Frm}$ instead of $\mathrm{Frm}_{\mathrm{part}}$ as well. This structure is known as a \emph{lax-idempotent comonad}, also called $\mathrm{KZ}$-comonad.

The theory of lax-idempotent (co)-monads on $2$-categories is very rich, and was originally developed by Kock \cite{KOCK199541} and Zöberlein \cite{Zöberlein1976}.\footnote{See also \cite[Section B1.1]{johnstone2002sketches}} The unique feature of a lax-idempotent comonad $T$ on a $2$-category $\mathcal{C}$ is that for a given object $c \in \mathcal{C}$ there is essentially a unique way for $c$ to have the structure of a coalgebra, given by the existence of a left adjoint to the counit $T(c) \rightarrow c$. Similarly, homomorphisms of algebras are given by maps $f : c \rightarrow d$ satisfying a left adjointability property. The way we phrased the definition of a stably (locally) compact frame and perfect (partial) frame homomorphism was intentional, as it highlights that they are just special cases of the definition of a coalgebra and a morphism of coalgebras for a lax-idempotent comonad. (Compare \cite{KOCK199541} Definition 2.1. and 2.4.) From this perspective, the following proposition is purely formal, as it is a special case of the statement that for a comonad $T$ on $\mathcal{C}$ there is an induced adjunction
\[\begin{tikzcd}
	{\mathrm{CoAlg}(\mathcal{C},T)} & {\mathcal{C}}
	\arrow[""{name=0, anchor=center, inner sep=0}, "forget", curve={height=-12pt}, from=1-1, to=1-2]
	\arrow[""{name=1, anchor=center, inner sep=0}, "T", curve={height=-12pt}, from=1-2, to=1-1]
	\arrow["\dashv"{anchor=center, rotate=-90}, shift left=2, draw=none, from=0, to=1]
\end{tikzcd}\]

\begin{proposition}[\cite{KOCK199541} Proposition 3.3] Consider the comonad $\Ind : \mathrm{Frm}_{\mathrm{part}} \rightarrow \mathrm{Frm}_{\mathrm{part}}$. There is an equivalence of $2$-categories
$$ \mathrm{CoAlg}( \mathrm{Frm}_{\mathrm{part}}, \Ind ) \simeq \mathrm{SLCFrm}_{p}. $$
The analogous statement is true with $\mathrm{Frm}$ and $\mathrm{SCFrm}$ as well.
\end{proposition}

\begin{corollary}[Tychonoff's Theorem for stably (locally) compact spaces]
The forget functor $\mathrm{SLC}_{p} \rightarrow \mathrm{Loc}_{\mathrm{part}}$ admits a left adjoint $\gamma$, given on frames by $F \mapsto \Ind(F)$. In particular, $\mathrm{SLC}_{p}$ admits all limits, and the forget functor to $\mathrm{Loc}_{\mathrm{part}}$ preserves them. The analogous statement is true with $\mathrm{Frm}$ and $\mathrm{SCFrm}$ as well.
\end{corollary}

\begin{remark}
Analogously, also the definition of a compactly assembled $\infty$-category, compactly assembled presentable $\infty$-category, as well as a dualizable $\infty$-category fit into the lax-comonadic picture by considering $\Ind$ as defined on $\Cat^\omega$ (the $\infty$-category of accessible $\infty$-categories with filtered colimits and filtered colimit preserving functors), $\PR^L$ and $\PR^L_{\st}$, ignoring the obvious set theoretic issues. In fact, all of these examples arise from simpler lax-idempotent monads, defined on $\Cat$, $\Cat^{rex}$ and $\Cat_{\st}$, respectively. The $2$-categories $\mathrm{Frm}_{\mathrm{part}}$ and $\mathrm{Frm}$ arise from considering $\mathrm{DLatt}_{\mathrm{lb}}$ and $\mathrm{DLatt}_{\mathrm{bd}}$, the $2$-categories of lower bounded (respectively bounded) distributive lattices.

At present the $(\infty,2)$-categorical version of the theory of lax-idempotent monads is only work in progress, with upcoming results due to Abellán--Blom \cite{abellanblom:laxidempotent}. Nonetheless, investigating these analogies has been a major driver for this article.
\end{remark}

\begin{example}[The directed Hilbert cube] If $S$ is a set, we call the product $I^S_\leq$ the \emph{directed Hilbert cube} of cardinality $S$, which is a stably compact space.
\end{example}

We will need to talk about \emph{saturated compact subsets}. Let us give a further definition.

\begin{definition}
Let $F$ be a frame. A nonempty subset $\mathcal{F} \subset F$ is called \emph{Scott open filter} if:
\begin{itemize}
\item $\mathcal{F}$ is upward closed,
\item $\mathcal{F}$ is closed under finite intersections, and
\item $\mathcal{F}$ is \emph{Scott open}, which means that if $U_i, i \in I$ is a directed family of opens and $\bigvee_{i \in I} U_i \in \mathcal{F}$, then there exists $i \in I$ such that $U_i \in \mathcal{F}$.
\end{itemize}
\end{definition}

Note that the definition of Scott openness of course abstractly mimics the property of the open neighborhood filter of a compact subset. This is substantiated by the following theorem. If $X$ is a topological space, and $K \subset X$ a subset, we say that $K$ is \emph{saturated}, if $K$ arises as an intersection of open sets. Write $Q(X)$ for the set of \emph{saturated compact} subsets of $X$.

\begin{theorem}[Hofmann-Mislove, see \cite{Gierz_Hofmann_Keimel_Lawson_Mislove_Scott_2003} Theorem II-1.20.] Let $X$ be a sober space. There is an isomorphism of posets
$$ Q(X)^{\op} \cong \mathrm{OFilt}( \mathcal{O}(X) ),$$
where a saturated compact subset $K \subset X$ is sent to the Scott open filter of open sets of $X$ containing $K$.
\end{theorem}

This allows us to switch between the topological and frame-theoretic viewpoint.

\begin{example}
Let $X$ be a sober space, and $x \in X$ a point. Then the open neighborhood filter $\mathcal{U}_x = \{ V \subset X ~\mathrm{ open} ~|~ x \in V \}$ is a Scott open filter. It corresponds to a saturated compact set $S_x$, which contains $x$ and is contained in every open neighborhood of $x$, which is the \emph{saturated compact closure} of $\{x\}$. We note that this implies that for every open $U$ and point $x \in U$ there exists a saturated compact set $S$ such that $x \in S \subset U$. If $X$ is Hausdorff, then $S_x = \{x\}$.
\end{example}

\begin{theorem}[Urysohn's Lemma for stably compact spaces, see \cite{Goubault-Larrecq_2013} Theorem 9.4.11] \label{urysohn}
Let $X$ be a stably compact space. For every compact saturated subset $S$ and closed subset $C$ of $X$ such that $S \cap C = \emptyset$, there exists a proper map $f : X \rightarrow I_\leq$ such that $f$ has constant value $0$ on $S$ and constant value $1$ on $C$.
\end{theorem}

Let us phrase Urysohn's Lemma in a more categorical way. 

\begin{theorem} \label{urysohnembedding}
Let $X$ be a stably compact space. There exists a perfect embedding
$$ i : X \hookrightarrow I_\leq^S$$
for some set $S$.
\end{theorem}

\begin{proof}
Let $S$ be the set of perfect maps $X \rightarrow I_\leq$. Then, by Tychonoff's Theorem, there exists a canonically induced perfect continuous map $i : X \rightarrow I_\leq^S$. For each $f \in S$, denote by $p_f : I_\leq^S \rightarrow I_\leq$ the associated projection. To see that $i$ is an embedding we need to show that $i^*$ is surjective. Let $U$ be an open of $X$. Denote by $C$ the closed complement. For each point $x \in U$ pick a saturated compact neighborhood $ x \in S_x \subset U$, together with a perfect map $f_x : X \rightarrow I_\leq$ which is $0$ on $ S_x$ and $1$ on $C$. Then $[0,1)$ is open in $I_\leq$ and hence $f_x^*([0,1)) = i^* (p_{f_x})^*([0,1)) \subset U$ is in the image of $i^*$, and moreover contains the point $x$. The image of $i^*$ is closed under suprema, and thus we see that $U = \bigvee_{x \in U} f_x^*([0,1))$ is also in the image of $i^*$.
\end{proof}

\begin{remark}
A curious observation about the proof of Theorem \ref{urysohnembedding} is that, when applied to the case when $X = P_{\Alex}$ for a finite poset $P$ the proof constructs a monotone embedding of $(P,\leq)$ into a finite-dimensional cube $(I^n,\leq)$.
\end{remark}

\begin{corollary}[Urysohn's Lemma for stably compact spaces, categorical formulation] \label{urysohncategorical}
The directed interval $I_{\leq}$ generates the category $\mathrm{SC}$ under limits.\footnote{We warn the reader that this notion is strictly weaker than the statement that the full subcategory spanned by $I_{\leq}$ in $\mathrm{SC}^{\op}$ is \emph{dense} in the sense that $\mathrm{SC}^{\op} \rightarrow \mathrm{Fun}( \{I_{\leq} \}, \mathrm{Set})$ is fully faithful, see \cite[Chapter X.6]{maclane1998categories}.}
\end{corollary}

\begin{proof}
Let $X$ be stably compact. Pick a perfect embedding $i : X \hookrightarrow I_\leq^S$ for some $S$. Then $X$ is given as the equalizer of its congruence relation  $  I_\leq^S \rightrightarrows I_\leq^S \amalg_X I_\leq^S$, which is a diagram in stably compact spaces. Perfectly embed the latter again into $I_\leq^{S'}$ for some set $S'$. Then $X$ is equal to the equalizer of  $  I_\leq^S \rightrightarrows I_\leq^{S'}$.
\end{proof}

\begin{remark}
One can show that $I_\leq$ is $\omega_1$-compact, when considered as an object of $\mathrm{SCFrm}$, where $\omega_1$ is any choice of uncountable regular cardinal. It follows that $\mathrm{SCFrm}$ is a presentable category. This is perhaps surprising, as the category of frames $\mathrm{Frm}$ is \emph{not} presentable. The proof given here is completely analogous to the proof of Ramzi's theorem given by Efimov in \cite[Appendix C]{efimov2025ktheorylocalizinginvariantslarge}, which states that the $\infty$-category $\PR^L_{\dual}$ is presentable.
\end{remark}

Let us now discuss one-point compactifications of stably locally compact spaces. Recall the adjunction
\[\begin{tikzcd}
	{\mathrm{Loc}_{\mathrm{part}}} & {\mathrm{Loc}_{\mathrm{pt}/}}
	\arrow[""{name=0, anchor=center, inner sep=0}, "{(-)^+}"', curve={height=6pt}, hook, from=1-1, to=1-2]
	\arrow[""{name=1, anchor=center, inner sep=0}, "{\mathrm{coim}}"', curve={height=6pt}, from=1-2, to=1-1]
	\arrow["\dashv"{anchor=center, rotate=-90}, draw=none, from=1, to=0]
\end{tikzcd}\]
discussed in Construction \ref{onepointcompact}.
\begin{proposition} \label{adjunctionstablylocallycompact}
The given adjunction between $\mathrm{Loc}_{\mathrm{part}}$ and  $\mathrm{Loc}_{\mathrm{pt}/}$ descends to an adjunction
\[\begin{tikzcd}
	{\mathrm{SLC}_{p}} & {\mathrm{SC}_{\mathrm{pt}/}}
	\arrow[""{name=0, anchor=center, inner sep=0}, "{(-)^+}"', curve={height=6pt}, hook, from=1-1, to=1-2]
	\arrow[""{name=1, anchor=center, inner sep=0}, "{\mathrm{coim}}"', curve={height=6pt}, from=1-2, to=1-1]
	\arrow["\dashv"{anchor=center, rotate=-90}, draw=none, from=1, to=0]
\end{tikzcd}\]
with fully faithful right adjoint.
\end{proposition}

\begin{proof}
We observe the following facts.
\begin{itemize}
\item If $F$ is a stably locally compact frame, then so is $F_\top$. Moreover, the newly added top element is unreachable from below by arbitrary joins, and so in particular compact, hence $F_\top$ is stably compact. (See also \cite[Proposition 9.1.10]{Goubault-Larrecq_2013}.)
\item The left adjoint $\mathrm{coim}$ assigns to a pointed stably compact space $x : \mathrm{pt} \rightarrow X$ the interior $U$ of the complement of $\{x\}$ (given as the join of $(x^*)^{-1}(0) \subset \mathcal{O}(X)$, or equivalently $x_*(0)$). The space $U$ is, as an open subspace of a stably compact space, again stably locally compact.
\item Given a partial frame homomorphism $f^* : F \rightarrow F'$ the associated pushforward on one-point compactifications is given by $(f_*)_\top : (F')_\top \rightarrow (F)_\top$, which agrees with $f_*$ on $F'$ and preserves the newly added top elements. If $f_*$ preserves directed suprema, so does $(f_*)_\top$, as can be checked by verifying two cases.
\item If $(F,x)$ and $(F',x')$ are two pointed frames with coimages $U = x_*(0)$ and $U' = {x'}_*(0)$, then a pointed frame homomorphism $f^*: F' \rightarrow F$ between pointed frames satisfies $f_*( x_* (0)) = {x'}_*(0)$, and therefore the induced right adjoint $\mathrm{coim}(f)_* : F_{/U} \rightarrow F'_{/U'}$ is simply given by $f_*$. If $f_*$ preserves directed suprema, so does $\mathrm{coim}(f)_*$.
\end{itemize}
\end{proof}

\begin{example} \label{scottopenofonepoint}
Assume that $X$ is stably locally compact and consider the one-point compactification $X^+$. Then any Scott open filter $\mathcal{F} \subset  \mathcal{O}(X)$ produces a Scott open filter of $\mathcal{O}(X^+)$, by adding the top element. All Scott open filters of $\mathcal{O}(X^+)$ are of this form, with the exception of $\{1\}$. Hence we see that $Q(X^+)$ arises from $Q(X)$ by \emph{freely adding a bottom element}.
\end{example}

\begin{corollary}[Urysohn's Lemma for stably locally compact spaces] \label{urysohnstablylocallycompact} Let $X$ be stably locally compact.
\begin{itemize}
\item Let $S \subset X$ be saturated compact, and $C \subset X$ closed, such that $S \cap C = \emptyset$. Then there exists a partial perfect map $f : X \rightarrow [0,1)_{\leq}$, which is identically $0$ on $S$ and has open domain of definition contained in $C^c$.
\item There exists a perfect embedding $X \hookrightarrow [0,1)_{\leq}^S$ for some set $S$.
\end{itemize}
In particular, the category $\mathrm{SLC}_p$ is generated under limits by $[0,1)_{\leq}$.
\end{corollary}

\begin{proof}
Let $S \subset X$ be saturated compact, and $C \subset X$ closed be given, such that $S \cap C = \emptyset$. Then $S \subset X \subset X^+$ remains saturated compact. The set $C' = C \cup \{\infty\}$ is closed in $X^+$, and it holds that $S \cap C' = \emptyset$ as subsets of $X^+$. Applying the Urysohn lemma for stably compact spaces produces a perfect map $f : X^+ \rightarrow I_\leq$, which is identically $0$ on $S$ and identically $1$ on $C'$. In particular, it is a pointed map and thus corresponds to a partial perfect map $f : X \rightarrow [0,1)_{\leq}$ with open domain given by $f^*([0,1))$, which is contained in $C^c$. 

The remaining two claims now follow completely analogously to the same statements given for stably compact spaces.
\end{proof}

\begin{remark}
As an application of the patch topology functor, which we will introduce later in Section \ref{patchtopology}, we recover the classical statement that any compact Hausdorff space embeds as a closed subset of some Hilbert cube $[0,1]^S$ for some set $S$, when equipped with the usual (euclidean) topology. We also get as a direct corollary the slightly lesser known fact that any locally compact Hausdorff space $X$ embeds as a \emph{closed} subset in some half-open cube $[0,1)^S$.
\end{remark}

\begin{remark}
Whereas the use of stably compact spaces requires the use of the axiom of choice at several points for the development of the theory, the same is not necessary for the notion of stably compact frame/locale, which has a robust constructive theory, a fact that has been observed early on during the development of pointfree topology, \cite{johnstone1983point}.  This comment is not merely philosophical: Any constructively valid theorem gives statements that are true internally to any 1-topos - Important special cases would be given by working over a given base space, or by working equivariantly with respect to some fixed group $G$. These examples are crucial in algebraic and geometric topology. As an example of how far one can go with applying constructive techniques when working within algebraic geometry, see the work of Ingo Blechschmidt \cite{blechschmidt2021usinginternallanguagetoposes}.

However, we will not insist on this distinction during this article for purely pragmatic reasons. First of all, neither the theory of higher categories, nor the theory of $K$-theory of large categories has a worked out constructive version at the moment of writing. Second of all, the proof strategy for Theorem \ref{maintheoremmain} that we employ in this article is highly dependent on the existence of enough points, and thus only classically valid.
\end{remark}

\subsection{Sheaves on stably locally compact spaces}

Let us now come to the question of how the notion of stable local compactness of spaces interacts with the notion of stable local compactness of topoi. Recall that any topos $\mathcal{X}$ has an associated frame $\mathcal{X}_{\leq 0}$, given by the set of subobjects of $1$.

\begin{theorem} \label{localicreflection}
Let $\mathcal{X}$ be a topos. Consider the frame $\mathcal{X}_{\leq 0}$.
\begin{itemize}
\item If $\mathcal{X}$ is locally compact, then $\mathcal{X}_{\leq 0}$ is locally compact.
\item If $\mathcal{X}$ is stably locally compact, then $\mathcal{X}_{\leq 0}$ is stably locally compact.
\item If $\mathcal{X}$ is compact, then $\mathcal{X}_{\leq 0}$ is compact.
\end{itemize}
\end{theorem}

\begin{proof}
By Proposition \ref{subobjectframeadjunction} we have an adjunction
\[\begin{tikzcd}
	{\mathcal{X}_{\leq 0}} & {\mathcal{X}}
	\arrow[""{name=0, anchor=center, inner sep=0}, "i"', curve={height=6pt}, hook, from=1-1, to=1-2]
	\arrow[""{name=1, anchor=center, inner sep=0}, "{\tau_{\leq -1}}"', curve={height=6pt}, two heads, from=1-2, to=1-1]
	\arrow["\dashv"{anchor=center, rotate=-90}, draw=none, from=1, to=0]
\end{tikzcd}\]
where the inclusion $i$ preserves filtered colimits. Hence the left adjoint $\tau_{\leq -1}$ preserves compact objects, proving (3). Points (1) and (2) follow by an application of Lemma \ref{existenceofadjoint}: If the adjoint $\hat{y}_\mathcal{X}$ exists, then $\hat{y}_{ \mathcal{X}_{\leq 0} }$ exists as well and is given by $    \Ind( \tau_{\leq -1} ) \hat{y}_\mathcal{X} i $, which, assuming that $\hat{y}$ commutes with binary products, also commutes with binary products.
\end{proof}

Let us now concern ourselves with the converse question. 

\begin{theorem}
Let $X, Y$ be a stably locally compact spaces and $f : X \rightarrow Y$ a partial perfect map.
\begin{itemize}
\item The topos $\Sh(X)$ is a stably locally compact.
\item If $X$ is compact, so is $\Sh(X)$. 
\item The induced functor $f^* : \Sh(Y) \rightarrow \Sh(X)$ is strongly continuous.
\end{itemize}
\end{theorem}

\begin{proof}
Applying $\Sh(-)$ to the commuting squares of partial frame homomorphisms
\[\begin{tikzcd}
	{\mathcal{O}(X)} & {\Ind(\mathcal{O}(X))} \\
	{\mathcal{O}(Y)} & {\Ind(\mathcal{O}(Y))}
	\arrow[""{name=0, anchor=center, inner sep=0}, "{\hat{y}}"{description}, curve={height=-12pt}, hook, from=1-1, to=1-2]
	\arrow[""{name=1, anchor=center, inner sep=0}, "k"{description}, curve={height=-12pt}, two heads, from=1-2, to=1-1]
	\arrow["{f^*}", from=2-1, to=1-1]
	\arrow[""{name=2, anchor=center, inner sep=0}, "{\hat{y}}"{description}, curve={height=-12pt}, hook, from=2-1, to=2-2]
	\arrow["{\Ind(f^*)}"', from=2-2, to=1-2]
	\arrow[""{name=3, anchor=center, inner sep=0}, "k"{description}, curve={height=-12pt}, two heads, from=2-2, to=2-1]
	\arrow["\dashv"{anchor=center, rotate=-90}, draw=none, from=0, to=1]
	\arrow["\dashv"{anchor=center, rotate=-90}, draw=none, from=2, to=3]
\end{tikzcd}\]
shows that $f^* : \Sh(Y) \rightarrow \Sh(X)$ is a $\PR^L$-retract of the induced functor $\Sh(\Ind(\mathcal{O}(Y)) \rightarrow \Sh(\Ind(\mathcal{O}(X)))$ of $\Ind(f^*)$. The frame $\Ind(\mathcal{O}(X))$ is a coherent frame and thus $\Sh(\Ind(\mathcal{O}(X)))$ is a compactly generated $\infty$-category, see Example \ref{coherentspaces} and \cite[Corollary 3.17]{lehner2025algebraicktheorycoherentspaces}. Similarly for $\Ind(\mathcal{O}(Y))$. The left adjoint functor induced by $\Ind(f^*)$ on sheaves preserves compact generators by construction and is thus strongly continuous. Applying Lemma \ref{retractsofcompactfunctors}, we see that $f^* : \Sh(Y) \rightarrow \Sh(X)$ is a strongly continuous functor between compactly assembled $\infty$-categories. In particular, $\Sh(-)$ sends perfect maps between stably locally compact spaces to perfect geometric morphisms. We can directly apply this to get the remaining two open claims: By Proposition \ref{quasiseparatedlocale}, the diagonal $\Delta : X \rightarrow X \times X$ is a perfect map, and hence $ \Delta_* :  \Sh(X) \rightarrow \Sh(X) \otimes \Sh(X)$ is a perfect geometric morphism, which shows stability. The statement about compactness follows the same way by considering the map $X \rightarrow \mathrm{pt}$ instead.
\end{proof}

\begin{remark}
The corresponding statement that for locally compact sober spaces $X$, the topos $\Sh(X)$ is also locally compact is unfortunately not true. This failure is already present at the level of $1$-topoi, see \cite[Example 5.5]{JOHNSTONE1982255}. Johnstone-Joyal characterize those spaces $X$ such that $\Sh(X,\mathrm{Set})$ is compactly assembled as those $X$ that are \emph{metastably locally compact}. It is an interesting question whether an analogous criterion exists for the analogous questions with $\Sh(X,\An)$ instead.
\end{remark}

\begin{corollary}\label{sheavesonstablylocallycompact}  The functor $\Sh : \mathrm{Frm}_{\mathrm{part}} \rightarrow \PR^L$ descends to a well-defined functor
$$ \Sh : \mathrm{SLCFrm}_{p} \rightarrow \PR^L_{\ca}.$$
\end{corollary}

\begin{remark}
The restriction of this claim to stably compact frames was done by Aoki in \cite[Theorem 4.29]{aoki2025schwartzcoidempotentscontinuousspectrum}. The proof given here is completely analogous to the argument given by Aoki.
\end{remark}

In the case of sheaves on stably compact spaces, there exists further structure.

\begin{theorem}[\cite{aoki2025schwartzcoidempotentscontinuousspectrum}, Corollary 5.13]
There exists an adjunction
\[\begin{tikzcd}
	{\mathrm{SCFrm}} & {\mathrm{CAlg}(\PR^L_{\ca}).}
	\arrow[""{name=0, anchor=center, inner sep=0}, "{\mathrm{Sh}}", curve={height=-12pt}, from=1-1, to=1-2]
	\arrow[""{name=1, anchor=center, inner sep=0}, "{\mathrm{Sm}^{con}}", curve={height=-12pt}, from=1-2, to=1-1]
	\arrow["\dashv"{anchor=center, rotate=-90}, draw=none, from=0, to=1]
\end{tikzcd}\]
\end{theorem}



\subsection{de Groot and Verdier duality for stably compact spaces} \label{sectionverdierduality}

There is a very useful duality for stably compact spaces, that can be understood as exchanging the roles of openness with that of (co)-compactness.  Note that the poset $Q(X)^{\op} \cong \mathrm{OFilt}( \mathcal{O}(X) )$ of saturated compact sets, or equivalently Scott open filters, always has directed colimits (computed just as unions of their respective filters), and finite intersections. In case $X$ is stably compact however, more holds: It becomes a frame itself.

\begin{theorem}[de Groot duality, see \cite{Goubault-Larrecq_2013} Theorem 9.1.38] Let $X$ be a stably compact space. Then the set of complements of saturated compact subsets gives a topology on $X$. Denote the resulting topological space by $X^\vee$. The space $X^\vee$ is again stably compact, and it holds that $(X^\vee)^\vee = X$. Moreover, a perfect map $f : X \rightarrow Y$ is again a perfect map $f : X^\vee \rightarrow Y^\vee$ when considering the cocompact topologies.
\end{theorem}

In other words, we obtain an involution $(-)^\vee : \mathrm{SC} \rightarrow \mathrm{SC}$. We note that $(-)^\vee$ is covariant in $1$-morphisms, but reverses $2$-morphisms.

\begin{example}
In the case of $X = I_\leq$ being the directed interval, we have $(I_\leq)^\vee = I_\geq$ is the downwards directed interval, with reversed topology.
\end{example}

\begin{example}
In the case of a coherent space $X$ corresponding to a bounded distributive lattice $D$ under Stone duality, the dual $X^\vee$ is again coherent and corresponds to $D^{\op}$. The restriction of de Groot duality to the subcategory of coherent spaces is also referred to as \emph{Hochster duality}.\footnote{de Groot duality can actually be understood as a formal extension of Hochster duality, using the perspective of lax-idempotent monads.}
\end{example} 

\begin{example}
If $X$ is compact Hausdorff, it is a standard fact that a subset $K \subset X$ is compact saturated iff it is closed. We thus see that for compact Hausdorff spaces it holds that $X^\vee = X$.
\end{example}

Crucial will be the following theorem, which explains the interaction of de Groot duality with the corresponding $\infty$-categories of sheaves.

\begin{theorem}[See \cite{aoki2025schwartzcoidempotentscontinuousspectrum}, 
Theorem 4.46.] \label{verdierduality}
Let $X$ be a stably compact space and let $\mathcal{C}$ be a presentable stable $\infty$-category. Then there is a canonical equivalence
$$ \mathrm{VD}: \Sh(X,\mathcal{C}) \rightarrow \Cosh(X^\vee,\mathcal{C})$$
pointwise given by
$$\mathrm{VD}(\mathcal{F})(X \setminus K) = \colim_{ U \supset K ~\mathrm{ open} } \mathrm{fib}( \mathcal{F}(X) \rightarrow \mathcal{F}(U) )$$
where $K$ is a compact saturated subset of $X$.
\end{theorem}

There is a related version of Theorem \ref{verdierduality}, which works without the assumption of stability of $\mathcal{C}$, or the requirement that $X$ is compact, which is known in the Hausdorff case under the statement ``Sheaves = $\mathcal{K}$-sheaves'', see \cite[Section 7.3.4]{lurieha}.

\begin{definition}
Let $X$ be a stably locally compact space. A $\mathcal{K}$-\emph{sheaf} on $X$ with values in a presentable $\infty$-category $\mathcal{C}$ is a functor $ \mathcal{F} : Q(X)^{\op} \rightarrow \mathcal{C}$ such that
\begin{itemize}
\item $\mathcal{F}( \emptyset ) = 1$
\item For each $C, C' \in Q(X)$ the square
\[\begin{tikzcd}
	{\mathcal{F}( C \cup C' )} & {\mathcal{F}(C')} \\
	{\mathcal{F}(C)} & {\mathcal{F}(C \cap C')}
	\arrow[from=1-1, to=1-2]
	\arrow[from=1-1, to=2-1]
	\arrow[from=1-2, to=2-2]
	\arrow[from=2-1, to=2-2]
\end{tikzcd}\]
is a pullback square.
\item Whenever $C = \bigcap_{i \in I} C_i$ is a filtered intersection of saturated compact sets, the natural map
$$ \colim_{i \in I} \mathcal{F}(C_i) \xrightarrow{\sim} \mathcal{F}(C) $$
is an equivalence.
\end{itemize}
We define $\Sh_\mathcal{K}( X ,\mathcal{C} )$ as the full subcategory of $\mathrm{Fun}( Q(X)^{\op}, \mathcal{C} )$ spanned by $\mathcal{K}$-sheaves.
\end{definition}

\begin{theorem}[\cite{aoki2025schwartzcoidempotentscontinuousspectrum}, Proposition 4.36]  \label{sheavesandksheaves} Let $X$ be a stably compact space, and let $\mathcal{C}$ be a presentable $\infty$-category such that filtered colimits are exact. Then there is an equivalence of $\infty$-categories
$$ (-)_{\mathcal{K}} : \Sh( X ,\mathcal{C} ) \xrightarrow{\sim}  \Sh_\mathcal{K}( X ,\mathcal{C} ) $$
that is obtained on objects by $\mathcal{G}_{\mathcal{K}}( K ) = \colim_{ U \supset K ~\mathrm{ open} } \mathcal{G}(U)$. 
\end{theorem}

\begin{remark}
Theorem \ref{sheavesandksheaves} remains true even if $X$ is only stably locally compact, with essentially the same proof as given by Aoki. However, we will not need this greater generality here.
\end{remark}

\begin{observation} \label{valuesofksheaf} 
Let $X$ be stably compact, and $\mathcal{G}$ a sheaf on $X$.
\begin{itemize} 
\item If $x \in X$ is a saturated point, then the value of $\mathcal{G}_{\mathcal{K}}$ at the saturated compact set $\{x\}$ agrees by definition with the stalk $\mathcal{G}_x$. Note that if $X$ is Hausdorff, any point is saturated.
\item The value $\mathcal{G}_{\mathcal{K}}(X) = \mathcal{G}(X) = \Gamma(X; \mathcal{G})$ is just the value of the global sections of $\mathcal{G}$.
\end{itemize}
\end{observation}

\begin{remark}
Theorem \ref{verdierduality} is in fact directly obtained from Theorem \ref{sheavesandksheaves}, by noting that if $\mathcal{F}$ is a $\mathcal{K}$-sheaf on $X$ and the target $\mathcal{C}$ is stable, the fiber of $ \mathcal{F}(X)^{\mathrm{const}} \rightarrow \mathcal{F}$ is on the nose a cosheaf on $X^\vee$. The process of taking these fibers is an equivalence, again because $\mathcal{C}$ is assumed to be stable. Here, $\mathcal{F}(X)^{\mathrm{const}}$ refers to the constant functor $Q(X)^{\op} \rightarrow \mathcal{C}$ with value $\mathcal{F}(X)$.
\end{remark}

\begin{remark} \label{deGrootdualitysubtelty}
Plugging in $\mathcal{C} = [1]$ in Theorem \ref{sheavesandksheaves} recovers the statement that $ X = (X^\vee)^\vee$ on the level of frames for $X$ stably compact. For stably locally compact spaces this gives a generalization of de Groot duality. Note that in the absence of compactness $Q(X)^{\op}$ is not a frame, as it is lacking a bottom element, however, that is the only defect.
\end{remark}

We will need a slight strengthening of Theorems \ref{verdierduality} and Theorems \ref{sheavesandksheaves}, namely that the equivalences mentioned in both theorems are in fact natural in perfect maps. We first remark that de Groot duality implies the following lemma, by dualizing the definition of a compact subset.

\begin{lemma} \label{reversecompactness}
Let $X$ be stably compact, and $K = \bigcap_{i \in I} K_i$ a filtered intersection of saturated compact sets $K_i$. If $U$ is an open subset of $X$ such that $K \subset U$, then there exists $i \in I$ such that $K_i \subset U$.
\end{lemma}

If $f : X \rightarrow Y$ is perfect, we see that for $K \subset Y$ a saturated compact subset, $f^{-1}(K) \subset U \subset X$ is again saturated compact. Hence we get a functor
$$ f^{-1} : Q(Y) \rightarrow Q(X)$$
preserving arbitrary intersections and finite unions.

\begin{lemma} \label{inverseopencompact}
Let $f : X \rightarrow Y$ be a perfect map, and $C \subset Y$ a saturated compact subset. Then for all open $U \supset f^{-1}(C)$ in $X$ there exists an open $V \subset C$ in $Y$ such that $U \supset f^{-1}(V) \supset f^{-1}(C)$.
\end{lemma}

\begin{proof}
Write $C = \bigcap_{C \ll C'} C'$, for saturated compact $C'$, where $C \ll C'$ means that there exists an open $V$ such that $C \subset V \subset C'$. Applying $f^{-1}$ we get $f^{-1}(C) = \bigcap_{C \ll C'} f^{-1}(C')$. If $U \supset f^{-1}(C)$ by Lemma \ref{reversecompactness} there thus exists $C \subset V \subset C'$ with $V$ open, $C'$ saturated compact, such that $U \supset f^{-1}(C') \supset f^{-1}(V) \supset f^{-1}(C)$.
\end{proof}

Using precomposition, we see that we have an induced functor
$$ f_+ : \Sh_\mathcal{K}( X ,\mathcal{C} ) \rightarrow  \Sh_\mathcal{K}( Y ,\mathcal{C} ), $$
given by $f_*(\mathcal{G})(K) = \mathcal{G}(f^{-1}(K))$.

\begin{proposition} \label{naturality}
Let $f : X \rightarrow Y$ be a perfect map between stably compact spaces. The equivalence provided in Theorem \ref{sheavesandksheaves} is natural in $f$, in the sense that the square
\[\begin{tikzcd}
	{\Sh( X ,\mathcal{C} )} & {\Sh_\mathcal{K}( X ,\mathcal{C} )} \\
	{\Sh( Y ,\mathcal{C} )} & {\Sh_\mathcal{K}( Y ,\mathcal{C} )}
	\arrow["\sim", from=1-1, to=1-2]
	\arrow["{f_*}"', from=1-1, to=2-1]
	\arrow["{f_+}", from=1-2, to=2-2]
	\arrow["\sim", from=2-1, to=2-2]
\end{tikzcd}\]
is commutative.
\end{proposition}

We note that the analogous statement for Verdier duality provided by Theorem \ref{verdierduality} follows immediately as well.

\begin{proof}
Tracing through definitions, we see that this amounts to the question whether for a given sheaf $\mathcal{F}$ on $X$, and $C \subset Y$ compact saturated the natural map
$$ \colim_{ V \supset C ~\mathrm{ open ~in }~Y } \mathcal{F}(f^{-1}(V))    \rightarrow \colim_{ U \supset f^{-1}(C) ~\mathrm{ open ~in }~X } \mathcal{F}(U)  $$
is an equivalence. But we see that the diagram on the left maps cofinally into the diagram on the right by Lemma \ref{inverseopencompact}.
\end{proof}

\begin{remark}
The statement of Proposition \ref{inverseopencompact} is very particular to perfect maps and fails for more general continuous maps. In the context of the six-functor formalism on locally compact Hausdorff spaces, see e.g. \cite{volpe2025operationstopology}, it corresponds to the fact that for proper maps $f$ there is an equivalence $f_! \simeq f_*$.
\end{remark}

\begin{remark} A more top down perspective on Verdier duality is to observe that it is about the statement that the $(\infty,2)$-categories $\mathrm{SC}$ and $\PR^L_{\dual}$ come equipped with dualities $(-)^\vee$, and these dualities commute with the functor that assigns to a stably compact space $X$ its $\infty$-category of sheaves with values in spectra. Both dualities are 2-functors, covariant in 1-morphisms, contravariant in 2-morphisms, and both are involutions. The proof of Verdier duality can be reduced via the theory of lax-idempotent monads to a simple computation for distributive lattices.
\end{remark}

\subsection{The patch topology} \label{patchtopology}

There is a universal way of adding open sets to the topology of a stably locally compact space $X$ to make it Hausdorff, while retaining local compactness. This is called the \emph{patch topology} of $X$. Let us give a short summary of how to work with this topology. The general reference will be the article by Escardó  \cite{ESCARDO200141}, however information can also be found in \cite{Gierz_Hofmann_Keimel_Lawson_Mislove_Scott_2003} and \cite{Goubault-Larrecq_2013}. Since Escardó's work is done purely in the language of locales, whereas older literature uses topological spaces, we will need to discuss quickly how to translate between those pictures.

\begin{definition} Let $L$ be a locale. We call a nucleus $N$ on $L$ \emph{perfect} if it preserves directed suprema.
\end{definition}

\begin{lemma} \label{perfectcharacterization}
Let $S \hookrightarrow X$ be a sublocale of a stably locally compact space $X$. Then the following are equivalent:
\begin{enumerate}
\item $S$ is a perfect sublocale.
\item The induced nucleus $N = i_* i^*$ on $X$ is perfect.
\item $S$ is obtained by a subset $S'$ of $X$, such that the induced subspace topology is stably locally compact, and the embedding $S' \hookrightarrow X$ is a perfect map.
\item The associated quotient map $p : X \amalg X \twoheadrightarrow E $ corresponding to the frame congruence of $S$ is perfect.
\end{enumerate}
In particular, the associated frame congruence $E$ is itself stably locally compact.
\end{lemma}

\begin{proof} We first remark that if $S$ is a perfect sublocale of a stably locally compact space $X$, it is itself stably locally compact by Proposition \ref{perfectsublocales1}.

The equivalence of (1) and (3) is now immediate from the equivalence of the category of stably locally compact spaces and stably locally compact frames given by Corollary \ref{stablycompactframesandspaces}.

To see the equivalence of (1) and (2) is also straightforward. If $i_* $ preserves filtered suprema, then so does $i_* i^*$, as $i^*$ is left adjoint. Conversely, if the nucleus $N$ of $X$ is given, the associated sublocale $S$ is obtained as the frame of opens $\{  U \in \mathcal{O}(X) ~|~ N(U) = U \} \subset \mathcal{O}(X)$, and $i_*$ is simply given by the inclusion. We only need to show that $N$ preserving directed suprema implies that this subset is closed under directed suprema. Let $U_i, i \in I,$ be a directed set of opens such that $U_i = N(U_i)$, and $U = \bigvee_{i \in I} U_i$. Then
$$N(U) = N( \bigvee_{i \in I} U_i ) = \bigvee_{i \in I} N(U_i) = \bigvee_{i \in I} U_i = U,$$
and hence we are done.

To see how (4) is equivalent to the rest, first observe that if (4) holds, then $S \hookrightarrow L$ is obtained as an equalizer of stably locally compact spaces and perfect maps, hence itself perfect. Conversely, if $S$ is a sublocale, the quotient map $p : X \amalg X \twoheadrightarrow E $ is given by the formula
$$ p_*(U,V) = ( U \wedge N(V), V \wedge N(U) ).$$
Hence if the nucleus $N$ is perfect, this map preserves directed suprema, and hence $p$ is perfect.
\end{proof}

We now come to the central notion of this section. The patch topology associated to a stably locally compact space. Denote by $\mathrm{LCH}_p$ the full subcategory of $\mathrm{SLC}_p$ spanned by locally compact Hausdorff spaces. We will simply write $\mathrm{LCH}$ and $\mathrm{SLC}$ for the corresponding wide subcategories given by only considering globally defined perfect maps. Furthermore, for a given stably locally compact space $X$, let us denote by $\mathcal{K}(X)$ the set of subsets of $X$ obtained via perfect embeddings. Note that $\mathcal{K}(X)$ is isomorphic to the opposite of the poset of perfect nuclei on $X$.

\begin{theorem}[The patch topology, \cite{ESCARDO200141}, Theorem 5.8] \label{patchcoreflection}
The inclusion of the full subcategory of locally compact Hausdorff spaces into stably locally compact spaces has a right adjoint,
\[\begin{tikzcd}
	{\mathrm{LCH}} & {\mathrm{SLC}}
	\arrow[""{name=0, anchor=center, inner sep=0}, curve={height=-6pt}, hook, from=1-1, to=1-2]
	\arrow[""{name=1, anchor=center, inner sep=0}, "\patch", curve={height=-6pt}, two heads, from=1-2, to=1-1]
	\arrow["\dashv"{anchor=center, rotate=-90}, draw=none, from=0, to=1]
\end{tikzcd}\]
given by sending a stably locally compact space $X$ to the stably locally compact frame $\mathcal{K}(X)^{\op}$ of perfect nuclei on $X$.
\end{theorem}

\begin{remark}
Since the space $\mathrm{pt}$ is a finite space, every continuous map $\mathrm{pt} \rightarrow X$ is automatically perfect. (See Example \ref{finitespaces}.) It is also Hausdorff, hence we see that
$$ X = \Map( \mathrm{pt}, X ) \cong \Map( \mathrm{pt}, X^{\patch} ), $$
and therefore the underlying set of points of $X^{\patch}$ agrees with that of $X$. The new topology on $X$ is referred to as the \emph{patch topology} of $X$. In light of Lemma \ref{perfectcharacterization} it is also immediate what this topology is. The closed subsets of $X^{\patch}$ are given exactly by the set $\mathcal{K}(X)$ of images of perfect embeddings.
\end{remark}

\begin{remark} The patch topology on a stably compact space gives a compact Hausdorff space. This is seen by observing that the perfect map $X \rightarrow \mathrm{pt}$ is sent to a perfect map $X^{\patch} \rightarrow \mathrm{pt}$.
\end{remark}

\begin{example} In the case of the directed interval $I_{\leq}$, it holds that $I_{\leq}^{\patch} = I$ is the interval equipped with the ordinary metric topology on $I$, see \cite[Example 9.1.33]{Goubault-Larrecq_2013}. More generally, since $patch$ preserves limits, if $S$ is a set we also see that for the directed Hilbert cube $I_{\leq}^S$ it holds that $(I_{\leq}^S)^{\patch} = I^S$. Even more generally, any perfect subspace $K \hookrightarrow I_{\leq}^S$ is obtained as an equalizer of perfect maps between stably compact spaces, and hence $K^{\patch}$ agrees with the induced subspace topology when considering $K$ as a subset of $I^S$.
\end{example}

\begin{example} 
If $X$ is a coherent space, corresponding to a bounded distributive lattice $D$, then $X^{\patch}$ is a profinite set, corresponding to the Boolean algebra $\mathrm{Bool}(D)$ given by Booleanization, \cite[Section 3]{ESCARDO200141}. The patch topology in this case is also referred to as the \emph{constructible topology} on $X$, see \cite[Section 3.2]{lehner2025algebraicktheorycoherentspaces}.
\end{example}

\begin{example}
Let $P$ be a poset which is compactly assembled when considered as an $\infty$-category. This notion is referred to as a \emph{domain} in \cite{Gierz_Hofmann_Keimel_Lawson_Mislove_Scott_2003}, and is also called \emph{continuous poset} by Efimov \cite{efimov2025ktheorylocalizinginvariantslarge}. There are two notable topologies on $P$, one being the Scott topology, where open sets are Scott open upwards closed subsets of $P$, \cite[Chapter II]{Gierz_Hofmann_Keimel_Lawson_Mislove_Scott_2003}, as well as the Lawson topology, which is generated from the Scott topology by adding the lower topology, \cite[Chapter III]{Gierz_Hofmann_Keimel_Lawson_Mislove_Scott_2003}. If $P$ is compact in the Lawson topology, then $P$ equipped with the Scott topology is stably compact, with its de Groot dual given by the lower topology, and its patch topology given by the Lawson topology, \cite[Proposition VI-6.24]{Gierz_Hofmann_Keimel_Lawson_Mislove_Scott_2003}. The condition for a domain $P$ to be compact can be understood in multiple different ways, see \cite[Chapter III-5]{Gierz_Hofmann_Keimel_Lawson_Mislove_Scott_2003}.
\end{example}

While the description of the patch topology in terms of perfect embeddings is elegant, we will need a concrete way on how to build it. Let us highlight two special classes of perfect embeddings for a given stably locally compact space $X$.

\begin{itemize}
\item Any \emph{closed subset} $C \subset X$ is perfect.
\item Any saturated compact subset $S \subset X$ is perfect.
\end{itemize}

\begin{example}
In the case of the directed interval $I_{\leq}$, saturated compact sets are of the form $[0,a]$ for $a \in I$, whereas closed sets are of the form $[b,1]$, for $b \in I$.
\end{example}

\begin{definition}
Let $X$ be a stably locally compact space. A subset $E \subset X$ is called \emph{elementary compact} if it is of the form $C \cup S$ for $C \subset X$ closed, and $S \subset X$ saturated compact. Write $\mathcal{E}(X) \subset \mathcal{K}(X)$ for the set of elementary compact subsets.
\end{definition}

Elementary compact subspaces are closed under finite unions. Moreover, they generate the patch topology in the following sense.

\begin{lemma}[\cite{ESCARDO200141}, Lemma 5.4]
Let $X$ be stably locally compact. The set of elementary compact subspaces $\mathcal{E}(X)$ generates $\mathcal{K}(X)$ under intersections.
\end{lemma}

In other words, when considering reverse inclusions, elementary compact sets form a basis of the frame $\mathcal{K}(X)^{\op} \cong \mathcal{O}(X^{\patch})$. From the perspective of sheaf theory, this means that---at least in principle---working with sheaves on the patch topology only requires an understanding of the poset $\mathcal{E}(X)$ and its induced Grothendieck topology.

\begin{remark} \label{deGrootdualitypatch}
In the case that $X$ is stably compact it is thus clear that the patch topology applied to the de Groot dual $X^\vee$ also in results in $X^{\patch}$. This is since de Groot duality simply interchanges the roles of closed and saturated compact subsets, result in the same sets of elementary compact subspaces.
\end{remark}

\begin{example}
For $X = I_\leq$ the directed interval, it is straightforward to see that $(I_\leq)^{\patch} = I$ is the standard real interval. Since taking the patch topology preserves limits we therefore see that for any directed Hilbert cube $I_\leq^S$ we have $(I^S_\leq)^{\patch} = I^S$. Using this we see that the stably compact versions of the Urysohn lemma produce classical versions of the Urysohn lemma for compact Hausdorff spaces.
\end{example}

As the last part of this section, observe that the Patch topology is even functorial in partial perfect maps. This is easiest to see on the level of nuclei. If $N$ is a perfect nucleus on $X$, and $f : X \rightarrow Y$ a partial perfect map, then $f_* N f^*$ is a perfect nucleus on $Y$. We claim that the adjunction given in Theorem \ref{patchcoreflection} generalizes to this functoriality as well. One could in principle retread the proof given by Escardó by analyzing what happens to partial frame homomorphisms, however, we will circumvent the trouble by using one-point compactifications. To start, let us recall the notion of the (classical) one-point compactification $X^*$ of a locally compact Hausdorff space, given by taking the topology on $X \amalg \{ * \}$ which consists of opens $U \subset X$ and complements of compact subsets $K \subset X$. The following is a standard fact about locally compact Hausdorff spaces.

\begin{proposition}[\cite{BunkeCStarLectures}, Lemma 5.2]
There exists an equivalence of categories
$$ \mathrm{LCH}_p \simeq \mathrm{CH}_{\mathrm{pt}/}, $$
given on objects by sending a locally compact Hausdorff space $X$ to its one-point compactification.
\end{proposition}

\begin{corollary} \label{locallycompacthausdorffadjunction}
There exists an adjunction
\[\begin{tikzcd}
	{\mathrm{LCH}_p} & {\mathrm{CH}}
	\arrow[""{name=0, anchor=center, inner sep=0}, "{(-)^*}"', curve={height=6pt}, from=1-1, to=1-2]
	\arrow[""{name=1, anchor=center, inner sep=0}, curve={height=6pt}, from=1-2, to=1-1]
	\arrow["\dashv"{anchor=center, rotate=-90}, draw=none, from=1, to=0]
\end{tikzcd}\]
with left adjoint given by the inclusion of compact Hausdorff space and continuous maps into $\mathrm{LCH}_p$, and the right adjoint by one-point compactification.
\end{corollary}

\begin{lemma} \label{openinclusionofpatch}
Let $X$ be stably locally compact. Then the inclusion $X^{\patch} \hookrightarrow (X^+)^{\patch}$ is open. 
\end{lemma}

\begin{proof} First note that $\{+\}$ is closed in $(X^+)^{\patch}$, and hence $X$ as the complementary set, equipped with the induced subspace topology, is open. Recall that by Example \ref{scottopenofonepoint}, the set of saturated compact sets of $X^+$ is given by the set of saturated compact sets of $X \subset X \amalg \{ + \}$ together with the entire set $X \amalg \{ + \}$. Since the topology on $(X^+)^{\patch}$ is generated from intersections of opens on $X^+$ and complements of saturated compact sets, we see that the induced subspace topology on the set $X$ is generated from intersections of opens in $X$ and complements of saturated compact sets in $X$. But this simply agrees with the patch topology on $X$.
\end{proof}

\begin{corollary} \label{patchofonepoint}
Let $X$ be stably locally compact. Then there is a natural equivalence
$$ (X^{\patch})^* \cong (X^+)^{\patch}.$$
\end{corollary}

\begin{proof}
The partial perfect map $(X^+)^{\patch} \rightarrow X^{\patch}$ given by the identity on the open domain $X^{\patch}$ gives, by applying the adjunction from Corollary \ref{locallycompacthausdorffadjunction} the continuous map
$$ (X^+)^{\patch} \rightarrow (X^{\patch})^*.$$
This is clearly a bijection. But a continuous bijection between compact Hausdorff spaces is a homeomorphism, therefore proving the claim.
\end{proof}

\begin{corollary} \label{patchcoreflectionpartial}
The inclusion of the full subcategory of locally compact Hausdorff spaces into stably locally compact spaces and partial perfect has a right adjoint,
\[\begin{tikzcd}
	{\mathrm{LCH}_p} & {\mathrm{SLC}_p}
	\arrow[""{name=0, anchor=center, inner sep=0}, curve={height=-6pt}, hook, from=1-1, to=1-2]
	\arrow[""{name=1, anchor=center, inner sep=0}, "\patch", curve={height=-6pt}, two heads, from=1-2, to=1-1]
	\arrow["\dashv"{anchor=center, rotate=-90}, draw=none, from=0, to=1]
\end{tikzcd}\]
given on objects by $X \mapsto X^{\patch}$.
\end{corollary}

\begin{proof}
First note that the adjunction from Proposition \ref{adjunctionstablylocallycompact} applied to stably compact spaces and compact Hausdorff spaces lifts to an adjunction
\[\begin{tikzcd}
	{\mathrm{CH}_{\mathrm{pt}/}} & {\mathrm{SC}_{\mathrm{pt}/}}.
	\arrow[""{name=0, anchor=center, inner sep=0}, curve={height=-6pt}, hook, from=1-1, to=1-2]
	\arrow[""{name=1, anchor=center, inner sep=0}, "{(-)^{\patch}}", curve={height=-6pt}, from=1-2, to=1-1]
	\arrow["\dashv"{anchor=center, rotate=-90}, draw=none, from=0, to=1]
\end{tikzcd}\]
The commutative square of left adjoints 
\[\begin{tikzcd}
	{\mathrm{SLC}_p} & {\mathrm{SC}_{\mathrm{pt}/}} \\
	{\mathrm{LCH}_p} & {\mathrm{CH}_{\mathrm{pt}/}}
	\arrow[from=1-2, to=1-1]
	\arrow[hook, from=2-1, to=1-1]
	\arrow[hook, from=2-2, to=1-2]
	\arrow[from=2-2, to=2-1]
\end{tikzcd}\]
gives, as an application of Lemma \ref{existenceofadjoint}, the statement that the inclusion $\mathrm{LCH}_p \hookrightarrow \mathrm{SLC}_p$, has a right adjoint given by sending a stably locally compact $X$ to the open complement of $\{+\}$ in $(X^+)^{\patch}$. But this agrees with $X^{\patch}$ by Lemma \ref{openinclusionofpatch}.
\end{proof}

\begin{remark}
Note that the \emph{specialization order} equips the underlying set of a coherent space $X$ with the structure of a poset. (This is nothing more than an instance of the category of locales being naturally a $2$-category via its enrichment in posets.) In case $X$ is a stably compact space, the relation $\leq$ is closed as a subset of $X^{\patch} \times X^{\patch}$. A compact Hausdorff space equipped with such a closed poset structure is called a \emph{compact pospace}. Given a compact pospace $(Y, \leq)$, the set of downward closed opens actually forms a stably compact topology on $Y$. These two procedures result in an equivalence of categories
$$ \mathrm{SC} \simeq \mathrm{CompPO},$$
see e.g.\ \cite[Proposition 9.4.10]{Goubault-Larrecq_2013}. We have the following table of correspondences.
\begin{table}[h]
\centering
\begin{tabular}{|l|l|}
\hline
\textbf{Stably compact space} & \textbf{Compact pospace} \\
\hline
Perfect subspace      & Closed subset \\
Saturated compact     & Closed and downward closed \\
Closed                & Closed and upward closed \\
de Groot dual         & Opposite order $\leq^{\mathrm{\op}}$ \\
Patch topology        & Forget order \\
\hline
\end{tabular}
\end{table}
\end{remark}

\section{Nisnevich-type descent for localizing invariants} \label{descentlocalizinginvariants}

In this section we are concerned with statements about descent for finitary localizing invariants applied to sheaves on stably (locally) compact spaces. This is somewhat inspired by the well-known notion of Nisnevich-Descent in the context of $K$-theory of schemes. Let us fix throughout the convention that $F : \Cat^{\perf} \rightarrow \mathcal{E}$ denotes a localizing invariant. We note that all statements made about $\Sh(X,\Sp)$ in this section generalize immediately when $\Sp$ is replaced by an arbitrary dualizable $\infty$-category $\mathcal{C}$, in light of Lemma \ref{localizinginvariantreduction}, which states that $F^{\cont}( - \otimes \mathcal{C})$ is again a localizing invariant. Before we begin, we want to record the following simple lemma about pullback squares in stable categories. 

\begin{lemma} \label{pullbackdetectedbyfibers}
Let $\mathcal{E}$ be a stable $\infty$-category. Then a square
\[\begin{tikzcd}
	{A_0} & {A_1} \\
	{B_0} & {B_1}
	\arrow["{f_A}", from=1-1, to=1-2]
	\arrow[from=1-1, to=2-1]
	\arrow[from=1-2, to=2-2]
	\arrow["{f_B}", from=2-1, to=2-2]
\end{tikzcd}\]
is a pullback square iff the induced map on horizontal fibers $\mathrm{fib}(f_A) \rightarrow \mathrm{fib}(f_B)$ is an equivalence.
\end{lemma}

\begin{proof}
Consider the diagram
\[\begin{tikzcd}
	{\mathrm{fib}(f_A)} & {A_0} & {B_0} \\
	0 & {A_1} & {B_1.}
	\arrow[from=1-1, to=1-2]
	\arrow[from=1-1, to=2-1]
	\arrow["\lrcorner"{anchor=center, pos=0.125}, draw=none, from=1-1, to=2-2]
	\arrow[from=1-2, to=1-3]
	\arrow["{f_A}", from=1-2, to=2-2]
	\arrow["{f_B}", from=1-3, to=2-3]
	\arrow[from=2-1, to=2-2]
	\arrow[from=2-2, to=2-3]
\end{tikzcd}\]
If the right hand square is a pullback, then so is the total square, and hence $\mathrm{fib}(f_A) \rightarrow \mathrm{fib}(f_B)$ is an equivalence. Conversely, if the map $\mathrm{fib}(f_A) \rightarrow \mathrm{fib}(f_B)$  is an equivalence, we see that the left and total squares are pullbacks. Using stability of $\mathcal{E}$, both squares are also pushouts, and hence we can use the pasting lemma for pushouts, \cite[Lemma 4.4.2.1]{luriehtt}, to conclude that the right square is a pushout.
\end{proof}

\begin{proposition}[Perfect Nisnevich Descent] \label{propernisn} Let $f : Y \rightarrow  X$ be a perfect map of stably locally compact spaces, and $C \subset X$ a closed subspace. Suppose that $f$ induces a homeomorphism on the open complements of $f^{-1}(C)$ and $C$. Then the square
    \[\begin{tikzcd}
	{F^{\cont}(\Sh(X,\Sp))} & {F^{\cont}(\Sh(C,\Sp))} \\
	{F^{\cont}(\Sh(Y,\Sp))} & {F^{\cont}(\Sh(f^{-1}(C),\Sp))}
	\arrow[from=1-1, to=1-2]
	\arrow["{f^*}", from=1-1, to=2-1]
	\arrow["{f^*}", from=1-2, to=2-2]
	\arrow[from=2-1, to=2-2]
\end{tikzcd}\]
is a pullback.
\end{proposition}

\begin{proof}
Straightforward: Observe that the horizontal fibers of both rows are equivalent by Proposition \ref{opencloseddecomposition}, then apply Lemma \ref{pullbackdetectedbyfibers}.
\end{proof}

The main application is the simple case of a proper inclusion, which follows directly from the Second Isomorphism Theorem (Proposition \ref{secondiso}).

\begin{corollary} \label{closedexcision} 
Let $X$ be stably locally compact, $K \subset X$ be a perfect inclusion, and $C \subset X$ be a closed inclusion. Then we have a pullback square
\[\begin{tikzcd}
	{F^{\cont}(\Sh(K \cup C,\Sp))} & {F^{\cont}(\Sh(C,\Sp))} \\
	{F^{\cont}(\Sh(K,\Sp))} & {F^{\cont}(\Sh(K \cap C,\Sp)).}
	\arrow[from=1-1, to=1-2]
	\arrow[from=1-1, to=2-1]
	\arrow["\lrcorner"{anchor=center, pos=0.125}, draw=none, from=1-1, to=2-2]
	\arrow[from=1-2, to=2-2]
	\arrow[from=2-1, to=2-2]
\end{tikzcd}\]
\end{corollary}

Our goal is to generalize the statement of Corollary \ref{closedexcision} to the case where $C$ is allowed to be an arbitrary perfect subspace $K'$. In fact, if the formula of Theorem \ref{maintheorem} is true, it must necessarily be the case that this form of perfect descent holds. We will do this in steps. The first will be the use of Verdier duality. Note that the fact that $(-)^\vee = \mathrm{Fun}^L(-,\mathrm{Sp})$ is an autoequivalence on $\PR^L_{\dual}$ implies that it preserves Verdier sequences. This will be critical in the following observation.

\begin{proposition} \label{saturatedcompactdescent}
Let $X$ be stably compact, $K \subset X$ be a perfect inclusion, and $C \subset X$ be a saturated compact inclusion. Then we have a pullback square
\[\begin{tikzcd}
	{F^{\cont}(\Sh(K \cup C,\Sp))} & {F^{\cont}(\Sh(C,\Sp))} \\
	{F^{\cont}(\Sh(K,\Sp))} & {F^{\cont}(\Sh(K \cap C,\Sp)).}
	\arrow[from=1-1, to=1-2]
	\arrow[from=1-1, to=2-1]
	\arrow["\lrcorner"{anchor=center, pos=0.125}, draw=none, from=1-1, to=2-2]
	\arrow[from=1-2, to=2-2]
	\arrow[from=2-1, to=2-2]
\end{tikzcd}\]
\end{proposition}

\begin{proof} The proposition follows analogously to Corollary \ref{closedexcision} by applying Verdier duality (Theorem \ref{verdierduality}). The square of sheaf categories in question is equivalent to
\[\begin{tikzcd}
	{\Cosh( K^\vee \cup C^\vee; \Sp)} & {\Cosh( C^\vee; \Sp)} \\
	{\Cosh( K^\vee; \Sp)} & {\Cosh( K^\vee \cap C^\vee; \Sp)}
	\arrow[from=1-1, to=1-2]
	\arrow[from=1-1, to=2-1]
	\arrow[from=1-2, to=2-2]
	\arrow[""{name=0, anchor=center, inner sep=0}, from=2-1, to=2-2]
	\arrow[""{name=0p, anchor=center, inner sep=0}, phantom, from=2-1, to=2-2, start anchor=center, end anchor=center]
\end{tikzcd}\]
This square is obtained from the corresponding one involving sheaves of dual spaces by applying $(-)^\vee$. The saturated compact inclusion $C \rightarrow X$ is sent to a closed inclusion $C^\vee \rightarrow X^\vee$ under duality, and we can simply verify that horizontal fibers of the square of sheaf categories are equivalent by Proposition \ref{secondiso}.
\end{proof}

\begin{remark}
To be more concrete, the fiber of
\[ \Sh(K \cup C,\Sp) \rightarrow \Sh(C,\Sp) \]
is given by $\Cosh(U,\Sp)$, where $U$ is the open subset of $X^\vee$ corresponding to the Scott open filter of the compact saturated set $C$. Since $U$ is \emph{not} in general stably compact, but only stably locally compact, Verdier duality is a bit more subtle. Note that this is exactly the same issue as mentioned in Remark \ref{deGrootdualitysubtelty}.
\end{remark}

\begin{warning}
It is not true that for perfect embeddings $K,L$ of $X$ the square
\[\begin{tikzcd}
	{\Sh(K \cup L; \Sp)} & {\Sh(L; \Sp)} \\
	{\Sh(K; \Sp)} & {\Sh(K \cap L; \Sp)}
	\arrow[from=1-1, to=1-2]
	\arrow[from=1-1, to=2-1]
	\arrow[from=1-2, to=2-2]
	\arrow[from=2-1, to=2-2]
\end{tikzcd}\]
is necessarily a pullback square in $\PR^L_{\st}$, unlike in the case when $X$ is Hausdorff. As a counterexample, consider $K = [0,a]_{\leq}, L = [b,1]_{\leq}$ for $a < b$ in $I_{\leq}$. Their intersection is empty, however $\Sh( ([0,a] \cup [b,1])_{\leq}; \Sp)$ does not split as a product.
\end{warning}

To sum up, we have excision for localizing invariants for perfect embeddings if either of the embeddings is closed or saturated compact. We are left with showing that this implies proper descent. Before we do so, we will introduce a technical lemma to simplify induction of descent-type conditions.

\subsection{The sheaf condition on $D_3$}

We will need a technical lemma that allows us a reduction of a sheaf condition on three generators to a statement about pullbacks. The conceptual reason for this lemma is that the space $X = \mathbf{S}^3$ given by the product of three copies of the Sierpinski space represents the functor on locales that chooses three opens. It is somewhat tautological that the sheaf condition for the case of three opens on an arbitrary locale can always be checked by pushforward to $\mathbf{S}^3$. The space $\mathbf{S}^3$ is simply a combinatorial cube equipped with the Alexandroff topology, and its frame of opens is $D_3$, the free distributive lattice generated by three elements. 

\begin{lemma} \label{sheafconditionthree}
    Let $\mathcal{D}$ be an $\infty$-category with finite limits. Then
    \[\begin{tikzcd}
	A && {A_3} \\
	& {A_2} && {A_{23}} \\
	{A_1} && {A_{13}} \\
	& {A_{12}} && {A_{123}}
	\arrow[from=1-1, to=1-3]
	\arrow[from=1-1, to=2-2]
	\arrow[from=1-1, to=3-1]
	\arrow[from=1-3, to=2-4]
	\arrow[from=1-3, to=3-3]
	\arrow[from=2-2, to=2-4]
	\arrow[from=2-2, to=4-2]
	\arrow[from=2-4, to=4-4]
	\arrow[from=3-1, to=3-3]
	\arrow[from=3-1, to=4-2]
	\arrow[from=3-3, to=4-4]
	\arrow[from=4-2, to=4-4]
\end{tikzcd}\]
is a limit diagram, i.e.\ $A$ is the limit over the bottom part of the diagram, iff the square
\[\begin{tikzcd}
	A & {A_3} \\
	{A_1 \times_{A_{12}} A_2} & {A_{13} \times_{A_{123}} A_{23}}
	\arrow[from=1-1, to=1-2]
	\arrow[from=1-1, to=2-1]
	\arrow[from=1-2, to=2-2]
	\arrow[from=2-1, to=2-2]
\end{tikzcd}\]
is a pullback square.
\end{lemma}

\begin{remark}
Cubical diagrams as in Lemma \ref{sheafconditionthree} are referred to as \emph{cartesian cubes} in the context of Goodwillie calculus, see e.g. \cite{anel2018goodwillie}. We expect Lemma \ref{sheafconditionthree} to be well-known, but were not aware of a direct reference, hence included it. The obvious generalization to $n$-dimensional cubes is included in Appendix \ref{sheafarbitraryrank}.
\end{remark}

\begin{proof}
Let us define two diagram shapes, namely
\[ D_1 = \begin{tikzcd}
	&&& \bullet \\
	\bullet && \bullet && \bullet \\
	&&& \bullet \\
	&& \bullet && \bullet
	\arrow[from=1-4, to=2-5]
	\arrow[from=1-4, to=3-4]
	\arrow[from=2-1, to=3-4]
	\arrow[curve={height=12pt}, from=2-1, to=4-3]
	\arrow[from=2-3, to=2-5]
	\arrow[from=2-3, to=4-3]
	\arrow[from=2-5, to=4-5]
	\arrow[from=3-4, to=4-5]
	\arrow[from=4-3, to=4-5]
\end{tikzcd} \hspace{3ex}\text{ and }\hspace{4ex} D_2 = \begin{tikzcd}
	& \textcolor{rgb,255:red,92;green,92;blue,214}{\bullet} && \bullet \\
	\bullet && \bullet && \bullet \\
	& \textcolor{rgb,255:red,92;green,92;blue,214}{\bullet} && \bullet \\
	&& \bullet && \bullet.
	\arrow[color={rgb,255:red,92;green,92;blue,214}, from=1-2, to=1-4]
	\arrow[color={rgb,255:red,92;green,92;blue,214}, from=1-2, to=2-3]
	\arrow[color={rgb,255:red,92;green,92;blue,214}, from=1-2, to=3-2]
	\arrow[from=1-4, to=2-5]
	\arrow[from=1-4, to=3-4]
	\arrow[color={rgb,255:red,92;green,92;blue,214}, from=2-1, to=3-2]
	\arrow[from=2-1, to=3-4]
	\arrow[curve={height=12pt}, from=2-1, to=4-3]
	\arrow[from=2-3, to=2-5]
	\arrow[from=2-3, to=4-3]
	\arrow[from=2-5, to=4-5]
	\arrow[color={rgb,255:red,92;green,92;blue,214}, from=3-2, to=3-4]
	\arrow[color={rgb,255:red,92;green,92;blue,214}, from=3-2, to=4-3]
	\arrow[from=3-4, to=4-5]
	\arrow[from=4-3, to=4-5]
\end{tikzcd}\]
Both are given as realizations of posets. The poset of $D_1$ embeds fully faithfully into $D_2$ as the set of black vertices. Now, if we have the starting datum of the following diagram $A_\bullet$, indexed by $D_1$, in $\mathcal{D}$ as represented by
\[\begin{tikzcd}
	&&& {A_2} \\
	{A_3} && {A_1} && {A_{12}} \\
	&&& {A_{23}} \\
	&& {A_{13}} && {A_{123}}
	\arrow[from=1-4, to=2-5]
	\arrow[from=1-4, to=3-4]
	\arrow[from=2-1, to=3-4]
	\arrow[curve={height=12pt}, from=2-1, to=4-3]
	\arrow[from=2-3, to=2-5]
	\arrow[from=2-3, to=4-3]
	\arrow[from=2-5, to=4-5]
	\arrow[from=3-4, to=4-5]
	\arrow[from=4-3, to=4-5]
\end{tikzcd}\]
then the right Kan extension to $D_2$ is given by
\[\begin{tikzcd}
	& \textcolor{rgb,255:red,92;green,92;blue,214}{{A_1 \times_{ A_{12}} A_2}} && {A_{2}} \\
	{A_3} && {A_{1}} && {A_{12}} \\
	& \textcolor{rgb,255:red,92;green,92;blue,214}{{A_{13} \times_{A_{123}} A_{23}}} && {A_{23}} \\
	&& {A_{13}} && {A_{123}}
	\arrow[color={rgb,255:red,92;green,92;blue,214}, from=1-2, to=1-4]
	\arrow[color={rgb,255:red,92;green,92;blue,214}, from=1-2, to=2-3]
	\arrow[color={rgb,255:red,92;green,92;blue,214}, from=1-2, to=3-2]
	\arrow[from=1-4, to=2-5]
	\arrow[from=1-4, to=3-4]
	\arrow[color={rgb,255:red,92;green,92;blue,214}, from=2-1, to=3-2]
	\arrow[from=2-1, to=3-4]
	\arrow[curve={height=12pt}, from=2-1, to=4-3]
	\arrow[from=2-3, to=2-5]
	\arrow[from=2-3, to=4-3]
	\arrow[from=2-5, to=4-5]
	\arrow[color={rgb,255:red,92;green,92;blue,214}, from=3-2, to=3-4]
	\arrow[color={rgb,255:red,92;green,92;blue,214}, from=3-2, to=4-3]
	\arrow[from=3-4, to=4-5]
	\arrow[from=4-3, to=4-5]
\end{tikzcd}\]
by the general formula for right Kan extension. The limits over both diagrams agree. Finally, observe that the inclusion of the spine $S$
\[\begin{tikzcd}
	& \textcolor{rgb,255:red,92;green,92;blue,214}{\bullet} \\
	\bullet \\
	& \textcolor{rgb,255:red,92;green,92;blue,214}{\bullet}
	\arrow[color={rgb,255:red,92;green,92;blue,214}, from=1-2, to=3-2]
	\arrow[color={rgb,255:red,92;green,92;blue,214}, from=2-1, to=3-2]
\end{tikzcd}\]
is final in $D_2$, hence we obtain the formula for the limit $\lim_{D_1} A_\bullet$ as the pullback
\[\begin{tikzcd}
	{\lim_{F} A_\bullet} & \textcolor{rgb,255:red,92;green,92;blue,214}{{A_1 \times_{ A_{12}} A_2}} \\
	{A_3} & \textcolor{rgb,255:red,92;green,92;blue,214}{{A_{13} \times_{A_{123}} A_{23}}.}
	\arrow[from=1-1, to=1-2]
	\arrow[from=1-1, to=2-1]
	\arrow["\lrcorner"{anchor=center, pos=0.125}, draw=none, from=1-1, to=2-2]
	\arrow[color={rgb,255:red,92;green,92;blue,214}, from=1-2, to=2-2]
	\arrow[color={rgb,255:red,92;green,92;blue,214}, from=2-1, to=2-2]
\end{tikzcd}\]
\end{proof}

\subsection{Extending excisiveness}

With Lemma \ref{sheafconditionthree} under our belt, we can extend from closed and saturated compact descent to descent for elementary compact sets. 

\begin{proposition}[Elementary compact descent] \label{elementaryexcision}
Suppose $F : \mathrm{SC}^{\op} \rightarrow \mathcal{D}$ is a functor, with $\mathcal{D}$ an $\infty$-category with finite limits, such that for all stably compact $X$, $K$ a compact subset of $X$, $C$ a closed subset of $X$, and $S$ a saturated compact subset of $X$ the squares
    \[\begin{tikzcd}
	{F(K \cup C)} & {F(C)} & {F(K \cup S)} & {F(S)} \\
	{F(K)} & {F(K \cap C)} & {F(K)} & {F(K \cap S).}
	\arrow[from=1-1, to=1-2]
	\arrow[from=1-1, to=2-1]
	\arrow["\lrcorner"{anchor=center, pos=0.125}, draw=none, from=1-1, to=2-2]
	\arrow[from=1-2, to=2-2]
	\arrow[from=1-3, to=1-4]
	\arrow[from=1-3, to=2-3]
	\arrow["\lrcorner"{anchor=center, pos=0.125}, draw=none, from=1-3, to=2-4]
	\arrow[from=1-4, to=2-4]
	\arrow[from=2-1, to=2-2]
	\arrow[from=2-3, to=2-4]
\end{tikzcd}\]
are pullback squares. Then it also holds that for any compact subset $E \subset X$ of the form $E = S \cup C$, with $S$ being saturated compact and $C$ being closed, that the square
\[\begin{tikzcd}
	{F(K \cup E)} & {F(K)} \\
	{F(E)} & {F(K \cap E)}
	\arrow[from=1-1, to=1-2]
	\arrow[from=1-1, to=2-1]
	\arrow["\lrcorner"{anchor=center, pos=0.125}, draw=none, from=1-1, to=2-2]
	\arrow[from=1-2, to=2-2]
	\arrow[from=2-1, to=2-2]
\end{tikzcd}\]
is a pullback square.
\end{proposition}

\begin{proof}
The proof is just an iterated application of Lemma \ref{sheafconditionthree}. Observe that by assumption we have the pullbacks
\[\begin{tikzcd}
	{F(S \cup C)} & {F(C)} & {F((K \cap S) \cup (K \cap C))} & {F(K \cap C)} \\
	{F(S)} & {F( S \cap C)} & {F( K \cap S)} & {F(K \cap S \cap C)}
	\arrow[from=1-1, to=1-2]
	\arrow[from=1-1, to=2-1]
	\arrow["\lrcorner"{anchor=center, pos=0.125}, draw=none, from=1-1, to=2-2]
	\arrow[from=1-2, to=2-2]
	\arrow[from=1-3, to=1-4]
	\arrow[from=1-3, to=2-3]
	\arrow["\lrcorner"{anchor=center, pos=0.125}, draw=none, from=1-3, to=2-4]
	\arrow[from=1-4, to=2-4]
	\arrow[from=2-1, to=2-2]
	\arrow[from=2-3, to=2-4]
\end{tikzcd}\]
since $S \subset S \cup C$ and $K \cap S \subset (K \cap S) \cup (K \cap C) = K \cap ( S \cup C )$ are saturated compact subsets. Applying Lemma \ref{sheafconditionthree}, this means that the square
\[\begin{tikzcd}
	{F(K \cup S \cup C)} & {F(K)} \\
	{F(S \cup C )} & {F(K \cap (S \cup C))}
	\arrow[from=1-1, to=1-2]
	\arrow[from=1-1, to=2-1]
	\arrow[from=1-2, to=2-2]
	\arrow[from=2-1, to=2-2]
\end{tikzcd}\]
is a pullback square iff $F(K \cup S \cup C)$ is the limit of the corresponding cube built from $F(K), F(S)$ and $F(C)$, which, again by Lemma \ref{sheafconditionthree}, is the case if the squares
\[\begin{tikzcd}
	{F(K \cup C \cup S )} & {F(S)} & {F(K \cup C) } & {F(C)} & {} \\
	{F(K \cup C)} & {F( (K \cap S) \cup (C \cap S))} & {F( K)} & {F( K \cap C)}
	\arrow[from=1-1, to=1-2]
	\arrow[from=1-1, to=2-1]
	\arrow["\lrcorner"{anchor=center, pos=0.125}, draw=none, from=1-1, to=2-2]
	\arrow[from=1-2, to=2-2]
	\arrow[from=1-3, to=1-4]
	\arrow[from=1-3, to=2-3]
	\arrow["\lrcorner"{anchor=center, pos=0.125}, draw=none, from=1-3, to=2-4]
	\arrow[from=1-4, to=2-4]
	\arrow[from=2-1, to=2-2]
	\arrow[from=2-3, to=2-4]
\end{tikzcd}\]
and
\[\begin{tikzcd}
	{F((K \cap S) \cup (C \cap S))} & {F(C \cap S)} \\
	{F(K \cap S)} & {F( K \cap C \cap S)}
	\arrow[from=1-1, to=1-2]
	\arrow[from=1-1, to=2-1]
	\arrow["\lrcorner"{anchor=center, pos=0.125}, draw=none, from=1-1, to=2-2]
	\arrow[from=1-2, to=2-2]
	\arrow[from=2-1, to=2-2]
\end{tikzcd}\]
are pullbacks. This is the case for all three squares since:
\begin{itemize}
\item $S \subset K \cup C \cup S$ is saturated compact.
\item $C \subset K \cup C$ is closed.
\item $C \cap S \subset (K \cap S) \cup (C \cap S) =  (K \cup C) \cap S$ is closed.
\end{itemize}
Therefore in all of these cases, we can use excision for either closed or saturated compact inclusions, hence we conclude the claim.
\end{proof}

Now assume we have a functor $F : \mathrm{SC}^{\op} \rightarrow \mathcal{D}$ as in Proposition \ref{elementaryexcision}, and that $\mathcal{D}$ is furthermore stable. Since crucially in this context pushout squares agree with pullback squares, we see that $F$ restricted to elementary compact sets satisfies a \emph{cosheaf} condition on elementary compacts (with covers given by intersections!). We claim that this, together with compatibility with filtered intersections, is enough to produce a cosheaf on the frame $\mathcal{O}(X^{\patch}) \cong \mathcal{K}(X)^{\op}$ of all perfect subspaces. The only caveat is that $F(1) \neq 0$ in general, an issue that can be solved by taking cofibers.

\begin{proposition} \label{properexcision}
    Let $X$ be a stably compact space, let $\mathcal{D}$ be stable and closed under colimits, and let $F : \mathcal{O}(X^{\patch}) \cong \mathcal{K}(X)^{\op} \rightarrow \mathcal{D}$ be a functor satisfying the two conditions: 
    \begin{enumerate}
        \item For all $K \subset X$ compact and $E \subset X$ elementary compact we have a pushout square \[\begin{tikzcd}
	{F(K \cup E)} & {F(K)} \\
	{F(E)} & {F(E \cap K).}
	\arrow[from=1-1, to=1-2]
	\arrow[from=1-1, to=2-1]
	\arrow[from=1-2, to=2-2]
	\arrow[from=2-1, to=2-2]
	\arrow["\lrcorner"{anchor=center, pos=0.125, rotate=180}, draw=none, from=2-2, to=1-1]
\end{tikzcd}\]
        \item Whenever $K = \bigcap_{i \in I} K_i$ is a directed intersection of compacts, we have
        \[ F( K ) \simeq \colim_{i \in I} F(K_i).\]
\end{enumerate}
Then the functor $F^\#$ defined via $F^\#(K) = \mathrm{cofib}( F(1) \rightarrow F(K) ) $ is a cosheaf on $X^{\patch}$.
\end{proposition}

Before we begin the proof, let us cite a standard lemma on the use of Grothendieck topologies.

\begin{lemma}[Comparison lemma, see \cite{hoyois_2014}, Lemma C.3] \label{comparisonlemma} Let $(P, \tau)$ be a poset equipped with a Grothendieck topology and $P_0 \subset P$ a subset of $P$ such that:
\begin{itemize}
\item Every object in $P$ can be covered by objects in $P_0$.
\item $P_0$ is closed under meets in $P$.
\end{itemize}
Let $\tau_0$ be the induced Grothendieck topology on $P_0$ by restriction. Then the restriction
$$ \Sh(P, \tau) \rightarrow \Sh(P_0, \tau_0) $$
is an equivalence of $\infty$-categories, with inverse given by right Kan extension.
\end{lemma}

Let us remark that the same statement thus follows for cosheaves, however with inverse given by left Kan extension. 

\begin{proof}[Proof of Proposition \ref{properexcision}] It is immediate that $F^\#$ restricted to the set of elementary compacts $\mathcal{E}(X)^{\op}$ is a cosheaf for the induced Grothendieck topology given by restriction from $\mathcal{O}(X^{\patch})$. Using Lemma \ref{comparisonlemma}, we are left to show that the natural map
$$ \eta : \mathrm{Lan} \mathrm{Res} F^\# \rightarrow  F^\#$$
is an equivalence, which is a statement we only need to verify objectwise for all compact $K \subset X$. We proceed in two steps:
\begin{enumerate}[label=(\alph*)]
    \item We first show that $\eta_K$ is an equivalence for $K$ of the form $K = E_1 \cap \cdots \cap E_n$ with $E_i$ elementary compact, by induction. The induction start is true since the inclusion of elementary compacts into all compacts is fully faithful. Now assume we know that $ \eta $ is an equivalence for intersections of $n-1$ many elementary compacts. Since the left Kan extension of $\mathrm{Res} F^\#$ is a cosheaf on  $\mathcal{O}(X^{\patch})$, the value of $\mathrm{Lan} \mathrm{Res} F^\#$ on $K$ is given as the pushout
    \[\begin{tikzcd}
	{F^\#( (E_1 \cap \cdots  \cap E_{n-1}) \cup E_n) } & {F^\#( E_n )} \\
	{F^\#( E_1 \cap \cdots  \cap E_{n-1}).} & {\mathrm{Lan} \mathrm{Res} F^\#( E_1 \cap \cdots  \cap E_{n-1} \cap E_n)}
	\arrow[from=1-1, to=1-2]
	\arrow[from=1-1, to=2-1]
	\arrow[from=1-2, to=2-2]
	\arrow[from=2-1, to=2-2]
	\arrow["\lrcorner"{anchor=center, pos=0.125, rotate=180}, draw=none, from=2-2, to=1-1]
\end{tikzcd}\]
    But this pushout agrees with $F^\#(E_1 \cap \cdots  \cap E_{n-1} \cap E_n)$ by assumption from point (1).
    \item For arbitrary $K$, we simply write $K$ as a directed intersection of finite intersections of elementary compacts, and see from using (a) together with assumption (2) that $\eta_K$ is also an equivalence.
\end{enumerate}
\end{proof}

\section{Descent on stably compact spaces} \label{descent}

In this section we present an extension of the following theorem due to Clausen.

\begin{theorem}[Clausen, see \cite{krause_nikolaus_puetzstueck} Theorem 3.6.15] \label{closedandprofinite} Let $\mathcal{D}$ be a compactly assembled presentable $\infty$-category and let $F : \mathrm{CH}^{\op} \rightarrow \mathcal{D}$ be a functor such that:
\begin{enumerate}
\item $F(\emptyset) = 1.$
\item \emph{Closed descent:} Whenever $K, L \subset X$ are two closed embeddings, then \[\begin{tikzcd}
	{F(K \cup L)} & {F(K)} \\
	{F(L)} & {F(K \cap L)}
	\arrow[from=1-1, to=1-2]
	\arrow[from=1-1, to=2-1]
	\arrow["\lrcorner"{anchor=center, pos=0.125}, draw=none, from=1-1, to=2-2]
	\arrow[from=1-2, to=2-2]
	\arrow[from=2-1, to=2-2]
\end{tikzcd}\]
is a pullback.
\item \emph{Profinite descent:} Whenever $X_i, i \in I,$ is a cofiltered system in $\mathrm{CH}$, then
$$F( \lim_{i \in I} X_i ) \simeq \colim_{i \in I} F(X_i).$$
\end{enumerate}
Then there exists a natural equivalence of functors
$$ F \simeq \Gamma(-, F( \mathrm{pt} ) ),$$
where for a given compact Hausdorff space $X$, the expression $\Gamma(X, F( \mathrm{pt} ) )$ refers to global sections of the constant sheaf with value $F( \mathrm{pt} ) \in  \mathcal{D}$.
\end{theorem}

This extension for stably compact spaces is obtained by switching closed descent with proper descent.

\begin{theorem} \label{properandcofilteredexcision} Let $\mathcal{D}$ be a compactly assembled $\infty$-category, and let $F : \mathrm{SC}^{\op} \rightarrow \mathcal{D}$ be a functor such that:
\begin{enumerate}
    \item $F(\emptyset) = 1.$
    \item \emph{Perfect descent:} Whenever $K, L \subset X$ are two perfect embeddings, then \[\begin{tikzcd}
	{F(K \cup L)} & {F(K)} \\
	{F(L)} & {F(K \cap L)}
	\arrow[from=1-1, to=1-2]
	\arrow[from=1-1, to=2-1]
	\arrow["\lrcorner"{anchor=center, pos=0.125}, draw=none, from=1-1, to=2-2]
	\arrow[from=1-2, to=2-2]
	\arrow[from=2-1, to=2-2]
\end{tikzcd}\]
is a pullback.
\item \emph{Cofiltered descent:} Whenever $X_i, i \in I,$ is a cofiltered system in $\mathrm{SC},$ then
$$F( \lim_{i \in I} X_i ) \simeq \colim_{i \in I} F(X_i).$$
\end{enumerate}
Then there is a natural equivalence of functors
$$F \simeq \Gamma( (-)^{\patch}; F(\mathrm{pt}) ).$$
\end{theorem}

Let us provide the setup for the proof of Theorem \ref{properandcofilteredexcision}. For the remainder of the section, fix a functor $F : \mathrm{SC}^{\op} \rightarrow \mathcal{D}$ satisfying the conditions of Theorem \ref{properandcofilteredexcision}.  For a given stably compact space $X$, restricting $F$ to the set of perfect embeddings $K \subset X$ produces a $K$-sheaf
$$ F_X : \mathcal{K}(X)^{\op} \rightarrow \mathcal{D}.$$
on $X^{\patch}$, since the set of perfect subsets of $X$ agrees with the set of compact subsets of $X^{\patch}$. Under the equivalence of $K$-sheaves with sheaves, $F_X$ corresponds to a sheaf $F^X \in \Sh(X^{\patch},\mathcal{D})$. We observe the following, using the facts given in Observation \ref{valuesofksheaf}.
\begin{itemize}
\item We have equivalences $\Gamma( X^{\patch}, F^X ) = F^X(X) \simeq F_X(X) = F(X)$ by construction. 
\item For each point $x \in X$, the stalks compute as $ F^X_x \simeq F(\{x\}) \simeq F( \mathrm{pt} )$.
\item The unique map $X \rightarrow \mathrm{pt}$ produces a canonical map $F(\mathrm{pt}) \rightarrow F(X)$, which using the adjunction gives a natural transformation
$$ \alpha_X : \underline{F(\mathrm{pt})} \rightarrow F^X $$
of sheaves in $\Sh(X^{\patch}, \mathcal{D})$.
\item The natural transformation $\alpha_X$ is a stalkwise equivalence for all points $x \in X$.
\end{itemize}
Therefore, we arrive at the following simple conclusion.

\begin{observation}
Let $X$ be a stably compact space. If $\alpha_X$ is an equivalence, then the statement that $F(X) \simeq \Gamma( X^{\patch}; F(\mathrm{pt}) )$ holds.
\end{observation}

Recall that a topos $\mathcal{X}$ is said to have \emph{enough points}, if the family of stalk functors $x^* : \mathcal{X} \rightarrow \An$ for all points $x$ of $\mathcal{X}$ is jointly conservative, i.e.\ if equivalences can be detected stalkwise. To have the same statement when considering as target the compactly assembled $\infty$-category $\mathcal{D}$, we cite the following lemma.

\begin{lemma}[\cite{haine2022nonabelianbasechangebasechangecoefficients}, Lemma 2.12]
Let $\{p^*_i : T \rightarrow S_i \}_{i \in I}$ be a jointly conservative family of finite limit preserving left adjoint functors between presentable $\infty$-categories, and let $\mathcal{E}$ be a compactly assembled presentable $\infty$-category. Then the family of left adjoints $\{p^*_i \otimes \mathcal{E} : T \otimes \mathcal{E} \rightarrow S_i \otimes \mathcal{E}\}_{i \in I}$ is jointly conservative.
\end{lemma}

The following is immediate.

\begin{lemma} \label{maintheoremforhypercomplete}
The statement $F(X) \simeq \Gamma( X^{\patch}; F(\mathrm{pt}) )$ holds for all stably compact spaces $X$ such that $\Sh(X^{\patch},\An)$ has enough points.
\end{lemma}

We note that the condition that $\Sh(X^{\patch},\An)$ has enough points fails for example for $X = I^S_\leq$ a directed Hilbert cube of infinite dimension $S$. Nonetheless, we now have a large class of examples for which Theorem \ref{properandcofilteredexcision} holds, due to the following two theorems.

\begin{theorem}[\cite{luriehtt} Corollary  7.2.1.17] Let $X$ be a topological space. Suppose that $\Sh(X, \An)$ is locally of homotopy dimension $\leq n$ for some integer $n$. Then $\Sh(X, \An)$ has enough points.
\end{theorem}

\begin{theorem}[\cite{luriehtt} Corollary 7.2.3.7.] Let $X$ be a paracompact topological space. The following conditions are equivalent.
\begin{enumerate}
\item  The covering dimension of $X$ is $\leq n$.
\item  The homotopy dimension of $\Sh(X, \An)$ is $\leq n$.
\item  For every closed subset $A \subset X$, every $m \geq n$, and every continuous map $f_0 : A \rightarrow S^m$, there exists $f : X \rightarrow S^m$ extending $f_0$.
\end{enumerate}
\end{theorem}

It follows that Theorem \ref{properandcofilteredexcision} holds for stably compact $X$ such that $X^{\patch}$ embeds as a closed subset of $I^N$ for some finite $N$, as $I^N$ is paracompact and has covering dimension $N$. At this point it might seem that we are stuck, since we cannot argue via stalks for arbitrary stably compact $X$. In the case of compact Hausdorff spaces, the proof argument given by Clausen proceeds by using the fact that every compact Hausdorff space is obtained as an inverse limit of finite-dimensional polyhedra, and then applying profinite descent. Our own proof will slightly diverge here from this argument, as we will instead argue directly that $\alpha_X$ is an equivalence for every stably compact space $X$. Before we do so, we should add some comments on the naturality of the assignment $X \mapsto \alpha_X$.

\begin{construction} \label{constructionofalpha}
Given a perfect map $f : X \rightarrow Y$, we see that there is a natural transformation of $K$-sheaves on $Y$,
$$ F_Y \rightarrow f_+ F_X $$
given for $K \subset Y$ perfect as the map $F(K) \xrightarrow{ F(f|_{f^{-1}(K)}) } F( f^{-1}(K) )$. Under the naturality provided by Proposition \ref{naturality}, this corresponds to a natural transformation of sheaves $ F^Y \rightarrow f_* F^X.$ The extreme case, $X \rightarrow \mathrm{pt}$, gives the map $F(\mathrm{pt}) \rightarrow X_* F^X$ that is adjoint to $\alpha_X$.

If $X \xrightarrow{f} Y \xrightarrow{g} Z$ are two perfect maps, it is clear that we have a commutative triangle of natural transformations of sheaves over $Z$,
\[\begin{tikzcd}
	{F^Z} & {g_*F^Y} \\
	& {(gf)_*F^X.}
	\arrow[from=1-1, to=1-2]
	\arrow[from=1-1, to=2-2]
	\arrow[from=1-2, to=2-2]
\end{tikzcd}\]
Adjoining $g_*$ over gives the commutative triangle
\[\begin{tikzcd}
	{g^*F^Z} & {F^Y} \\
	{} & {f_*F^X.}
	\arrow[from=1-1, to=1-2]
	\arrow[from=1-1, to=2-2]
	\arrow[from=1-2, to=2-2]
\end{tikzcd}\]
Moreover, the map $g^* F^Z \rightarrow f_*F^X$ naturally factors as the commutative triangle
\[\begin{tikzcd}
	{g^*F^Z} \\
	{f_*f^* g^* F^Z} & {f_*F^X.}
	\arrow[from=1-1, to=2-1]
	\arrow[from=1-1, to=2-2]
	\arrow[from=2-1, to=2-2]
\end{tikzcd}\]
By pasting triangles, we arrive at the natural commutative square of sheaves on $Y$,
\[\begin{tikzcd}
	{g^*F^Z} & {F^Y} \\
	{f_*f^*g^* F^X} & {f_*F^X.}
	\arrow[from=1-1, to=1-2]
	\arrow[from=1-1, to=2-1]
	\arrow[from=1-2, to=2-2]
	\arrow[from=2-1, to=2-2]
\end{tikzcd}\]
The special case of $Z = \mathrm{pt}$ gives a map of arrows from $\alpha_Y \rightarrow f_* \alpha_X$.
\end{construction}

With this understood, we can get rid of the finite dimensionality assumption by the observation that an arbitrary directed Hilbert cube $I^S_\leq \cong \mathrm{lim}_{ N \subset S \text{ finite}} I^N_\leq$ is obtained as an inverse limit of stably compact spaces with finite dimensional patch topology. This will follow from the (strong) homotopy invariance provided by the following lemma.

\begin{lemma} \label{homotopyinvariance}
Let $X$ be stably compact such that $X^{\patch}$ embeds as a closed subset of $[0,1]^M$ for some finite $M$, and let $S$ be a set of arbitrary cardinality. Then
$$ F( X ) \simeq  F(X \times I^S_\leq ) $$
via the canonical projection $X \times I^S_\leq \rightarrow X$.
\end{lemma}

\begin{proof}
Using cofiltered descent, we see that
$$ F(X \times I^S_\leq  ) \simeq \colim_{ N \subset S \text{ finite}} F(X \times I^N_\leq  ). $$
The spaces $X \times I^N_\leq$ also satisfy that $(X \times I^N_\leq)^{\patch}  \cong X^{\patch} \times I^N \subset I^{M \amalg N}$ is paracompact with finite covering dimension, hence Lemma \ref{maintheoremforhypercomplete} together with contractibility of $I^N$ shows that the colimit diagram is in fact the constant diagram with value $F(X)$.
\end{proof}

As a direct application we see the following.

\begin{lemma} \label{homotopyinvariancehilbert} Let $N \subset S$ be a finite subset, and denote by $p^N : I^S_\leq \rightarrow I^N_\leq$ the associated projection between directed Hilbert cubes. Then there exists a natural equivalence of sheaves
$$ F^{I^N_\leq} \simeq p^N_* F^{I^S_\leq} $$
on $I^N$.
\end{lemma}

\begin{proof} The corresponding $K$-sheaf of $F^{I^N_\leq}$ is given by $F_{I^N_\leq}$, and the corresponding $K$-sheaf of $p^N_* F^{I^S_\leq}$ is given by
$$ (K \subset I^N_\leq) \mapsto F_{I^S_\leq}( K \times I^{S \setminus N}_\rightarrow) = F( K \times I^{S \setminus N}_\rightarrow). $$
We can see that their values are naturally equivalent, as by Lemma \ref{homotopyinvariance}, we have $F( K \times I^{S \setminus N}_\rightarrow) \simeq F( K ) \simeq F_X(K)$.
\end{proof}

\begin{proposition} \label{alphahilbertcube}
The natural transformation $\alpha_X$ is an equivalence for $X = I^S_\leq$ a directed Hilbert cube for an arbitrary set $S$.
\end{proposition}

\begin{proof}
Since $ I^S_\leq = \lim_{ N \subset S \text{ finite}} I^N_\leq$, taking patch topology and using the fact that taking sheaves with values in $\mathcal{D}$ sends cofiltered limits of stably compact spaces to filtered colimits in $\PR^L$, we have that
$$ \Sh( I^S , \mathcal{D} ) \simeq  \colim_{ N \subset S \text{ finite}} \Sh( I^N , \mathcal{D} ), $$
with the colimit taken in  $\PR^L$, i.e.\ as a limit of $\infty$-categories along right adjoints. Thus, the map
$$\alpha_{I^S_\leq} : \underline{F(\mathrm{pt})} \rightarrow F^{I^S_\leq}$$ corresponds to the family of maps $p^N_*(\alpha_{I^S_\leq})$ for $N \subset S$ finite. We claim that $p^N_*(\alpha_{I^S_\leq})$ is equivalent as a natural transformation to $ \alpha_{I^N_\leq} $. We have seen that $ F^{I^N_\leq} \simeq p^N_* F^{I^S_\leq} $ by Lemma \ref{homotopyinvariancehilbert}. We have also seen that the map $p^N$ induces a contractible and essential geometric morphism after applying the patch topology functor (see Example \ref{projectionhilbertcube}), hence $p^N_*$ preserves constant objects by Lemma \ref{preservationofconstantobjects}.
\end{proof}

\begin{lemma} \label{inheritanceembedding}
If $i : K \hookrightarrow X$ is a perfect embedding of stably compact spaces, and $\alpha_X$ is an equivalence, then $\alpha_K$ is an equivalence as well.
\end{lemma}

\begin{proof}
The map of natural transformations $\alpha_X \rightarrow i_* \alpha_K$ produces by adjunction a map $ i^* \alpha_X \rightarrow \alpha_K$. We claim that this map is an equivalence of natural transformations, which is equivalent to saying that the natural transformations $i^* X^*( F(\mathrm{pt}) ) \rightarrow Y^*( F(\mathrm{pt}) )$ and $ i^*( F^X ) \rightarrow  F^K$ are equivalences. The former is clear, as $i^*$ preserves constant sheaves. The latter follows from the observation that the corresponding $K$-sheaf of $i^*( F^X )$ is obtained by restricting $F_X$ to $\mathcal{K}(K)$. This restriction is given by precomposing with the inclusion
$$ \mathcal{K}(K) \cong \mathcal{K}(X)_{/K}  \hookrightarrow \mathcal{K}(X)$$
which preserves non-empty infima and arbitrary suprema, hence precomposition preserves $K$-sheaves.
\end{proof}

We can now finish the proof of Theorem \ref{properandcofilteredexcision}.

\begin{proof}[Proof of Theorem \ref{properandcofilteredexcision}.]
If $X$ is a stably compact space, embed $X$ into a directed Hilbert cube, using Theorem \ref{urysohnembedding}. The combination of Proposition \ref{alphahilbertcube} and Lemma \ref{inheritanceembedding} implies that $\alpha_X$ is an equivalence.
\end{proof}

\begin{remark}
The only part where an argument via stalks was used during this proof was to show that $\alpha_{X}$ is an equivalence for $X = I_\leq^N$ a directed, finite-dimensional cube. Joint surjectivity of the stalk functors in this case unfortunately fails if $\mathcal{D}$ is not assumed to be compactly assembled, see the discussion in \cite{krause_nikolaus_puetzstueck}, Section 3.6, in particular Corollary 3.6.16. 
\end{remark}

\begin{remark}
We have been somewhat imprecise in the actual construction of a natural transformation $\Gamma((-)^{\patch}, F(\mathrm{pt})) \rightarrow F$. Let us address potential concerns here.\footnote{We thank Maxime Ramzi for suggesting this argument.}

Consider the functor $\mathrm{Sh}((-)^{\patch},\mathcal{D}) : \mathrm{SC} \rightarrow \widehat{\mathrm{Cat}}_\infty$, with functoriality given by $f : X \rightarrow Y$ being mapped to $f_* : \mathrm{Sh}(X^{\patch},\mathcal{D}) \rightarrow \mathrm{Sh}(Y^{\patch},\mathcal{D})$. The inclusion of the point $\mathrm{pt} \rightarrow \mathrm{SC}$ is terminal, hence there is an induced natural transformation $  \mathrm{Sh}((-)^{\patch},\mathcal{D}) \Rightarrow \underline{\mathcal{D}} $, where $\underline{\mathcal{D}} : \mathrm{SC}^{\op} \rightarrow \widehat{\mathrm{Cat}}_\infty$ is the constant functor with value $\mathcal{D}$. Valuewise it is just given by forming global sections.

Upon applying unstraightening, we get an induced functor between cartesian fibrations
\[\begin{tikzcd}
	{\int_{\mathrm{SC}^{\op}} \mathrm{Sh}} && {\int_{\mathrm{SC}^{\op}} \underline{D} \simeq \mathrm{SC}^{\op} \times \mathcal{D}} \\
	& {\mathrm{SC}^{\op}}
	\arrow["\Gamma", from=1-1, to=1-3]
	\arrow["p"', from=1-1, to=2-2]
	\arrow["q", from=1-3, to=2-2]
\end{tikzcd}\]
where the $\infty$-category $\int_{\mathrm{SC}^{\op}} \mathrm{Sh}$ is given on objects as pairs $(X \in \mathrm{SC}, \mathcal{F} \in \mathrm{Sh}(X^{\patch},\mathcal{D}))$, and with morphisms $(Y,\mathcal{G}) \rightarrow (X, \mathcal{F})$ given by a pair of a perfect map $f : X \rightarrow Y$ and a natural transformation $\mathcal{G} \Rightarrow f_*( \mathcal{F} )$ in $\mathrm{Sh}(Y^{\patch},\mathcal{D})$. The functor $\Gamma$ is given on fibers over a given stably compact space $X$ as $X_* : \mathrm{Sh}(X^{\patch},\mathcal{D}) \rightarrow \mathcal{D}$.

The functor $\Gamma$ has a left adjoint relative $\mathrm{SC}^{\op}$, as can be checked directly via \cite[Proposition 7.3.2.6]{lurieha}, with the labels $$\mathcal{E} = \mathrm{SC}^{\op}, \mathcal{C} = \mathrm{SC}^{\op} \times \mathcal{D}, \mathcal{D} = \int_{\mathrm{SC}^{\op}} \mathrm{Sh} \text{ and } G = \Gamma.$$

Upon taking sections, this leads to the adjunction
\[\begin{tikzcd}
	{\mathrm{Fun}_{/\mathrm{SC}^{\op}}( \mathrm{SC}^{\op}, \int_{\mathrm{SC}^{\op}} \mathrm{Sh})} & {\mathrm{Fun}_{/\mathrm{SC}^{\op}}(\mathrm{SC}^{\op} , \mathrm{SC}^{\op} \times \mathcal{D}) } & {\mathrm{Fun}(\mathrm{SC}^{\op}, \mathcal{D})}
	\arrow[""{name=0, anchor=center, inner sep=0}, "\Gamma"', curve={height=18pt}, from=1-1, to=1-2]
	\arrow[""{name=1, anchor=center, inner sep=0}, curve={height=18pt}, from=1-2, to=1-1]
	\arrow["{\simeq }"{description}, draw=none, from=1-2, to=1-3]
	\arrow["\dashv"{anchor=center, rotate=-90}, shift right, draw=none, from=1, to=0]
\end{tikzcd}\]
The point of Construction \ref{constructionofalpha} is that it provides a lift $F^{-}$ of $F : \mathrm{SC}^{\op} \rightarrow \mathcal{D}$ along $\Gamma$. The counit of the above adjunction applied to $F^{-}$ recovers valuewise for each stably compact space $X$ the natural transformation $\alpha_X : \underline{F(\mathrm{pt})} \rightarrow F^X$. When applying $\Gamma$, this provides the natural transformation $ \Gamma( (-)^{\patch}, F(\mathrm{pt}) ) \Rightarrow F$.
\end{remark}

\begin{remark} Theorem \ref{closedandprofinite} and Theorem \ref{properandcofilteredexcision} can be phrased more structurally as stating that there are equivalences of $\infty$-categories
\[\begin{tikzcd}
	{\mathrm{Fun}^{desc}(\mathrm{SC}^{\op}, \mathcal{D})} & {\mathrm{Fun}^{desc}(\mathrm{CH}^{\op}, \mathcal{D})} & {\mathcal{D}}
	\arrow[""{name=0, anchor=center, inner sep=0}, "res", curve={height=-12pt}, from=1-1, to=1-2]
	\arrow[""{name=1, anchor=center, inner sep=0}, "{{\patch}^*}", curve={height=-12pt}, from=1-2, to=1-1]
	\arrow[""{name=2, anchor=center, inner sep=0}, curve={height=12pt}, from=1-2, to=1-3]
	\arrow[""{name=3, anchor=center, inner sep=0}, "{\Gamma(-;-)}"', curve={height=12pt}, from=1-3, to=1-2]
	\arrow["\sim"{description}, draw=none, from=0, to=1]
	\arrow["\sim"{description}, shift left=5, draw=none, from=3, to=2]
\end{tikzcd}\]
where with $\mathrm{Fun}^{desc}$ we mean functors satisfying the descent conditions of either Theorem \ref{closedandprofinite} or Theorem \ref{properandcofilteredexcision}. In this sense, Theorem \ref{properandcofilteredexcision} actually contains Theorem \ref{closedandprofinite} as a special case, as precomposing a functor on $\mathrm{CH}$ that satisfies closed and profinite descent with the patch functor $patch : \mathrm{SC} \rightarrow \mathrm{CH}$ gives a functor that satisfies perfect and cofiltered descent.
\end{remark}

\section{The main theorem}

\begin{theorem} \label{maintheoremmain}
Let $X$ be a stably compact space, $\mathcal{C}$ a dualizable $\infty$-category and $F : \Cat^{\perf} \rightarrow \mathcal{E}$ a finitary localizing invariant with values in a dualizable $\infty$-category $\mathcal{E}$. There exists a natural equivalence
    $$ F^{\cont}( \Sh(X, \mathcal{C}) ) \simeq H^\bullet( X^{\patch}, F^{\cont}(\mathcal{C}) ),$$
where $H^\bullet( X^{\patch}, F^{\cont}(\mathcal{C}) ) = \Gamma( X^{\patch}, F^{\cont}(\mathcal{C})^{sh} )$ is the $\mathcal{E}$-valued sheaf cohomology of $X^{\patch}$ with value in the constant sheaf associated to $F^{\cont}(\mathcal{C})$.
\end{theorem}

\begin{proof}
We apply Theorem \ref{properandcofilteredexcision} to the functor $ F_{\mathcal{C}} = F^{\cont}( \Sh(-, \mathcal{C}) ) :  \mathrm{SC}^{\op} \rightarrow \mathcal{E}.$ To do so we need to check the three conditions:
\begin{enumerate}
\item The statement $F^{\cont}( \Sh(\emptyset, \mathcal{C}) ) \simeq F^{\cont}( 0 ) \simeq 0$ is clear.
\item Cofiltered descent for $F_{\mathcal{C}}$ follows by using that $F$ is assumed to be a finitary invariant together with Corollary \ref{cofiltereddescent}.
\item Perfect descent for $F_{\mathcal{C}}$ follows by stacking the following results:
\begin{itemize}
\item Corollary \ref{closedexcision} shows that closed excision holds.
\item Proposition \ref{saturatedcompactdescent} shows that saturated compact excision holds.
\item These two results together imply elementary compact descent using Proposition \ref{elementaryexcision}.
\item Finally elementary compact descent together with cofiltered descent implies perfect descent by Proposition \ref{properexcision}.
\end{itemize}
\end{enumerate}
\end{proof}

The extension of the result of Theorem \ref{maintheoremmain} to the case of stably locally compact spaces is straightforward. We recall that for a locally compact Hausdorff space $X$, and $\mathcal{E}$ some dualizable $\infty$-category the \emph{compactly supported cohomology} $H_{cs}^\bullet( X, E )$ of $X$ with value some $E \in \mathcal{E}$ is defined as the fiber of the map
$$ H^\bullet( X^*, E ) \rightarrow H^\bullet( \mathrm{pt}, E ) \simeq E,$$ 
induced by the inclusion of the point at infinity into the one-point-compactification $X^*$ of $X$. (See e.g.\ \cite[Definition 5.6.]{volpe2025operationstopology} for more information.)

\begin{theorem} \label{maintheoremmain2}
Let $X$ be a stably locally compact space, $\mathcal{C}$ a dualizable $\infty$-category and $F : \Cat^{\perf} \rightarrow \mathcal{E}$ a finitary localizing invariant with values in a dualizable $\infty$-category $\mathcal{E}$. There exists a natural equivalence
$$ F^{\cont}( \Sh(X, \mathcal{C}) ) \simeq H_{cs}^\bullet( X^{\patch}, F^{\cont}(\mathcal{C}) ).$$
\end{theorem}

\begin{proof}
The open-closed decomposition of $X^+$ into the open subspace $X$ and its closed complement $\{\infty\} \cong \mathrm{pt}$ gives the fiber sequence
\[\begin{tikzcd}
	{F^{\cont}(\Sh(X, \mathcal{C}))} & {F^{\cont}(\Sh(X^+,\mathcal{C}))} & {F^{\cont}( \mathcal{C}).}
	\arrow[from=1-1, to=1-2]
	\arrow[from=1-2, to=1-3]
\end{tikzcd}\]
Now, Theorem \ref{maintheoremmain} together with Corollary \ref{patchofonepoint} imply that the middle term computes as $F^{\cont}(\Sh(X^+,\mathcal{C})) \simeq H^\bullet( (X^{\patch})^*, F^{\cont}(\mathcal{C}) )$, and hence we conclude that
$$ F^{\cont}(\Sh(X, \mathcal{C})) \simeq H_{cs}^\bullet( X^{\patch}, F^{\cont}(\mathcal{C}) ).$$
\end{proof}

\begin{corollary} \label{sheavesequalcosheaves}
Let $X$ be stably locally compact, let $F$ be a finitary localizing invariant $ \Cat^{\perf} \rightarrow  \mathcal{E}$ with $\mathcal{E}$ presentable stable, and let $\mathcal{C}$ be a dualizable $\infty$-category. Then
$$F^{\cont}(\Sh(X,\mathcal{C})) \simeq F^{\cont}(\Cosh(X,\mathcal{C})).$$
\end{corollary}

\begin{proof} The statement for $\mathcal{E}$ being dualizable follows directly for stably compact spaces $X$ by using Verdier duality (Theorem \ref{verdierduality}), as we have the zigzag
$$ \Sh(X,\mathcal{C}) \rightarrow  \Sh(X^{\patch},\mathcal{C}) \leftarrow  \Sh(X^{\vee},\mathcal{C}) \simeq \Cosh(X,\mathcal{C}), $$
where all maps are equivalences on finitary localizing invariants by Theorem \ref{maintheoremmain}. (The patch topology of the de Groot dual $X^{\vee}$ agrees with $X^{\patch}$, see Remark \ref{deGrootdualitypatch}.) Note that this zig-zag itself is natural in proper maps of stably compact spaces. Therefore for $X$ only stably locally compact, we have
$$ \begin{array}{c}
F^{\cont}(\Sh(X,\mathcal{C})) \simeq  \mathrm{fib}( F^{\cont}(\Sh(X^+,\mathcal{C})) \rightarrow F^{\cont}( \mathcal{C} ) ) \\
\simeq \mathrm{fib}( F^{\cont}(\Cosh(X^+,\mathcal{C})) \rightarrow F^{\cont}( \mathcal{C} ) ) \simeq  F^{\cont}(\Cosh(X,\mathcal{C})),
\end{array} $$
finishing the claim. For $\mathcal{E}$ not necessarily dualizable, we can use the fact that $F$ factors through the universal finitary localizing invariant $\mathcal{U} : \Cat^{\perf} \rightarrow  \mathrm{Mot}$, using the result by Efimov that states that $\mathrm{Mot}$ is dualizable \cite{efimov2025rigiditycategorylocalizingmotives}.
\end{proof}

\appendix

\section{The sheaf condition over free distributive lattices of arbitrary finite rank} \label{sheafarbitraryrank}

While ultimately not needed for the purposes of this article, it can be useful to have a generalization of Lemma \ref{sheafconditionthree} to the case of cubical diagrams of arbitrary dimension. The main argument is exactly the same, but we need to set up a few definitions. Let us define the following posets:
\begin{itemize}
\item We define the \emph{spine} $S = [1] \amalg_{[0]} [1]$ to be the shape
\[\begin{tikzcd}
	& b \\
	a & c
	\arrow[from=1-2, to=2-2]
	\arrow[from=2-1, to=2-2]
\end{tikzcd}\]
obtained by glueing two copies of $[1]$ along the top element. We denote its elements by $a,b,c$ as labelled in the diagram.
\item Let $N$ be a set with $n$ elements. Denote by $\mathcal{C}(N) = \mathcal{P}(N) \setminus \{\emptyset\}$, a cube with the bottom corner removed, ordered via inclusion of subsets. Observe that $[1]^n \cong \mathcal{P}(N)$ is the cone of $\mathcal{C}(N)$. Fix an index $i \in N$. Then there exist two copies of $\mathcal{C}(N \setminus \{i\})$ in $\mathcal{C}(N)$ given as
\[\begin{array}{rcl}
 \mathcal{C}_0 &=& \{ A \subset N ~|~ A \neq \emptyset \text{ and } A \cap \{i\} = \emptyset \} \\
 \mathcal{C}_1 &=& \{ A \subset N ~|~ A \neq \emptyset, A \neq \{i\} \text{ and } A \cap \{i\} = \{i\} \}.
\end{array}\]
The full subposet given by the union of $\mathcal{C}_0$ and $\mathcal{C}_1$ is isomorphic to $\mathcal{C}(N \setminus \{i\} ) \times [1]$, and agrees with to $\mathcal{C}(N) \backslash \{i\}$.
\item Furthermore, recall that the \emph{join} $P \star P'$  of two posets $(P,{\leq_P})$ and $(P',\leq_{P'})$ is given by the set $P \amalg P'$ equipped with the order $\leq_\star$, which is the smallest order which agrees with the orders $\leq_P$ and $\leq_{P'}$ when restricted to $P$ and $P'$ respectively, and satisfies $p \leq_\star p'$ for all $p \in P$ and $p' \in P'$.\footnote{This agrees with the notion of \emph{join of simplicial sets} in the sense of Lurie \cite[1.2.8]{luriehtt}, in the sense that the simplicial join of the nerves of $P$ and $P'$ is given by the nerve of the posetal join  $P \star P'$.}
\end{itemize} 

\begin{lemma}
Let $\mathcal{D}$ be a category with finite limits, and let $F : [1]^n \rightarrow \mathcal{D}$ be a diagram. Then $F$ is a limit diagram iff the square
\[\begin{tikzcd}
	{F( \emptyset)} & {F( \{i\})} \\
	{\lim_{ \mathcal{C}_0 } F|_{ \mathcal{C}_0 }} & {\lim_{ \mathcal{C}_1 } F|_{ \mathcal{C}_1 }}
	\arrow[from=1-1, to=1-2]
	\arrow[from=1-1, to=2-1]
	\arrow[from=1-2, to=2-2]
	\arrow[from=2-1, to=2-2]
\end{tikzcd}\]
is a pullback square.
\end{lemma}

\begin{proof}
Consider the inclusion of $\mathcal{C}(N)$ into $S \star ~(\mathcal{C}(N) \backslash \{i\})$, by sending $\{i\}$ to the element $b \in S$. We now right Kan extend $F$ from $\mathcal{C}(N)$ along this inclusion. The entries at $S$ are then given by 
\[\begin{tikzcd}
	 & {F( \{i\})} \\
	{\lim_{ \mathcal{C}_0 } F|_{ \mathcal{C}_0 }} & {\lim_{ \mathcal{C}_1 } F|_{ \mathcal{C}_1 }}
	\arrow[from=1-2, to=2-2]
	\arrow[from=2-1, to=2-2]
\end{tikzcd}\]
Since right Kan extension preserves limits, the limit of the right Kan extended diagram agrees with $\lim_{ \mathcal{C}(N)} F $. The inclusion of $S$ into $S \star ~(\mathcal{C}(N) \backslash \{i\})$ is final, and hence we obtain the pullback square
\[\begin{tikzcd}
	{\lim_{ \mathcal{C}(N)} F } & {F( \{i\})} \\
	{\lim_{ \mathcal{C}_0 } F|_{ \mathcal{C}_0 }} & {\lim_{ \mathcal{C}_1 } F|_{ \mathcal{C}_1 },}
	\arrow[from=1-1, to=1-2]
	\arrow[from=1-1, to=2-1]
	\arrow["\lrcorner"{anchor=center, pos=0.125}, draw=none, from=1-1, to=2-2]
	\arrow[from=1-2, to=2-2]
	\arrow[from=2-1, to=2-2]
\end{tikzcd}\]
showing the claim.
\end{proof}

\begingroup
\setlength{\emergencystretch}{8em}
\printbibliography
\endgroup

\end{document}